\documentclass[10pt]{article}

\usepackage{amsmath, amssymb, amsthm, slashed, mathtools, setspace, hyperref, cancel}
\usepackage[nolabel]{showlabels}
\usepackage[T1]{fontenc}
\usepackage{url}
\usepackage{fullpage}
\usepackage[utf8]{inputenc}
\newcommand{\Ls}{{\widetilde{L}}}
\newcommand{\uLs}{{\underline{\widetilde{L}}}}
\newcommand{\Bs}[1]{{\widetilde{S}_{#1}}}

\newcommand{\drs}{{\partial_{r^*}}}
\newcommand{\dts}{{\partial_{t^*}}}
\newcommand{\wtt}{\widetilde{t}}
\newcommand{\wtx}{\widetilde{x}}

\newcommand{\wpa}{\widetilde{\partial}}

\newcommand{\dxsi}[1]{{\widetilde{\partial}_{{#1}}}}
\newcommand{\Dxsi}[1]{{\widetilde{D}_{#1}}}
\newcommand{\wt}{\widetilde}
\newcommand{\wopa}{\widetilde{\slashed\partial}}

\newcommand{\opar}{{\slashed\partial}}
\newcommand{\ops}[1]{{\widetilde{\slashed{\partial}}_{#1}}}
\newcommand{\Ds}{{\widetilde{\slashed{D}}}}

\newcommand{\Os}[1]{{\widetilde{\Omega}_{{#1}}}}
\newcommand{\Ss}{{\widetilde{S}}}
\newcommand{\okos}{\widetilde{K^s}{}}

\newcommand{\us}[1][{}]{{u^{*{#1}}}}
\newcommand{\uls}[1][{}]{{\underline{u}^{*{#1}}}}
\newcommand{\abus}[1][{}]{{|u^{*}|^{#1}}}
\newcommand{\tm}[1][{}]{{\tau_{-}^{#1}}}
\newcommand{\tp}[1][{}]{{\tau_{+}^{#1}}}
\newcommand{\tz}[1][{}] {{\tau_{0}^{#1}}}
\newcommand{\omu}{{\overline{\mu}}}

\newcommand{\mhat}{{\widehat{m}}}

\newcommand{\Lie}[1]{\mathcal{L}_{#1}}
\newcommand{\OLie}[1]{\widetilde{\mathcal{L}}_{#1}}
\newcommand{\LieC}[2][{ }]{{\mathcal{L}_{#2}^{\mathbb{C}{#1}}}}

\newcommand{\mcL}{{\mathcal{L}}}
\newcommand{\mcT}{{\mathcal{T}}}
\newcommand{\mcU}{{\mathcal{U}}}

\newcommand{\mcE}{{\mathcal{E}}}
\newcommand{\T}{\mathcal{T}}
\newcommand{\mbL}{{\mathbb{L}}}
\newcommand{\wX}{\widetilde{X}}
\newcommand{\wY}{\widetilde{Y}}

\newcommand{\DT}[1]{{^{({#1})}}\pi}
\newcommand{\TDT}[1]{{{^{({#1})}\widetilde{\pi}}}}
\newcommand{\TDTD}[1]{{{^{({#1})}\widetilde{\Pi}}}}
\newcommand{\DTD}[1]{{{^{({#1})}{\Pi}}}}

\newcommand{\lnorm}{\left\lVert}
\newcommand{\rnorm}{\right\rVert}
\newcommand{\lrnorm}[1]{{\left\lVert{#1}\right\rVert}}
\newcommand{\lrangle}[1]{{\left\langle{#1}\right\rangle}}

\newcommand{\intsig}[1]{{\int_{\Sigma_{#1}}}}
\newcommand{\slk}{{\sum_{\leq k}}}
\newcommand{\slj}{{\sum_{\leq j}}}
\newcommand{\slkI}[1]{{\sum_{\substack{\leq k \\ {#1}}}}}

\newcommand{\al}{{\alpha}}
\newcommand{\ual}{{\underline{\alpha}}}
\def\bF{\mathbf{F}}
\def\bphi{\boldsymbol{\phi}}

\def\bA{\mathbf{A}}
\def\bJ{\mathbf{J}}
\def\bG{\mathbf{G}}

\newcommand{\sgn}{{\text{sgn}}}

\newcommand{\ochi}{{\overline{\chi}}}

\newcommand{\wg}[1][{}]{{w_\gamma^{{#1}}}}
		
\newtheorem{Theorem}{Theorem}[section]
\newtheorem{Lemma}[Theorem]{Lemma}
\newtheorem{Proposition}[Theorem]{Proposition}
\newtheorem{Remark}[Theorem]{Remark}

\numberwithin{equation}{section}

\usepackage[backend=bibtex, bibencoding=ascii, style=numeric, citestyle=numeric]{biblatex}
\title{Global Stability for Charged Scalar Fields in an Asymptotically Flat Metric in Harmonic Gauge}
\date{}
\author{Christopher Kauffman}
\begin{document}
\maketitle
\begin{abstract}
We prove global stability for the Charge-Scalar Field system on a background spacetime which is close to $1+3$-dimensional Minkowski space and whose outward light cones converge to those for the Schwarzschild metric at null infinity. The key technique to this proof is the use of a modified null frame, depending only on the mass $M$ of the metric, which captures the asymptotic behavior of the metric at future null infinity. Our results are analogous to results obtained in Minkowski space by Lindblad and Sterbenz in \cite{LS} up to a change in coordinates, and will in the sequel be used to prove the full structure of the Einstein-Charge scalar field system in these modified harmonic coordinates.
\end{abstract}
\tableofcontents
\section{Introduction}
		\subsection{The Maxwell-Klein-Gordon System}
Given a background Lorentzian manifold $(\mathcal{M}, g)$ and a real one-form $\bA$ on $\mathcal{M}$, we define the complex covariant derivative
\begin{equation}
D_\alpha = \nabla_\alpha + i\bA_\alpha,
\end{equation}
where $\nabla$ is the Levi-Civita connection on $g$.
Then, defining the two-form $\bF = d\bA$ and a complex scalar function $\bphi$, we say $(\bF, \bphi)$ is a solution of the Maxwell-Klein-Gordon system with potential $\bA$ if
\begin{subequations}
\label{MKG}
\begin{align}
\label{MKGWave}\Box_g^{\mathbb{C}}\bphi &= D^\alpha D_\alpha \bphi = 0, \\
\label{currentvector}
\nabla^\beta \bF_{\alpha\beta} &= \bJ[\bphi]_\alpha = \mathfrak{I}\left(\bphi\overline{D_\alpha\bphi}\right), \\
\nabla^\beta( \star\bF)_{\alpha\beta} &= 0.
\end{align}
\end{subequations}
Here and in what follows, $\mathfrak{R}, \mathfrak{I}$ denote the projections onto the real and imaginary parts. We call $\Box_g^{\mathbb{C}}$ the complex wave operator, and $\bJ^\alpha$ is the current vector. The coupling between $\bphi$ and $\bF$ arises from the right hand side of \eqref{currentvector} along with the commutator
\begin{equation}
\label{commrelation}
[D_\alpha, D_\beta]\psi = i\bF_{\alpha\beta}\psi.
\end{equation}
The energy-momentum tensor of this system is
\begin{equation}
\label{FullEMTensor}
T_{\alpha\beta} = T[\bphi, \bF]_{\alpha\beta} = \mathfrak{R}\left(D_\alpha\bphi\overline{D_\beta\bphi} - \frac12g_{\alpha\beta}D_\gamma\bphi\overline{D^\gamma\bphi}\right) + \bF_{\alpha\gamma}\bF_\beta^{\,\,\gamma} - \frac14g_{\alpha\beta}\bF_{\gamma\delta}\bF^{\gamma\delta}.
\end{equation}
We may also separate it into its scalar and field portions, respectively
\begin{subequations}
\begin{align}
T[\bphi]_{\alpha\beta}&= \mathfrak{R}\left(D_\alpha\bphi\overline{D_\beta\bphi} - \frac12 g_{\alpha\beta}D_\gamma\bphi\overline{D^\gamma\bphi}\right), \label{def:EMphi}\\
T[\bF]_{\alpha\beta} &= \bF_{\alpha\gamma}\bF_\beta^{\,\,\gamma} - \frac14g_{\alpha\beta}\bF_{\gamma\delta}\bF^{\gamma\delta}. \label{def:EMF}
\end{align}
\end{subequations}
These satisfy the identities
\begin{equation}
\nabla^\beta T[\bphi]_{\alpha\beta} = \bF_{\alpha\gamma}\bJ^\gamma, \qquad \nabla^\beta T[\bF]_{\alpha\beta} =  -\bF_{\alpha\gamma}\bJ^\gamma, \qquad g^{\alpha\beta}T[\bF]_{\alpha\beta} = 0.
\end{equation}
The scalar divergence identity follows from \eqref{commrelation} along with the identity
\begin{equation}
\label{eq:LeibnizRule}
\nabla_\alpha(\phi\overline{\psi}) = D_\alpha\phi\overline\psi + \phi\overline{D_\alpha\psi},
\end{equation}
and the corresponding vector divergence identity follows from antisymmetry along with
\begin{equation}
d\bF = 0.
\end{equation}
In the sequel we will use our stability results to couple the system \eqref{MKG} to Einstein's field equations,
\begin{equation}\label{def:EFE}
R_{\mu\nu} = T_{\mu\nu} - \tfrac12g_{\mu\nu}g^{\alpha\beta}T_{\alpha\beta}.
\end{equation}

Given the system \eqref{MKG} along with suitable initial conditions on $\bF$ and $\bphi$, one has some freedom in the choice of the potential $\bA$, called the gauge, which it is not necessary to resolve. For a given real function $\psi$, if $(\bF, \bphi)$ solves \eqref{MKG} with potential $\bA$, then $(\bF, e^{-i\psi}\bphi)$ solves \eqref{MKG} with potential $\bA + d\psi$. From a historical perspective, fixing a gauge has aided in proving local and global existence results, even for the more general Yang-Mills Higgs equations. Eardley and Moncrief \cite{EM} proved local and global existence using the temporal gauge $\bA_0 = 0$, and subsequently, Klainerman and Machedon \cite{KM}, \cite{KM2} extended their result using elliptic estimates arising from the Coulomb gauge $\nabla_x\cdot\bA = 0$. However, these results did not rule out polynomially growing energy, so their use in establishing decay was limited. Improved estimates for small initial data were established later by Shu \cite{S91}, \cite{S92} in the massless case, and Psarrelli \cite{Ps} in the massive case, in which \eqref{MKGWave} is replaced with an appropriate Klein-Gordon type equation. Lindblad and Sterbenz \cite{LS}, and later Bieri, Miao, and Shashahani \cite{BMS} built on the results of Shu to prove global stability for small initial data in a gauge-invariant way. Decay results for large Maxwell data were established by Yang \cite{Y152, Y151}; however, coupling these to Einstein's equations is somewhat more difficult.

Our work builds on the techniques in \cite{LS}; however, one runs into the problem that the conformal Morawetz- and Stricharz-type estimates used are in a sense unstable with respect to metric perturbations satisfying the harmonic coordinate condition
\[
\Box_g(x^\mu) = 0.
\]
The stability of components of $g$ in these coordinates was shown by Lindblad and Rodnianski \cite{LR10} and was extended to systems with electromagnetic components by Loizelet \cite{Loize2008} and later by Speck \cite{S12}. However, the estimates they used required less precise decay rates on components.

Alternatively, a geometric approach to this system, as in the landmark work of Christodoulou and Klainerman \cite{CK93} and its extension to the Maxwell system by Zipser \cite{Z00} would alleviate this issue somewhat by using a true null foliation, but at the immediate cost of greatly increased complexity.
We instead use estimates on the metric in harmonic coordinates found by Lindblad \cite{L15} to reframe this problem in generalized wave coordinates, with a corresponding modified null frame.

\subsection{The Background Spacetime}
We consider spacetimes $(\mathcal{M}, g)$ close to the Minkowski spacetime on a time interval $[0,T]$, in the sense that $g$ satisfies certain $L^2$ and $L^\infty$ estimates consistent with small-data solutions to Einstein's Vacuum Equations in harmonic gauge. We decompose the metric as follows:
\begin{equation}
\label{Metric}
g_{\alpha\beta} = m _{\alpha\beta} + h^0_{\alpha\beta} + h^1_{\alpha\beta}, \qquad h^0_{\alpha\beta} = M \chi(\tfrac{r}{t+1})r^{-1}\delta_{\alpha\beta}
\end{equation}
where $M$ is the ADM mass for $g$, $\chi = \chi(y)$ is a smooth cutoff function equal to 1 for $y\geq \frac34$ and 0 for $y\leq \frac14$, such that $\partial\chi$ decays like $t^{-1}$, and $h^1$ is a small error which decays initially like $\langle r\rangle^{-3/2-\alpha}$ for some $\alpha > 0$. Under this decomposition, our decay bounds take the form
\begin{subequations}
\label{MetLIint}
\begin{align}
 M &< \varepsilon_g, \\
\langle t+r^*\rangle|\overline{\wpa}\Lie{\wX}^Ih^1|+ \langle t-r^*\rangle|\wpa\Lie{\wX}^Ih^1|+|\Lie{\wX}^Ih^1| &< \varepsilon_g\tp[-1+\delta], \\
\langle t-r^*\rangle|\wpa\Lie{\wX}^I h^1|_{\mcL\mcT} + |\Lie{\wX}^I h^1|_{\mcL\mcT} &< \varepsilon_g\tm[\gamma]\tp[-1-\gamma+\delta].
\end{align}
\end{subequations}
Here, $k\geq 11$ is an integer, $I$ is a multiindex with $|I|\leq k-6$, $\tau_{\pm}$ are defined in \eqref{taupm}, the subscript $\mcL\mcT$ is defined in \eqref{def:nullvectorsets}, the set $\wX \in \{\wpa_\alpha, \Os{\alpha\beta}, \Ss\}$ is defined in \eqref{LorentzFields}, $\gamma > \frac12$, and $\varepsilon_g>0$ is a small constant.

 Additionally, for a small constant $\mu > 0$, we may define the metric weight
\begin{equation}\label{def:weight}
\wg = \wg(r^* - t) =\begin{cases}
1+ (1+| r^* - t|)^{-2\mu}  & r^* \leq t, \\
1 +(1+| r^* - t|)^{1+2\gamma} & r^* \geq t,
\end{cases}
\end{equation}
where $r^*$ is defined in \eqref{def:rstar}. Then, for $|I| \leq k$ we assume the $L^2$ bounds
\begin{subequations}
\label{MetL2int}
\begin{align}
\label{L21int}
\left\lVert|\partial\Lie{\wX}^Ih^1|\wg[1/2]\right\rVert^2_{L^2(\mathbb{R}^3)} + \left\lVert\tm[-1]|\Lie{\wX}^Ih^1|\wg[1/2]\right\rVert^2_{L^2(\mathbb{R}^3)}&\leq \varepsilon_g^2 (1+t)^{\delta}, \\
\label{L22int}
\left\lVert\big(|\partial\Lie{\wX}^Ih^1|_{\Ls\Ls}+ |\overline\partial\Lie{\wX}^Ih^1|\big)(\wg\!\!')^{1/2}\right\rVert^2_{L^2([0,T]\times\mathbb{R}^3)} &\leq \varepsilon_g^2(1+T)^{\delta},\\
\label{L23int}
\left\lVert{\!\langle r^*\!\!\!-\!t\rangle}^{-1-s}{\!\langle r^*\!\!\!+\!t\rangle}^{-1+s}|\Lie{\wX}^Ih^{1\!}|_{\mathcal{L}\mathcal{L}}
(w_\gamma')^{1/2}
\right\rVert_{L^2([0,T]\!\times\mathbb{R}^3)}&\leq \!\varepsilon_g
\end{align}
\end{subequations}
for $|I| \leq k$.
The fields we use are adaptations of the traditional commutator fields in Minkowski space. In addition to \cite{L15}, they have been used by Oliver in \cite{O16}, and later by Sterbenz and Oliver in \cite{OS} for a more general class of metrics. In the sequel we will show that these bounds follow from a bootstrap assumption for the Einstein field equations in harmonic coordinates,
\begin{equation}\label{eq:EECoupled}
\Box_gg_{\mu\nu} = {P}[\partial g, \partial g]_{\mu\nu}(g) + T_{\mu\nu} - \frac12g_{\mu\nu}\text{tr}_g(T),
\end{equation}
where ${P}$ is quadratic in derivatives of $g$ and behaves nicely in the null decomposition. $T$ will generally be smaller and decay faster than corresponding components of $P$, which appears in the vacuum equations, so the background spacetime of the coupled system may be approximated by solutions of the vacuum equations.

\subsection{The Modified Coordinates and Weights}
The presence of the mass term in \eqref{Metric} means that, when attempting to establish a conformal Morawetz-type estimate, we cannot simply treat the background spacetime as a perturbation of Minkowski space. Instead, we use an interpolated tortoise coordinate $r^*$ and approximate optical functions $\us = t + r^*$, $\uls = t - r^*$. Since $h^0$ is supported when $r \gg M$, we do not have to worry about trapped geodesics, so we may use a conformal estimate rather than the purely radial estimate of \cite{DR}. This allows for a more intuitive treatment of the charged part of $\bF$. Similar estimates for Maxwell's equations on Schwarzschild have been carried out by Anderson and Blue \cite{AB}, and Sterbenz and Tataru \cite{ST}, as well as for certain quasilinear wave equations by Lindblad and Tohaneanu \cite{LToh}. 

Additionally, decay estimates on $h^1$ in the vacuum case mean that we cannot hope to recover the full conformal Morawetz estimate. We will instead prove a fractional Morawetz estimate, as in \cite{LS}, with weights $\us[2s], \uls[2s]$ which depend on the initial decay of $h^1$. We show the base estimate, with a discussion of the modified fields and the necessity of peeling estimates for $h$, in section \ref{ModelMorawetz}. A similar estimate for wave equations was shown by Lindblad and Schlue in \cite{LSch}

We define the adapted tortoise coordinate
\begin{equation}\label{def:rstar}
r^* = r + M\chi\ln(r),
\end{equation}
where $\chi$ is as in \eqref{Metric}.
It follows that $r^* = r$ in the far interior $r < (t+1)/4$ and $r^* = r+M\ln(r)$ for $r > 3(t+1)/4$. We then define the modified coordinates $(\wtt, \wtx)$ and approximate optical functions $\us, \uls$ by:
\begin{equation}\label{def:newcoords}
\wtt = \wtx^0 = t^* = t, \qquad \wtx^i = \omega^ir^*, \qquad\us = t - r^*, \qquad \uls = t + r^*
\end{equation}
It follows from \eqref{MetLIint} that $g^{\alpha\beta}\partial_\alpha \us\partial_\beta\us$ decays more rapidly along the light cone than $t-r$. Additionally, we can define the optical weights
\begin{equation}
\label{taupm}
\tp[2] = (1+\uls[2]) \qquad \tm[2] = (1+\us[2]) \qquad \tz = \tm/\tp.
\end{equation}

Taking $\partial_r = \omega^i\partial_i$ and $\opar_i = \partial_i - \omega_i\partial_r$, we can define
\begin{equation}\label{def:newderivatives}
\drs = \tfrac{1}{\partial_r(r^*)}\partial_r, \qquad \dts = \dxsi{0} = \partial_t - \partial_t(r^*)\drs, \qquad \dxsi{i} = \omega_i \drs + \tfrac{r}{r^*}\opar_i, \qquad \ops{i} = \tfrac{r}{r^*}\opar_i.
\end{equation}
We have a natural modified null frame
\begin{equation}\label{def:nullframe}
\Ls = \dts + \drs, \qquad \uLs = \dts - \drs, \qquad \Bs{i} = \tfrac{r}{r^*}S_i,
\end{equation}
where $S_i = \{S_1, S_2\}$ are piecewise defined fields forming an orthonormal frame tangent to the sphere (in the Minkowski metric). We define the sets
\begin{equation}\label{def:nullvectorsets}
\mcL = \{\Ls\}, \qquad \mcT = \{\Ls, \Bs{1}, \Bs{2}\}, \qquad \mcU = \{\Ls, \uLs, \Bs{1}, \Bs{2}\},
\end{equation}
and use the following notation for partial norms
\[
|T|_{\mathcal{X}\mathcal{Y}} = \sup_{X \in \mathcal{X}, Y \in \mathcal{Y}} |T_{XY}|.
\]
For norms of tensors where vector fields are not specified, we use the full null frame $\mcU$, e.g., for $(0,2)$-tensors $T$, we have
\[
|T| = |T|_{\mcU\mcU}
\]

Our modified null frame has two advantages: First, components of $g$ and its Lie derivatives have improved bounds in the null decomposition, and second, these vectors commute well with the modified Lorentz fields used in \cite{L15}, which is necessary in order to pass the nice weak null structure to derivatives. The peeling estimates obtained from this modification mirror those for Minkowski, and are necessary in order to close the argument.

\subsection{The Main Theorem}

Before stating our result, we define certain tensorial quantities which appear. For a generic 2-form $\bG$, we define the modified null decomposition:
\begin{subequations}\label{def:chargedecomp}
\begin{alignat}{4}
\al_i[\bG] &= \bG_{\Ls\Bs{i}} & \qquad\ual_i[\bG] &= \bG_{\uLs \Bs{i}} \\
\rho[\bG] &= \frac12\bG_{\uLs\Ls} & \sigma[\bG] &= \frac12\bG_{\Bs{1}\Bs{2}}
\end{alignat}
\end{subequations}
When there is no ambiguity, we remove the explicit dependence on $\bG$.
Since the tangential terms are not uniquely defined, the following terms often show up in our calculations.
\begin{equation}\label{def:angularnorms}
|\Ds\phi|^2 = |D_{\Bs{1}}\phi|^2 + |D_{\Bs{2}}\phi|^2 \qquad |\al|^2 = |\al_1|^2 + |\al_2|^2 \qquad |\ual|^2 = |\ual_1|^2+|\ual_2|^2.
\end{equation}

We additionally define the electromagnetic decomposition
\begin{equation}
\mathbf{E}_i[\bG] = \bG_{0i}, \qquad \mathbf{B}_i[\bG] = (\star\bG)_{0i},
\end{equation}
where $\star\bG$ is the Hodge dual of $\bG$.
We can decompose $\mathbf{E}$ into its divergence-free and curl-free components, $\mathbf{E}_{df}$ and $\mathbf{E}_{cf}$ respectively.

Before stating our result, we define the norms governing our initial conditions. For a generic $(0,j)$-tensor $\mathbf{T}$, we define:

\begin{equation}
\left\lVert \mathbf{T}\right\rVert^2_{H^{k, s_0}(\mathbb{R}^3)} = \sum_{|I| \leq k} \sum_{\alpha_i \in (0,3)}\int_{\mathbb{R}^3} (1+r^2)^{s_0+|I|}|\underline\nabla^I \mathbf{T}(\dxsi{\alpha_1}, ...\dxsi{\alpha_j})|^2 \, dx,
\end{equation}
where $I$ is a multiindex. Here, $\underline\nabla$ and $\underline{D}$ are the covariant derivatives restricted to time slices.

Likewise, for a complex scalar field $
\phi$, we have the corresponding quantity
\begin{equation}
\left\lVert \phi\right\rVert^2_{H^{k, s_0}(\mathbb{R}^3)} = \sum_{|I| \leq k} \int_{\mathbb{R}^3} (1+r^2)^{s_0+ |I|}|\underline{D}^I \phi|^2 \, dx.
\end{equation}

\begin{Theorem}
\label{Main}
Take constants $s, s_0, \gamma, \mu, \delta, {\overline{\mu}}$ such that $\frac12 < s < 1 < s_0 < 3/2$, $2s-1  < \gamma <1$, $\gamma>1/2$, $0< \mu < 1/2$, and $0 < 4\delta, 4\overline{\mu} < \min(1-2\mu, s-\frac12, 1-s, \gamma-\frac12, 1-\gamma, 1+\gamma-2s)$. Additionally, take an integer $k \geq 11$.

There exists constants $\varepsilon_0$, $\varepsilon_1> 0$ such that if the metric, in the decomposition \eqref{Metric}, satisfies \eqref{MetLIint} and \eqref{MetL2int} for $\varepsilon_g < \varepsilon_1$, and if we take initial conditions $\mathbf{E}_0, \mathbf{B}_0, \bphi_0, \,\,\dot{\!\!\bphi}_0$ for $\bF$ and $\bphi$ satisfying

\begin{equation}\label{def:IC}
\left\lVert \mathbf{E}_{0df} \right\rVert_{H^{k, s_0}(\mathbb{R}^3)} + \left\lVert \mathbf{B}_0 \right\rVert_{H^{k, s_0}(\mathbb{R}^3)}+ \left\lVert D\bphi_0 \right\rVert_{H^{k, s_0}(\mathbb{R}^3)} + \lVert\,\,\dot{\!\!\bphi}_0 \rVert_{H^{k, s_0}(\mathbb{R}^3)} < \varepsilon
\end{equation}
at time $t = 0$ for $\varepsilon < \varepsilon_0$, then solutions to \eqref{MKG} exist for all time, with
\[
\mathbf{E}[\bF](0,x) = \mathbf{E}_{0}, \qquad \mathbf{B}[\bF](0,x) = \mathbf{B}_0, \qquad \bphi(0, x) = \bphi_0, \qquad D_t \bphi (0,x) =  \,\,\dot{\!\!\bphi}_0
\]
 and satisfy the bounds
\begin{subequations}
\begin{align}
|\alpha| + \left|\tfrac{1}{r^*}D_{\Ls}(r^*\bphi)\right|\chi_{r>((t+2)/2)} + |D_{\Ls}\bphi|\chi_{r<((t+2)/2)} &\lesssim \varepsilon\tp[-s-3/2], \\
|\ual| + |D_{\uLs}\bphi| &\lesssim \varepsilon\tp[-1]\tm[-1/2-s], \\
|\sigma|  &\lesssim \varepsilon\tp[-1-s]\tm[-1/2],\\
|\slashed{D}\bphi| &\lesssim \varepsilon\tp[-2]\tm[1/2-s],\\
|\rho| &\lesssim \varepsilon\left(\tp[-1-s]\tm[-1/2]\chi_{r^* < t} + \tp[-2]\chi_{r^* \geq t}\right), \\
|\bphi| &\lesssim \varepsilon\tp[-1]\tm[1/2-s],
\end{align}
\end{subequations}
\end{Theorem}
\begin{Remark}
More precise bounds, are given in Theorem \ref{Bootstrap}.
\end{Remark}

We outline the proof, which spans the remainder of this paper.

Section 2 opens with definitions and notation that we will use later on, and the bulk of the section is devoted to analytical identities and estimates which follow from our bounds on $h^1$. We conclude this section by proving several Morawetz estimates which provide motivation for the particular form of the conformal energy estimates we use, including the fractional Morawetz estimate, the modified null frame, and the weight $w$.

In sections 3 and 4, we establish fractional Morawetz-type estimates which will be used to bound derivatives of $\bF^1$ and $\bphi$ respectively, which are roughly adapted from \cite{LS}.

Sections 5 and 6 establish $L^\infty$ estimates on field quantities. These are straightforward weighted Klainerman-Sobolev estimates, with some additional care taken to account for the contribution of the charge and certain error terms which arise from the fact that $u^*$ is only an approximate optical function. Additionally, we set up a Strichartz-type estimate which will be useful to close the result.

In section 7 we prove the main theorem, up to certain bounds on commutator terms which we show in sections 8 and 9.

Section 10 is an appendix in which we prove some common analytical estimates we use.

	\subsection{Acknowledgements}
		I would like to thank Hans Lindblad first and foremost for his invaluable guidance, and to everybody who has helped me with valuable advice and conversations, including but not limited to Gustav Holzegel, Chenyun Luo, Dan Ginsberg, Tim Candy, and Volker Schlue, in addition to support through ERC consolidator grant 772249 and NSF Grant DMS-1500925.
\section{Notation and Preliminary Identities}
	\subsection{General Notation}
		Here and in what follows, we will use $\bF$ and $\bphi$ specifically to refer to solutions of the Maxwell-Klein-Gordon system. We will typically use (nonboldface) $\phi$ and $\psi$ to refer to generic scalar functions, and $\bG$ to refer to generic 2-forms. With some overload of notation, we define
\begin{equation}
\bJ[\psi]_\alpha = \mathfrak{I}(\psi \overline{D_\alpha\psi}), \qquad\bJ[\bG]_\alpha = \nabla^\beta \bG_{\alpha\beta}.
\end{equation}

As we will be working in two different coordinate systems, we outline the differences here in order to ensure clarity. When we are working in harmonic coordinates, we will use Greek indices $\alpha, \beta, \hdots$ to refer to components with respect to this frame.

It will occasionally be simpler to describe concepts in terms of the generalized harmonic coordinates $(\wtt, \wtx)$, defined in \eqref{def:newcoords}. We will use $a, b, c, d$ to refer to components in the modified frame. We will use $\widetilde{g}$, $\widetilde{h}^1$, $\widetilde{\Gamma}$ to refer to components of corresponding quantities. For purely geometric quantities, and in the null decomposition, we will use them interchangeably.

We will also use $i, j, k, l$ to refer to spatial variables, which are raised and lowered with respect to $m$ unless otherwise specified.

Indices are raised and lowered with respect to the metric $g$, except for the background metrics $m$ and $\mhat$, for which we will use the inverse metric, with the exceptions of section \ref{MM1}, for which we will use the Minkowski metric, and section \ref{MM23} and a portion of section 4, for which we will use the metric $\mhat$.

A sharp ($\sharp$) or flat ($\flat$) appearing in a Lie derivative denotes that we have raised or lowered (respectively) the index before applying the Lie derivative.

Given a covector $\omega_\alpha$ and a vector $X$ with components $X^\alpha$, we say
\[
\omega_X = \omega(X) = X^\alpha\omega_\alpha.
\]

We can extend this to $(k,0)$-tensors, so that we use $F_{XY}$ and $F(X,Y)$ interchangeably. Similarly, given a frame $\{X_i\}$, we define the decomposition of a $(0,k)$ tensor $G$ using
\begin{equation}
G^{\alpha_1\hdots \alpha_k} = \sum_{i_1, \hdots i_k = 1}^4 G^{X_{i_1} \hdots X_{i_k}}X_{i_1}^{\alpha_1}\hdots X_{i_k}^{\alpha_k}
\end{equation}

 Additionally, we define the contraction operator $i_{\wX}$ such that for a $2$-form $\bG$, $(i_{\wX}\bG)(\wY) = \bG(\wX, \wY).$

We take $a \lesssim b$ to mean 
\[a \leq C_{s, s_0, \delta, \gamma, \mu, {\overline{\mu}}, k}b,\]
 and similarly, $a \approx b$ means $a\lesssim b, b\lesssim a$. In particular, $C$ may always be chosen independently of $\varepsilon, \varepsilon_g$, and $T$, for suitable $\varepsilon_0, \varepsilon_1$. As we will be revising the values of $C$ and $\varepsilon_1$ throughout the course of the proof, we include the following requirement: If a statement is true for a given choice of $\varepsilon_1$, then it is true for all $\varepsilon'_1\in (0, \varepsilon_1]$ without changing the implicit constant. Therefore, after every step, we may increase $C$ and decrease the $\varepsilon$ quantities accordingly without jeopardizing our results.

 Norms with a numerical subscript $p$ will denote the spacetime norm:
\[
\lnorm\phi\rnorm_p = \lnorm\phi\rnorm_{L^p([0,T]\times\mathbb{R}^3)}.
\]
Norms on other domains will be unambiguously denoted. We will use $\Sigma_\tau$ to mean $\{t = \tau\}\times\mathbb{R}^3$.
	\subsection{Modified Coordinates and Vector Fields}
		Recalling the definition \eqref{def:newcoords}, we define the coordinate transformation symbols
\begin{equation}\label{ADef}
A_\alpha^a = \partial_\alpha\wt{x}^a, \qquad \wt{A}_a^\alpha = \wpa_a x^\alpha,
\end{equation}
so, e.g. 
\[
\partial_\alpha = A_\alpha^a \wpa_a, \qquad \wpa_a = \wt{A}_a^\alpha\partial_\alpha.
\]
Straightforward calculation gives
\begin{subequations}\label{def:Asymbols}
\begin{align}
A_\alpha^0 &= \delta_\alpha^0, \\
A_0^i &= -M\omega^i\tfrac{r}{(t+1)^2}\chi'\ln(r), \\
A_i^j &= \delta_i^j + \tfrac{M\chi\ln r}{r}(\delta_i^j - \omega_i\omega^j) + \omega_i\omega^j\Big(\tfrac{M}{t+1}\chi'\ln r + \tfrac{M\chi}{r}\Big).
\end{align}
\end{subequations}
Then the following bounds hold:
\begin{equation}\label{est:ABounds}
|A_\alpha^a - \delta_\alpha^a| + |\wt{A}^\alpha_a - \delta^\alpha_a|\lesssim \tfrac{M\lceil\chi\rceil\ln r}{r} \qquad |\partial_\beta A_\alpha^a| + |\partial_\beta \wt{A}^\alpha_a| \lesssim \tfrac{M\lceil\chi\rceil\ln r}{r^2}
\end{equation}
This follows from straightforward calculation for $A_\alpha^a$, and the identity $\wt{A}_a^\alpha A_\alpha^b = \delta_a^b$.

We define the modified commutator fields
\begin{equation}\label{LorentzFields}
\wt{\mathbb{L}} = \{\wpa_a, \quad \Os{ij} = \wtx^{i}\wpa_j - \wtx^j\wpa_i = \Omega_{ij}, \quad  \Os{0j}= \wtt\wpa_j + \wtx^j\wpa_0, \quad \Ss = \wtt\wpa_0 + \wtx^i\wpa_i\}.
\end{equation}
The notation $\wX, \wY$ always refer to vectors in this set, and $\wX^I$ will mean an (ordered) set of these vectors multiindexed by $I$. We can define all other possible values of $\Os{ab}$ with the convention $\Os{ab} = -\Os{ba}$. Then, we define
\begin{equation}
\label{CX}
c_{\wX} = \begin{cases}
2 & \wX = \Ss, \\
0 & \wX = \wpa_a\text{ or } \wX = \Os{ab}.
\end{cases}
\end{equation}
For $M=0$, $\Lie{\wX}m = c_{\wX}m$. Given a collection of Lorentz fields $\wX$ and a multiindex $I$, we define $c_{\wX}^I$ to be the product of $c_{\wX}$ for each $\wX$ indexed by $I$. Next, recalling \eqref{def:nullframe}, we define the radial Lorentz boost field
\begin{equation}\label{RadBoost}
\Os{0r^*} = \sum_i\omega^i\Os{0i} = t^*\drs + r^*\dts = \tfrac12\left(\uls\Ls - \us\uLs\right).
\end{equation}
Then,
\begin{equation}
\Os{0j} = \omega_j\Os{0r} + t\wopa_j,
\end{equation}
which will simplify later commutator estimates. The corresponding commutators are
\begin{subequations}
\label{FieldComms}
\begin{alignat}{6}
[\wpa_a, \wpa_b] &= 0, \qquad &  [\wpa_a, \Os{ij}] &= \delta_{a[i}\wpa_{j]}, \qquad & [\wpa_a, \Ss] &= \wpa_a \\
[\Ss, \Os{ab}] &= 0, &[\wpa_a, \Os{0b}] &= \delta_{a(0}\wpa_{b)}, & [\Os{0i}, \Os{jk}]&= \delta_{ij}\Os{0k} - \delta_{ik}\Os{0j} \\
&&[\Os{0i}, \Os{0j}] &=\Os{ij},&[\Os{ij}, \Os{kl}] &= -\delta_{ik}\Os{jl} + \delta_{il}\Os{jk} + \delta_{jk}\Os{il} - \delta_{jl}\Os{ik}.
\end{alignat}
\end{subequations}
as well as
\begin{subequations}
\label{FrameComms}
\begin{alignat}{6}
[\Ls, \wpa_0] &= 0, \qquad & [\uLs, \wpa_0] &= 0, \qquad& [\Bs{j}, \wpa_0] &= 0, \\
[\Ls, \wpa_i] &= -\tfrac{1}{r^*}\wopa_{i}, \qquad & [\uLs, \wpa_{i}] &= \tfrac{1}{r^*}\wopa_i, \qquad & [\Bs{j}, \wpa_{i}] &= \tfrac{1}{r^*}a_{ij}^k(\omega) \Bs{k} + \Bs{j}(\omega_i)\drs,\\
[\Ls, \Ss] &= \Ls, & [\uLs, \Ss] &= \uLs, & [\Bs{j}, \Ss] &= \Bs{j}, \\
[\Ls, \Os{ij}] &= 0, & [\uLs, \Os{ij}] &= 0, & [\Bs{k}, \Os{ij}] &= b_{ijk}^l(\omega)\Bs{l}, \\
[\Ls, \Os{0i}] &= \omega_i \Ls + \tfrac{r^* - t}{r^*}\ops{i}, \quad & [\uLs, \Os{0i}] &= -\omega_i\uLs + \tfrac{r^* + t}{r^*}\ops{i},\quad & [\Bs{j}, \Os{0i}] &= \Bs{j}(\omega_i) \Os{0r} + \tfrac{t}{r^*}c_{ij}^l(\omega)\Bs{l}.
\end{alignat}
\end{subequations}

Here, $a, b,$ and $c$ are homogeneous functions of degree 0 in $r$ which satisfy the conditions
\begin{equation}
a_{i1}^1 = a_{i2}^2 = b_{ij1}^1 = b_{ij2}^2 = c_{i1}^1 = c_{i2}^2 = 0, \qquad
a_{i1}^2 + a_{i2}^1 = b_{ij1}^2 + b_{ij2}^1 = c_{i1}^2 + c_{i2}^1 = 0.
\end{equation}
We recall the Lie derivative formulas for one- and two-forms respectively:
\begin{align}
\label{LieOneForm}
(\Lie{X}\omega)_Y &= X(\omega_Y) - \omega([X,Y]), \\
\label{LieTwoForm}
(\Lie{X} \bG)_{YZ} &= X(\bG_{YZ}) - \bG([X,Y], Z) - \bG(Y, [X,Z]).
\end{align}

\begin{Lemma}\label{NullComms}
Given a two-form $\bG$, and $\tz = \tm/\tp$, and for a set of commutator fields $\wX^I$, it follows that
\begin{subequations}
\begin{align}
\wX^I(\al_i[\bG]) &\lesssim \sum_{\substack{|J| \leq |I|}}|\alpha[\Lie{\wY}^J\bG]| + |\tz\rho[\Lie{\wY}^J\bG]| + |\tz\sigma[\Lie{\wY}^J\bG]| + \tz[2]|\ual[\Lie{\wY}^J\bG]|, \\
\wX^I(\rho[\bG]) &\lesssim \sum_{\substack{|J| \leq |I|}}|\alpha[\Lie{\wY}^J\bG]| + |\rho[\Lie{\wY}^J\bG]| + |\sigma[\Lie{\wY}^J\bG]| + \tz|\ual[\Lie{\wY}^J\bG]|, \\
\wX^I(\sigma[\bG]) &\lesssim \sum_{\substack{|J| \leq |I|}}|\alpha[\Lie{\wY}^J\bG]| + |\rho[\Lie{\wY}^J\bG]| + |\sigma[\Lie{\wY}^J\bG]| + \tz|\ual[\Lie{\wY}^J\bG]|, \\
\wX^I(\ual_i[\bG]) &\lesssim \sum_{\substack{|J| \leq |I|}}|\alpha[\Lie{\wY}^J\bG]| + |\rho[\Lie{\wY}^J\bG]| + |\sigma[\Lie{\wY}^J\bG]| + |\ual[\Lie{\wY}^J\bG]|.
\end{align}
\end{subequations}
when $r^* \geq \tfrac{t+1}{4}$.
\end{Lemma}
\begin{proof}
This follows from induction using \eqref{FrameComms}. Specifically, we define the class $\Phi^M$ to be the set of functions which can be written as a finite sum
\[
\sum_{m + n\geq M}\Big(\tfrac{1}{r^*}\Big)^m\Big(\tfrac{\us}{r^*}\Big)^nP(\tfrac{r^*}{t+1}, \tfrac{t}{r^*}, \omega),
\]
for a compactly supported function $P$. Then, if $\phi \in \Phi^m$, $\wX\phi \in \Phi^m$. Additionally, for any $\wX$, we can write
\begin{subequations}
\begin{align}
[\Ls, \wX] &= \phi^0_1\Ls + \sum_j\phi^1_{1j}\Bs{j}, \\
[\Bs{i}. \wX] &= \phi^0_{2i}\Ls +\sum_j\phi^0_{2ij}\Bs{j} + \phi^1_{2i}\uLs, \\
[\uLs, \wX] &= \phi^0_{3}\Ls +\sum_j\phi^0_{3j}\Bs{j} + \phi^0_{4}\uLs,
\end{align}
\end{subequations}
where all $\phi^0$ functions are in $\Phi^0$, and all $\phi^1$ functions are in $\Phi^1$. Repeating this and noting that if $\phi \in \Phi^m$, then, for some $C$ depending on $\phi$, and for $r^* > \frac{t+1}{4}$,
\[
|\phi|\leq C\big(\tfrac{\langle\us\rangle}{r^*}\big)^m.
\]
\end{proof}
	\subsection{Assumptions on the Metric}\label{subsec:MetricAssumptions}
		The modified null frame gives better decay for certain components of the metric, which we will quantify in this section. We define
\begin{equation}
g_{\alpha\beta} = m_{\alpha\beta} + h^0_{\alpha\beta} + h^1_{\alpha\beta}, \qquad g^{\alpha\beta} = m^{\alpha\beta} + H_0^{\alpha\beta}  + H_1^{\alpha\beta},
\end{equation} where
\begin{equation}
h^0_{\alpha\beta} =  \tfrac{M\chi}{r}\delta_{\alpha\beta}, \qquad H_0^{\alpha\beta} =  - \tfrac{M\chi}{r}\delta^{\alpha\beta}.
\end{equation}
We define $\widehat{m}$ to be the Minkowski metric with respect to the new frame:
 \begin{equation}\label{def:hatm}
 \widehat{m} = -d\widetilde{t}^{\,2} + \sum_{i}d\widetilde{x_i}^2
 \end{equation}
The following estimate holds.
\begin{Proposition}
Define 
\begin{equation}
m^0_{\alpha\beta} = m_{\alpha\beta} + h^0_{\alpha\beta}, \qquad m_0^{\alpha\beta} = m^{\alpha\beta} + H_0^{\alpha\beta},
\end{equation}
as well as
\begin{equation}\label{def:sdelta}
\slashed{\delta}_{ij} = \delta_{ij} - \omega_{i}\omega_{j}, \qquad \slashed{\delta}^{ij} =  \delta^{ij} - \omega^{i}\omega^{j},
\end{equation}
Additionally, define $\slashed{\delta}_{\alpha\beta}, \slashed{\delta}^{\alpha\beta}$ using \eqref{def:sdelta}for $1\leq\alpha, \beta \leq 3$ and 0 if $\alpha = 0$ or $\beta = 0$. Then,
\begin{equation}\label{MetricApprox1}
|\Lie{\wX}^I (m^0 - (1-\tfrac{M\chi}{r})\widehat{m} + \tfrac{2M\chi\ln r - 2M\chi}{r}\slashed{\delta})|_{\mcU\mcU} \leq C(\tfrac{M\chi\langle\ln(1+r)\rangle^2}{r^2} + \tfrac{M\lceil \chi' \rceil \langle \ln (1+r) \rangle}{r})
\end{equation}
and
\begin{equation}\label{MetricApprox2}
|\Lie{\wX}^I (m_0 - (1+\tfrac{M\chi}{r})\widehat{m} - \tfrac{2M\chi\ln r - 2M\chi}{r}\slashed{\delta})|_{\mcU\mcU} \leq C(\tfrac{M\chi\langle\ln(1+r)\rangle^2}{r^2} + \tfrac{M\lceil \chi' \rceil \langle \ln (1+r) \rangle}{r})
\end{equation}
\end{Proposition}
\begin{proof}
We sketch out the proof, as a more thorough version appears in the sequel. This is trivially true when $\chi = 0$, so we may reduce this to standard derivatives in the $\wt{x}_a$ frame without worrying about behavior of the origin. Using \eqref{def:Asymbols}, every component of the difference can be written in the form
\[
\frac{MP_1(\tfrac{\chi}{r}, \tfrac{\chi\ln r}{r}, \tfrac{\chi'\ln r}{r}, \frac{t}{r}, \omega)}{1+MP_2(\tfrac{\chi}{r}, \tfrac{\chi\ln r}{r}, \tfrac{\chi'\ln r}{r}, \frac{t}{r}, \omega)},
\]
where each term in $P_2$ has a factor of $\tfrac{\chi}{r}$, $\tfrac{\chi\ln r}{r}$, or $\tfrac{\chi'\ln r}{r}$, and each term in $P_1$ contains a factor of $(\tfrac{\chi}{r})^2$, $(\tfrac{\chi\ln r}{r})^2$, $\tfrac{\chi}{r}\tfrac{\chi\ln r}{r}$, or $\tfrac{\chi'\ln r}{r}$. By induction, this holds also for Lie derivatives, if we let the polynomials depend on $\frac{1}{r}$ and $\frac{ \ln r}{r}$ terms containing higher derivatives of $\chi$. The bound \eqref{MetricApprox1} directly follows, noting that the denominator is uniformly bounded below for small $M$. The bound \eqref{MetricApprox2} follows from a similar argument after expressing each quantity in the modified coordinate system.
\end{proof}

Now we look at $h^1, H_1$. Since the main result is in a sense a part of a bootstrap argument for the full EMKG system, we assume $L^\infty$ estimates on low derivatives of $h^1$ and $H_1$ and $L^2$ estimates on high derivatives of $h^1$ and $H_1$. For now, we use the abstract $h$ for both (i.e. if an estimate is true for $h$, then it holds for $h^1$ and $H_1$, with indices raised and lowered as appropriate). We look at the set of metrics $g$ satisfying the $L^\infty$ norms

\begin{subequations}
\label{MetLI}
\begin{align}
 M &< \varepsilon_g, \\
\label{est:BadComp}
|\Lie{\wX}^Ih| &< \varepsilon_g\tp[-1+\delta], \\
|\Lie{\wX}^I h|_{\mcL\mcT} &< \varepsilon_g\tz[\gamma]\tp[-1+\delta],
\end{align}
\end{subequations}
for a multiindex $I, |I| \leq k - 6$, and
\begin{subequations}
\label{MetL2}
\begin{align}
\label{L21}
\left\lVert|\partial\Lie{\wX}^Ih|w_\gamma^{1/2}\right\rVert_{L^2(\mathbb{R}^3)} + \left\lVert\tm[-1]|\Lie{\wX}^Ih|w_\gamma^{1/2}\right\rVert_{L^2(\mathbb{R}^3)}&< \varepsilon_g (1+t)^{\delta/2}, \\
\label{L22}
\left\lVert\big(|\overline\partial\Lie{\wX}^Ih| + |\partial\Lie{\wX}^Ih|_{\Ls\Ls}\big) (w'_g)^{1/2}\right\rVert_{L^2([0,T]\times\mathbb{R}^3)}&< \varepsilon_g(1+T)^{\delta/2}, \\
\label{L23}\left\lVert\tp[s-1]\tm[-1-s]|\Lie{\wX}^Ih|_{\Ls\Ls}\right\rVert_{L^2([0,T]\times\mathbb{R}^3)} &< \varepsilon_g
\end{align}
\end{subequations}
for $|I| \leq k$ and
\begin{equation}
w_\gamma = \begin{cases}
1 + (1+t-r^*)^{-2\mu}& r^* \leq t, \\
1+(1 + r^* - t)^{1+2\gamma} & r^* \geq t.
\end{cases}
\end{equation}
In the energy estimate which will appear in the sequel, the second norm on the left of \eqref{L21} follows from a  weighted Hardy estimate. The norms in \eqref{L22} follow from a weighted energy estimate and the harmonic coordinate condition respectively.

In order for the estimate to close, we will assume $0<4\delta, 4\overline{\mu} <\min(1-2\mu, s-\frac12, 1-s, \gamma - \frac12, 1-\gamma, 1+\gamma-2s)$, $2\mu + \delta < 1$, and $k\geq 11$, and $\varepsilon_g>0$ is sufficiently small. Due to the freedom in choosing harmonic coordinates, we assume the initial conditions that at time $t = 0$,
\begin{equation}
g_{00} = -\left(1-\tfrac{M\chi}{r}\right), \qquad g_{0i} = 0
\end{equation}

We have analogous estimates on components of the raised metric. In particular, if we look at components of $g^{-1}$ in the null frame we see that every error term appearing in $g^{\uLs\uLs}$ and $g^{\uLs\Bs{i}}$ contains a term decaying like $\tz[\gamma']\tp[-1+\delta]$ or $\tp[-2+2\delta]$. It therefore follows from \eqref{MetricApprox2} and \eqref{MetLI} that
\begin{subequations}
\label{RMetric}
\begin{align}
|g^{\uLs\uLs}| + |g^{\uLs\Bs{}}| &\lesssim \varepsilon_g\tz[\gamma]\tp[-1+\delta], \\
\left|g^{\Ls\uLs} + \frac12\right| + \left|g^{\Bs{i}\Bs{i}} - 1\right| &\lesssim \varepsilon_g\tp[-1+\delta], \\
|g^{\Ls\Ls}| + |g^{\Bs{1}\Bs{2}}|+ |g^{\Ls\Bs{i}}| &\lesssim \varepsilon_g\tp[-1+\delta].
\end{align}
\end{subequations}

	\subsection{Lie Derivatives and Commutators}
		We recall the definition of the deformation tensor
\begin{equation}
\label{DT}
\DT{X} = \Lie{X}g = 2\cdot\text{symm}(\nabla X).
\end{equation}
If $X$ is Killing or conformal Killing, $\DT{X}$ is 0 or a scalar multiple of $g$ respectively. In general, we cannot assume any Killing or conformal Killing fields. However, if $g$ is close to Minkowski, we can still establish useful estimates on $\DT{X}$, as follows.

We first take a notational tool $\OLie{\wX}$, defined for the modified Lorentz fields, such that
\begin{subequations}
\label{OLieDef}
\begin{align}
\OLie{{\wX}}(T^{\alpha\beta}) &= (\Lie{{\wX}}T)^{\alpha\beta} + c_{\wX}T^{\alpha\beta}, \\
\OLie{{\wX}}(T_{\alpha\beta}) &= (\Lie{{\wX}}T)_{\alpha\beta} - c_{\wX}T_{\alpha\beta},
\end{align}
\end{subequations}
where the $c_{\wX}$ are the Killing coefficients defined in \eqref{CX}. We define the iterated reduced deformation tensors
\begin{subequations}
\begin{align}
\TDT{{\wX}^I}_{\alpha\beta} &:= (\OLie{{\wX}}^Ig)_{\alpha\beta}, \\
\TDTD{{\wX}^I}^{\alpha\beta}& := (\OLie{{\wX}}^Ig^{-1})^{\alpha\beta}.
\end{align}
\end{subequations}
These are similar up to a sign, which follow from expanding $\TDT{{\wX}^I} = \OLie{{\wX}}^I m^0 + \OLie{{\wX}}^Ih^1$, and bounding the two terms on the right using \eqref{MetricApprox1} and either \eqref{MetLI} or \eqref{MetL2}. Additionally, the reduced deformation tensors satisfy
\begin{equation}\label{eq:DTraised}
\TDT{{\wX}}^{\alpha\beta} = -\TDTD{{\wX}}^{\alpha\beta}, \qquad \TDT{{\wX}^I}^{\alpha\beta} = -\TDTD{{\wX}^I}^{\alpha\beta} + O(\sum_{|I'|\leq |I|}|\TDTD{{\wX}^{I'}}|_{\mcU\mcU}^2),
\end{equation}
which follow from applying $\OLie{\wX}^I$ to the identity
\[
g_{\alpha\beta} = g_{\alpha\gamma}g^{\gamma\delta}g_{\delta\beta}
\]
once and multiple times respectively. For any vector field $X$, taking the trace of \eqref{DT} gives
\begin{equation}\label{id:DivX}
\nabla\cdot X = \tfrac12 g^{\alpha\beta}(\Lie{X}g)_{\alpha\beta},
\end{equation}
and consequently, subtracting off $c_X\text{tr}(g)$ and expanding the Lie derivatives using \eqref{OLieDef} gives
\begin{equation}\label{id:DivX2}
\left|\wY\left(\Lie{\wX}^I(\nabla \cdot \wX_1)\right)\right|\lesssim \sum_{\substack{|I_1|+|I_2| \leq |I|+1}}\left|\wY\left(g^{\alpha\beta}\TDT{\wX^{I_2}}_{\alpha\beta}\right)\right| + \left|\wY\left(\TDTD{\wX^{I_1}}^{\alpha\beta}\TDT{\wX^{I_2}}_{\alpha\beta}\right)\right|.
\end{equation}
These quantities also satisfy analogous estimates to \eqref{MetLI}:
\begin{Proposition}\label{prop:ModifiedDTBounds}
For sufficiently small $\varepsilon_g$, the following estimates follow from \eqref{MetLI}
\begin{subequations}
\label{DTLI}
\begin{align}
|\TDT{{\wX}^I}| +|\TDTD{{\wX}^I}|&\lesssim \varepsilon_g\tp[-1+\delta], \\
|\TDT{\wX^I}_{\mcL\mcT}|+|\TDTD{\wX^I}_{\mcL\mcT}|  &\lesssim \varepsilon_g\tm[\gamma]\tp[-\gamma-1+\delta],
\end{align}
\end{subequations}
for $|I| \leq k - 6$. Additionally,
\begin{subequations}
\label{DTLIm}
\begin{align}
|\TDT{{\wX}^I} - \OLie{\wX}^I h^1| + |\TDTD{{\wX}^I} - \OLie{\wX}^I H_1|&\lesssim \varepsilon_g\tp[-1+\delta], \\
|\TDT{{\wX}^I} - \OLie{\wX}^I h^1|_{\mcL\mcT}  + |\TDTD{{\wX}^I} - \OLie{\wX}^I H_1|_{\mcL\mcT} &\lesssim \varepsilon_g\tm[\gamma]\tp[-\gamma-1+\delta],
\end{align}
\end{subequations}
for all $I$, where the constant in $\lesssim$ depends on $I$.
\end{Proposition}
\begin{proof}
These follow directly from noting 
\[
\OLie{\wX}\mhat = 0, \qquad \OLie{\wX} g = \OLie{\wX}(g-\mhat)
\]
 in both the upper and lower indices, and iterating the estimates \eqref{MetricApprox1} and \eqref{MetricApprox2}.
\end{proof}

We now commute standard and covariant derivatives through Lie derivatives. First, for standard derivatives $\partial_\gamma$, and for tensors $T^{\alpha_1\alpha_2...\alpha_m}_{\beta_1\beta_2...\beta_n}$, we have that
\begin{align}
[\partial_\gamma, \Lie{X}]T &= -(\partial_\gamma\partial_\delta X^{\alpha_1})T^{\delta\alpha_2...\alpha_m}_{\beta_1\beta_2...\beta_n} - ... -(\partial_\gamma\partial_\delta X^{\alpha_m})T^{\alpha_1\alpha_2...\delta}_{\beta_1\beta_2...\beta_n} + \\
&\quad+ (\partial_\gamma\partial_{\beta_1} X^{\delta})T^{\alpha_1\alpha_2...\alpha_m}_{\delta\beta_2...\beta_n} + ... + (\partial_\gamma\partial_{\beta_n} X^{\delta})T^{\alpha_1\alpha_2...\alpha_m}_{\beta_1\beta_2...\delta}\nonumber.
\end{align}
We have an analogous result for the covariant derivative:
\begin{align}
[\nabla_\gamma, \Lie{X}]T &= -(\nabla_\gamma\nabla_\delta X^{\alpha_1})T^{\delta\alpha_2...\alpha_m}_{\beta_1\beta_2...\beta_n} - ... -(\nabla_\gamma\nabla_\delta X^{\alpha_m})T^{\alpha_1\alpha_2...\delta}_{\beta_1\beta_2...\beta_n} + \\
&\quad+ (\nabla_\gamma\nabla_{\beta_1} X^{\delta})T^{\alpha_1\alpha_2...\alpha_m}_{\delta\beta_2...\beta_n} + ... + (\nabla_\gamma\nabla_{\beta_n} X^{\delta})T^{\alpha_1\alpha_2...\alpha_m}_{\beta_1\beta_2...\delta}\nonumber.
\end{align}
In each case, if $T$ is a scalar, the corresponding commutator is 0. In the Minkowski metric, these are again 0 whenever $X$ is a Lorentz field, as $X^\alpha$ is constant or linear in the standard frame. 

For all vector fields $X$ and all antisymmetric $(2,0)$-tensors $\bG$, we have the identity
\begin{equation}
\label{FComms}
[\nabla_\beta, \Lie{X}]\bG^{\alpha\beta} = -(\nabla_\delta\nabla_\beta X^\beta)\bG^{\alpha\delta}.
\end{equation}
This is straightforward to prove:
\begin{align*}
[\nabla_\beta, \Lie{X}]\bG^{\alpha\beta} &= \nabla_\beta\left(X^\delta \nabla_\delta \bG^{\alpha\beta} - (\nabla_\delta X^\alpha)\bG^{\delta\beta} - (\nabla_\delta X^\beta)\bG^{\alpha\delta}\right) - X^\delta\nabla_\delta\nabla_\beta \bG^{\alpha\beta} + (\nabla_\delta X^\alpha)\nabla_\beta \bG^{\delta\beta}- \\
&= X^\delta[\nabla_\beta, \nabla_\delta]\bG^{\alpha\beta} - (\nabla_\beta\nabla_\delta X^\alpha)\bG^{\delta\beta} - (\nabla_\beta\nabla_\delta X^\beta)\bG^{\alpha\delta}.
\end{align*}

Expanding the first term using the Riemann curvature tensor, symmetrizing the derivatives in the middle term and commuting the derivatives in the last term, then taking advantage of the antisymmetry of $F$ and the Bianchi identity
\[
R^\alpha_{\beta\gamma\delta} + R^\alpha_{\gamma\delta\beta} + R^{\alpha}_{\delta\beta\gamma} = 0
\]
gives us the desired identity.

Likewise, we can define the complex Lie derivative 
\begin{equation}
\LieC{X} = \Lie{X} + iA_X.
\end{equation}
We can write the commutators
\begin{subequations}
\begin{align}
\label{Comm1}[D_\beta, \LieC{\wX}]\psi &= i\bF_{\beta \wX}\psi,\\
\label{Comm2}[D^\alpha, \LieC{\wX}]\psi &= ig^{\alpha\beta}\bF_{\beta \wX}\psi + \DT{\wX}^{\alpha\beta}D_\beta\psi,\\
\label{Comm3A}[D_\alpha, D_\beta]\eta^\alpha &= -R^\alpha_{\gamma\alpha\beta}\eta^\gamma + i\bF_{\alpha\beta}\eta^\alpha\\
\label{Comm3B}\wX^\beta[D_\alpha, D_\beta]\eta^\alpha - \eta^\beta [\nabla_\alpha, \nabla_\beta]\wX^\alpha &= i\bF_{\alpha\beta}\eta^\alpha \wX^\beta \\
\label{Comm3}[D_\alpha, \LieC{\wX}]\eta^\alpha &= i\bF_{\alpha\beta}\eta^\alpha \wX^\beta - (\nabla_\beta\nabla_\alpha \wX^\alpha)\eta^\beta, \\
\label{Comm3D}[\nabla_\alpha, \Lie{\wX}]\eta^\alpha &= \nabla_\beta(\nabla\cdot \wX)\eta^\alpha
\end{align}
\end{subequations}
The identity \eqref{Comm1} follows from expanding and using the identity $[D_\alpha, D_\beta]\phi = i\bF_{\alpha\beta}\phi$ and is a direct analogue of the Cartan formula, \eqref{Comm2} follows from writing $D^\alpha = g^{\alpha\beta} D_\beta$, then applying \eqref{Comm1} and \eqref{eq:DTraised}, \eqref{Comm3A} comes from rewriting $D_\alpha = \nabla_\alpha +i\bA_\alpha$ and expanding the commutator, \eqref{Comm3B} follows from the interchange symmetry of the Riemann curvature tensor, and \eqref{Comm3} follows from \eqref{Comm3B} and straightforward calculation. Combining \eqref{Comm2} and \eqref{Comm3} gives
\begin{equation}
\label{CommBox}
[\Box_g^{\mathbb{C}}, D_{\wX}]\psi = [D_\alpha D^\alpha, \LieC{{\wX}}]\psi= iD^\beta(\bF_{\beta {\wX}}\psi) + D_\alpha(\DT{{\wX}}^{\alpha\beta}D_\beta\psi) + i\bF_{\alpha {\wX}}D^\alpha\psi - \nabla_\beta(\nabla\cdot {\wX})D^\beta\psi.
\end{equation}

	\subsection{The Charge Contribution}
		\label{ChargeIntro}
It follows from elliptic theory that even $\bphi$ is compactly supported, we cannot assume that $\bF$ will decay faster than $r^{-2}$ (to see this, decompose $\mathbf{E}[\bF]$ into its divergence free and curl free components, and consider the potential function of the curl-free part). Fortunately, it was shown in \cite{LS} that $\bF$ can be decomposed into $\bF^0+\bF^1$, where $\bF^0$ is explicitly defined and $\bF^1$ decays rapidly. We adapt the consideration there to our class of spacetimes.

We define the charge
\begin{equation}\label{def:q}
q[\bF](t) = \intsig{t}\!\! -\sqrt{|g|}\bJ^0 dx
\end{equation}
Since $\bJ$ is divergence-free we may drop the dependence on time assuming sufficient decay of $\bphi\overline{D\bphi}$.
We define the charge 1-form
\begin{equation}
\mathbf{A}^{\!q} = \left(\int_0^r \frac{q}{4\pi}\frac{\overline{\chi}(s^*-t-2)}{s^{*2}}\, ds^*\right)\,dt,
\end{equation}
where  $s^* = s + M\chi\ln(s)$. Additionally, $\overline\chi$ is a smooth increasing function satisfying
\begin{equation}
\label{ochi}
\overline\chi(y) = \begin{cases}
1 & y > 1, \\
0 & y < 0,
\end{cases}
\end{equation}
We can now define $\bF^0 = d\mathbf{A}^{\!q}$, or
\begin{equation}
\bF^0_{0i} = \omega^{i}\left(\frac{q}{4\pi}\frac{\overline\chi(r^*-t-2)\partial_r(r^*)}{r^{*2}}\right).
\end{equation}
We take the null decomposition:
\begin{subequations}
\begin{align}
\rho[\bF^0] &= \left(\frac{q}{4\pi}\frac{\overline\chi(r^*-t-2)}{r^{*2}}\right), \\
\alpha[\bF^0] &= \underline\alpha[\bF^0] = \sigma[\bF^0] = 0.
\end{align}
\end{subequations}

We can use this to establish component estimates on all Lie derivatives of $\bF^0$. Fortunately our choice of $\bF^0$ makes this process relatively straightforward. We have in particular the estimates
\begin{subequations}
\label{ChargeLInfty}
\begin{align}
|\alpha[\Lie{\wX}^I\bF^0]| &\lesssim |q|\lceil\overline{\chi}\rceil\tm\tp[-3], \\
|\rho[\Lie{\wX}^I\bF^0]| &\lesssim |q|\lceil\overline{\chi}\rceil\tp[-2], \\
|\sigma[\Lie{\wX}^I\bF^0]| &\lesssim |q|\lceil\overline{\chi}\rceil\tp[-2], \\
|\ual[\Lie{\wX}^I\bF^0]| &\lesssim |q|\lceil\overline{\chi}\rceil\tp[-2].
\end{align}
\end{subequations}

These follow from the commutator terms \eqref{FrameComms}, using an analogous argument to Lemma \eqref{NullComms}. We put off discussion of the associated current vector until later.
	\subsection{A Model Morawetz Inequality}
			\label{ModelMorawetz}
Here we prove some model conformal Morawetz-type inequalities, which will motivate features of our energy estimate. We consider the equation
\[
\partial_\alpha(g^{\alpha\gamma}\partial_\gamma\phi) = 0
\]
in Minkowski space, where $g$ is close to the Minkowski metric.
 We take the null frame $\{L = \partial_t+\partial_r, \underline{L} = \partial_t - \partial_r, S_1, S_2\}$,
where $S_j$ are piecewise defined orthonormal fields tangent to spheres of fixed radius, and the conformal field 
\begin{equation}
K_0 = \tfrac12(\tp[2])(\partial_t+\partial_r) + \tfrac12(\tm[2])(\partial_t-\partial_r) = (1+t^2+r^2)\partial_t + 2tr\partial_r.
\end{equation}
 Additionally, we have the optical weights in Minkowski space
\[
\tp[2] = 1+(t+r)^2 \qquad \tm[2] = 1+(t-r)^2, \qquad \tz[2] = \tm[2]/\tp[2].
\]
\subsubsection{The model inequality}\label{MM1}
Our first estimate will show the use of peeling estimates on the metric, as well as give insight as to why additional decay coming from harmonic coordinates is necessary.
\begin{Theorem}\label{MMFirst}
Let $H^{\alpha\beta} = g^{\alpha\beta} - m^{\alpha\beta}$ satisfy
\begin{subequations}\label{MinkPeeling}
\begin{align}
\tp |L(H)|_{LL}+\tm |\underline{L}(H)|_{LL}+|H_{LL}| &\leq \varepsilon_H\tz[2], \\
\tp |L(H)|_{TU}+\tm |\underline{L}(H)|_{TU}+|H_{TU}| &\leq \varepsilon_H\tz, \qquad U\in\{L, S_1, S_2, \}, U\in\{L, S_1, S_2, \underline{L}\} \\
\tp |L(H)|_{\underline{L}\underline{L}}+\tm |\underline{L}(H)|_{\underline{L}\underline{L}}+|H_{\underline{L}\underline{L}}| &\leq \varepsilon_H,
\end{align}
\end{subequations}
where indices are raised and lowered according to the Minkowski metric and 
\[
|X(H)|_{YZ} = |X(H_{\alpha\beta})Y^\alpha Z^\beta|.
\]
There exists a constant $\varepsilon_H > 0$ such that for a smooth function $\phi$ with compact support, and energy
\begin{equation}
\mathcal{E}(t) = \intsig{t}\Big(\tfrac{\tp[2]}{4}\left(\tfrac{L(r\phi)}{r}\right)^2 + \tfrac{\tm[2]}{4}\left(\tfrac{\underline{L}(r\phi)}{r}\right)^2 + \tfrac{\tp[2]+\tm[2]}{4}\sum_j|S_j\phi|^2\Big) \, dx,
\end{equation}
we have the estimate
\begin{equation}
\label{TestMor}
\mathcal{E}(T) \lesssim \mathcal{E}(0) + \varepsilon_H\int_0^T\tfrac{\mathcal{E}(t)}{1+t} + \int_0^T\intsig{t}\left|\tfrac{K_0(r\phi)}{r}\partial_\alpha(g^{\alpha\gamma}\partial_\gamma\phi)\right|.
\end{equation}

\end{Theorem}
This is a standard conformal energy estimate as seen in, e.g.,  \cite{H}, but we take the time to highlight two features. First, in \eqref{MinkPeeling} we require sharper decay on $H_{LL}$ than  (and which plays a major part in the asymptotic system). Second, Gronwall's Lemma gives (slowly) growing energy even if all components of $H$ decay like $\tz[-2]$.

\begin{proof}
 Our main tool is the divergence theorem applied to the quantity
\begin{equation}
P^\alpha = -\left(\tfrac{K_0(r\phi)}{r}g^{\alpha\gamma}\partial_\gamma\phi - \tfrac12 K_0^\alpha g^{\gamma\delta}\partial_\gamma\phi\partial_\delta\phi + \tfrac12(L^\alpha+\underline{L}^\alpha)\phi^2\right).
\end{equation}
The field $K_0$ is conformal Killing (but not Killing) with respect to the Minkowski metric. Integrating along time slices gives
\begin{align}
\label{P0Sample}
E(t) = \intsig{t}P^0 = \intsig{t}&\left(\tfrac{K_0(r\phi)}{r}\partial_t\phi + \tfrac{\tp[2]+\tm[2]}{4}(-L\phi\underline{L}\phi + |S_j\phi|^2) - \phi^2\right) \\
& + \left(-\tfrac{K_0(r\phi)}{r}(H^{L\gamma}\partial_\gamma\phi + H^{\underline{L}\gamma}\partial_\gamma\phi)+ \tfrac14(\tp[2]+\tm[2])H^{\gamma\delta}\partial_\gamma\phi\partial_\delta\phi\right), \nonumber
\end{align}
Therefore,
\begin{equation}
E(T)-E(0) = \int_0^T\intsig{t}\partial_\alpha P^\alpha \, dx \, dt.
\end{equation}
We must show that $E$ and $\mathcal{E}$ are equivalent, and that the integral on the right hand side is bounded by the right hand side of \eqref{TestMor} up to a constant. Writing
\[
\partial_t\phi  = \tfrac12\left(\tfrac{L(r\phi)}{r} + \tfrac{\underline{L}(r\phi)}{r}\right), \qquad -L\phi\underline{L}\phi = -\left(\tfrac{L(r\phi)}{r}\right)\left(\tfrac{\underline{L}(r\phi)}{r}\right) - \tfrac{2\partial_r(r\phi)}{r}\tfrac{\phi}{r} + \left(\tfrac{\phi}{r}\right)^2
\]
so the first line of \eqref{P0Sample} is equivalent to
\begin{equation}
\intsig{t}\Big(\tfrac{\tp[2]}{4}\left(\tfrac{L(r\phi)}{r}\right) + \tfrac{\tm[2]}{4}\left(\tfrac{\underline{L}(r\phi)}{r}\right) + \tfrac{\tp[2]+\tm[2]}{4}\Big(\sum_j|S_j\phi|^2 - \tfrac{2\partial_r(r\phi)}{r}\tfrac{\phi}{r} + \left(\tfrac{\phi}{r}\right)^2\Big) - \phi^2\Big).
\end{equation}
Adding the spatial divergence
\begin{equation}\label{id:SpatDiv}
\partial_i\left[\left(\tfrac{\tp[2]+\tm[2]}{4}\tfrac{\omega_i\phi^2}{r}\right)\right] = \tfrac{\tp[2]+\tm[2]}{4}\tfrac{\phi^2}{r^2} + \tfrac{\tp[2]+\tm[2]}{4}\tfrac{2\phi\partial_r\phi}{r} + \phi^2.
\end{equation}
and expanding $r\partial_r\phi = \partial_r(r\phi)-\phi$ gives $\mathcal{{E}}(t)$.

Now we consider the terms in \eqref{P0Sample} containing $H$, for which we will use the estimates \eqref{MinkPeeling}
It suffices to show that the error terms can be bounded uniformly by $\frac12\mathcal{E}(t)$. We show this for the terms containing $H^{\gamma\delta}\partial_\gamma\phi\partial_\delta\phi$. Other terms follow similarly. Writing $|\overline\partial\phi|^2 = |L\phi|^2 + |S_1\phi|^2 + |S_1\phi|^2$, then
\begin{equation}\label{est:HMink}
|H^{\gamma\delta}\partial_\gamma\phi\partial_\delta\phi| \lesssim  \varepsilon_H\tz[2]|\underline{L}\phi|^2 + \varepsilon_H\tz|\underline{L}\phi||\overline\partial\phi| + \varepsilon_H|\overline\partial\phi|^2\lesssim  \varepsilon_H\tz[2]|\underline{L}\phi|^2 +  \varepsilon_H|\overline\partial\phi|^2,
\end{equation}
so Lemma \ref{PHBasic} implies
\begin{equation}
 \intsig{t}\tp[2]|L(\phi)|^2 + \tm[2]|\underline{L}(\phi)|^2 + \tp[2]\Big|\tfrac{\phi}{r}\Big|^2\, dx,\lesssim \mathcal{E}(t),
\end{equation}
so
\begin{equation}
\intsig{t} \frac14(\tm[2]+\tp[2])|H^{\gamma\delta}\partial_\gamma\phi\partial_\delta\phi| \lesssim\intsig{t} \varepsilon_H\tz[2]\tp[2]|\underline{L}\phi|^2 + \varepsilon_H\tp[2]|\overline\partial\phi|^2 \lesssim  \varepsilon_H\mathcal{E}(t).
\end{equation}
he quantity on the right is bounded by $\frac12\mathcal{E}(t)$ for sufficiently small $\varepsilon_H$. We consider the divergence
\begin{align}
\label{SampleMor1}
\partial_\alpha P^\alpha &= -\tfrac{K_0(r\phi)}{r}\partial_\alpha(g^{\alpha\beta}\partial_\beta\phi) - g^{\alpha\beta}\partial_\alpha\Big(\tfrac{K_0(r\phi)}{r}\Big)\partial_\beta\phi + \tfrac12 K_0(g^{\alpha\beta})\partial_\alpha\phi\partial_\beta\phi + \tfrac12(\partial_\alpha K_0^\alpha)(g^{\beta\gamma}\partial_\beta\phi\partial_\gamma\phi) + \\
&+ g^{\alpha\beta}K_0(\partial_\alpha\phi)\partial_\beta\phi - 2\phi\partial_t\phi. \nonumber
\end{align}
To bound this, we start with the identities
\[
\tfrac{K_0(r\phi)}{r} = K_0(\phi) + 2t\phi, \qquad \partial_\alpha K_0^\alpha = 8t, \qquad [\partial_t, K_0] = 2S,\qquad [\partial_i, K_0] = \Omega_{0i}.
\]
Then,
\[
\partial_\alpha\Big(\tfrac{K_0(r\phi)}{r}\Big) = \partial_\alpha K_0\phi + \partial_\alpha(2t\phi) = [\partial_\alpha, K_0]\phi + K_0(\partial_\alpha\phi)  + 2\delta_\alpha^0\phi + 2t\partial_\alpha\phi,
\]
so
\[
- g^{\alpha\beta}\partial_\alpha\Big(\tfrac{K_0(r\phi)}{r}\Big)\partial_\beta\phi = -g^{\alpha\beta}[\partial_\alpha, K_0]\phi\partial_\beta\phi - g^{\alpha\beta}K_0(\partial_\alpha\phi)\partial_\beta\phi - 2g^{0\beta}\phi\partial_\beta\phi - 2tg^{\alpha\beta}\partial_\alpha\phi\partial_\beta\phi.
\]
Equation \eqref{SampleMor1} can be rewritten as
\begin{align*}
\partial_\alpha P^\alpha &= -\tfrac{K_0(r\phi)}{r}\partial_\alpha(g^{\alpha\beta}\partial_\beta\phi) -g^{\alpha\beta}[\partial_\alpha, K_0]\phi\partial_\beta\phi + \tfrac12 K_0(g^{\alpha\beta})\partial_\alpha\phi\partial_\beta\phi + 2t(g^{\beta\gamma}\partial_\beta\phi\partial_\gamma\phi) - 2H^{0\beta}\phi\partial_\beta\phi . \nonumber
\end{align*}
Noting
\begin{equation}
m^{\alpha\beta}[\partial_\alpha, K_0]\phi\partial_\beta\phi = 2tm^{\alpha\beta}\partial_\alpha\phi\partial_\beta\phi
\end{equation}
as well as the null decomposition
\begin{equation}
L^{\alpha}[\partial_\alpha, K_0]= 2(t+r)L\phi , \qquad \underline{L}^{\alpha}[\partial_\alpha, K_0] = 2(t-r)\underline{L}\phi, \qquad S_j^\alpha[\partial_\alpha, K_0] = 0,
\end{equation}
then, from the bounds \label{MinkPeeling}, \eqref{est:HMink}, a null decomposition, and the estimate $t < \tp$ we can say
\begin{subequations}\label{est:HBoundsMink}
\begin{align}
| 2t(H^{\beta\gamma}\partial_\beta\phi\partial_\gamma\phi)| & \lesssim \varepsilon_H\tm[2]\tp[-1]|\underline{L}\phi|^2 + \varepsilon_H\tp|\overline\partial\phi|^2 \\
|H^{0\beta}\phi\partial_\beta\phi| &\lesssim \varepsilon_H\tz|\phi||\underline{L}\phi| + \varepsilon_H|\phi||\overline\partial\phi| \\
& \lesssim \varepsilon_H(\tp[-1]|\phi|^2 + \tp|\overline\partial\phi|^2 + \tm[2]\tp[-1]|\underline{L}\phi|^2)\nonumber \\
|H^{\alpha\beta}[\partial_\alpha, K_0]\phi\partial_\beta\phi| &\lesssim \varepsilon_H\tp(|\overline\partial\phi|^2 + \tz|\overline\partial\phi||\underline{L}\phi| +\tz[2]|\underline{L}\phi|^2) \\
|K_0(H^{\alpha\beta})\partial_\alpha\phi\partial_\beta\phi| &\lesssim \varepsilon_H\tp(|\overline\partial\phi|^2 + \tz|\overline\partial\phi||\underline{L}\phi| +\tz[2]|\underline{L}\phi|^2).
\end{align}
\end{subequations}
Noting cancellations in \eqref{SampleMor1} and bounding the integral of each term in \eqref{est:HBoundsMink} in space by $(1+t)^{-1}\mathcal{E}(t)$ completes the proof.
\end{proof}

This estimate in itself is not particularly useful, in that applying Gronwall's lemma gives slowly growing energy, which is undesirable, and because in harmonic coordinates we cannot expect $H_{LL}$ to decay faster than $\tz[-1]$.

We approach this problem from both sides. First, we use the fractional Morawetz estimate used by Lindblad and Sterbenz in \cite{LS}. Given this estimate, we would only need $g_{LL}$ to decay like $\tz[2s]\tp[-\epsilon]$ for some $s>\frac12, \epsilon > 0$, with analogous bounds for other components and derivatives of $H$. Next, we use a modified null frame which gives us this decay, see \eqref{MetricApprox1} and \eqref{MetLI}.
\subsubsection{The modified inequalities}\label{MM23}
We now wish to verify that the modified metric gives better decay. We define
\begin{equation}
\wt{K} = (1+\uls^2)\Ls + (1+\us^2)\uLs
\end{equation}
\begin{Lemma}\label{MMThing}
Given the inverse metric $m_0 = m + H_0$, and the energy 
\[
E_0[\phi](T) = \intsig{t}\left(\tp[2]\left(\left|\tfrac{D_\Ls(r^*\phi)}{r^*}\right|^2 + |\slashed{D}\phi|^2\right) + \tm[2]\left|\tfrac{D_\uLs(r^*\phi)}{r^*}\right|^2\right) \, dx,
\] 
it follows that
\begin{equation}
\label{mtEst}
E_0[\phi](T) - E_0[\phi](0) \lesssim M\int_0^T\tfrac{E_0[\phi](t)\langle\ln(1+t)\rangle^2}{1+t} + \int_0^T \intsig{t} \left|\tfrac{K_0^*(r^*\phi)}{r^*}\wpa_\alpha (\widetilde{m}_0^{\alpha\beta}\wpa_\beta\phi) \right|,
\end{equation}
where $\widetilde{m}_0$ is $m+H_0$ expressed in the $\wt{x}^\alpha$ coordinates. Alternatively, for any $c, \overline{\mu}> 0$,
\begin{equation}
\label{mtEst2}
E_0[\phi](T) - E_0[\phi](0) \lesssim M\int_0^T\tfrac{E_0[\phi](t)\langle\ln(1+t)\rangle^2}{1+t} + c\int_0^T\intsig{t}\tm[-1-2\overline{\mu}]\tp[2]\Big|\tfrac{\Ls(r^*\phi)}{r^*}\Big|^2 + c^{-1}\int_0^T \intsig{t}\tp[2]\tm[1+2\overline{\mu}] \left|\wpa_\alpha (\widetilde{m}_0^{\alpha\beta}\wpa_\beta\phi) \right|^2.
\end{equation}
\end{Lemma}
\begin{proof}
The estimate \eqref{mtEst} follows from repeating the proof of \eqref{MMFirst}, in the $\widetilde{x}^\alpha$ coordinates on the background metric $\mhat$ after multiplying the bounds in \eqref{MinkPeeling} by $\langle\ln(1+t)\rangle^2$ to account for the slightly worse bounds of \eqref{MetricApprox2}. \eqref{mtEst2} follows from the Cauchy-Schwartz inequality.
\end{proof}
In \eqref{mtEst} we have slowly growing energy even when $\partial_\alpha (m_0^{\alpha\beta}\partial_\beta\phi)$ vanishes, which we will mitigate using the fractional field $\okos$. Additionally, for the Maxwell-Klein-Gordon system, bounding the second term on the right of \eqref{mtEst2} is nontrivial. By introducing a weight $w$ which grows in $r^* - t$, we can introduce a spacetime term to our energy. We state the following result without proof, which follows directly from multiplying $P^\alpha$ by $w$ and noting that $\nabla w$ is approximately null.
\begin{Lemma}
Given the inverse metric $m_0$, and the energy 
\[
E_w[\phi](T) = \intsig{T}\left(\tp[2]\left(\left|\tfrac{\Ls(r^*\phi)}{r^*}\right|^2 + |\wopa\phi|^2\right) + \tm[2]\left|\tfrac{\uLs(r^*\phi)}{r^*}\right|^2\right)w \, dx,
\]
along with the interior spacetime energy
\[
S_w[\phi](T) = \int_0^T\intsig{t}\Big(\tp[2]\Big|\tfrac{\Ls(r^*\phi)}{r^*}{\Big|}^2 + \tm[2]|\wopa\phi|^2\Big)w'\, dx \, dt,
\]
where 
\[w = w(r^* - t) =
\begin{cases}
1+(1+(t-r^*))^{-2\overline{\mu}} & r^* < t, \\
1+(1+(r^*-t))^{1+2\gamma} & r^* > t,
\end{cases}
\qquad \qquad
w' \approx 
\begin{cases}
(1+(t-r^*))^{-1-2\overline{\mu}} & r^* < t, \\
1+(1+(r^*-t))^{2\gamma} & r^* > t,
\end{cases}
\]
for some constants $\overline{\mu}, \gamma > 0$, we have the estimate
\begin{equation}
E_w[\phi](T) + S_w[\phi](T)  \lesssim E_w[\phi](0) + M\int_0^T\frac{E_w[\phi](t)\langle\ln(1+t)\rangle^2}{1+t} + \int_0^T \intsig{t} \left|\frac{K_0^*(r^*\phi)}{r^*}\wpa_\alpha (\widetilde{m}_0^{\alpha\beta}\wpa_\beta\phi) w\right|.
\end{equation}
\end{Lemma}
\section{\texorpdfstring{$L^2$ Estimates for $\bF$}{L2 Estimates for F}}
	In this section we prove an energy estimate which we will use to bound certain quantities derived from $\bF$. Our basic approach here follows the fractional Morawetz estimate used in \cite{LS}, with the substitution of the modified vector fields, and with additional calculations taken in order to bound the error terms. 
\subsection{Structure of the Estimate}
We first define the field
\begin{equation}
\okos = \frac12\left((1+\uls[2s])\Ls + (1+\abus[2s])\uLs\right)
\end{equation}
for $1/2 < s < (1+\gamma - 4\delta)/2 < 1$, which, for $r^* \leq t$, can be seen as an interpolation between the fields $\wt{Z} = \tfrac12(\uls\Ls + \us\uLs)$ and $\wt{K}_0 = \frac12(\uls[2]\Ls + \us[2]\uLs)$, which correspond to conformal Killing fields in Minkowski space. We add the field $\wpa_t$ in order to ensure that we have a timelike field close to the light cone.

Our core energy estimate will follow from the divergence theorem applied to a quantity $- T[\bG]_{\alpha\beta}\okos^\alpha w$. This does not apply directly to our field $\bF$, since the consequent weighted energy is finite only if the charge is 0. We will instead take the charge decomposition given in Section \ref{ChargeIntro}, and use this on $\bF^1$ and its derivatives.

\subsection{Weights and notation}
 Before we arrive at the statement for the basic estimate, we define the weights
\begin{subequations}
\label{WeightDefs}
\begin{align}
w &= \tm[2(s_0-s)]\overline\chi(-\us) + (1-\overline{\chi}(-\us)),\\
\widetilde{w} &=  (1 + (2-\us)^{2(s_0-s)})\overline{\chi}(-\us) + (1+(2+\us)^{-2{\overline{\mu}}})(1-\overline\chi(-\us)) + \\
&\quad + (1+\uls)^{-2{\overline{\mu}}}\left((2-\us)^{2(s_0-s) + 2{\overline{\mu}}}\overline{\chi}(-\us) + 1 - \overline{\chi}(-\us)\right),\nonumber \\
w_{\overline{\mu}} &= \tm[2(s_0-s)]\overline\chi(-\us) + \tm[2{\overline{\mu}}](1-\overline{\chi}(-\us)),\\
w' &= \tm[2(s_0-s) - 1]\overline\chi(-\us) + \tm[-1-2{\overline{\mu}}](1-\overline{\chi}(-\us)).
\end{align}
\end{subequations}
We recall the assumption that $s_0 < 3/2$, and $0 < 2{\overline{\mu}} < s-1/2$. Here $\overline\chi$ is the same as in equation \eqref{ochi}.

We briefly discuss these weights, which are adapted from weights in \cite{LS}. Our energy estimate will be conducted with respect to the weight $\widetilde{w}$, which satisfies the following properties:
\begin{equation}
\label{tildew}
\widetilde{w} \approx w, \qquad
-\frac12\uLs(\widetilde{w}) \approx w', \qquad
-\frac12\Ls(\widetilde{w}) \approx \tz[1+2{\overline{\mu}}]w'.
\end{equation}
Note that $w'$ describes the derivative of $\widetilde{w}$, and we never differentiate $w$ directly. Additionally, $w_{\overline{\mu}}$ satisfies
\[
\tm w' w_{\overline{\mu}} \approx w^2,
\]
which will be used in our current norm.
We also define
\[
\tau_w = \begin{cases}
\tm[1+2{\overline{\mu}}] & r^* \leq t, \\
\tm & r^* > t.
\end{cases}
\]
We motivate this by noting
\[
\tau_w^{-1} \approx \frac{w'}{w}, \qquad \tm[-1-2{\overline{\mu}}]\leq \tau_w^{-1} \leq \tm[-1].
\]

Recalling the definitions \eqref{def:chargedecomp} and \eqref{def:angularnorms} we also define
\begin{subequations}
\label{def:GNorms}
\begin{align}
\label{def:OG}|\overline{\bG}|^2 &= |\al[\bG]|^2 + |\rho[\bG]|^2 + |\sigma[\bG]|^2,\\
|\bG_N|^2 &= |\alpha[\bG]|^2 + |\rho[\bG]|^2 + |\sigma[\bG]|^2 + |\alpha[\bG]||\ual[\bG]|, \\
|\bG|^2 &= |\al[\bG]|^2 + |\rho[\bG]|^2 + |\sigma[\bG]|^2 + |\ual[\bG]|^2,
\end{align}
\end{subequations}
The components in $|\overline{\bG}|^2$ is analogous to derivatives of a scalar field $\phi$ along the light cone, and $|\bG_N|^2$ is analogous to a form satisfying the null condition. Likewise, we can define the current associated with $\bG$,
\begin{equation}
\bJ[\bG]_\alpha = \nabla^\beta \bG_{\alpha\beta},
\end{equation}
and the spacetime weighted current $L^2$ norm
\begin{equation}
\label{JL2}
\left\lVert \bJ\right\rVert_{L^2[w]} = \left\lVert \tp[s]\tz[-1/2-{\overline{\mu}}]\tm[1/2] \bJ^{\uLs} w^{1/2}_{\overline{\mu}}\right\rVert_2 + \left\lVert\tp[s]\tm[1/2]|\bJ^{\Bs{i}}|w_{\overline{\mu}}^{1/2}\right\rVert_2 + \left\lVert\tz[s-1/2-{\overline{\mu}}]\tm[s+1/2]|\bJ^{\Ls}| w_{\overline{\mu}}^{1/2}\right\rVert_2.
\end{equation} 
We define also the energies we will use in our theorem. For any two-form $\bG_{\alpha\beta}$ defined on $[0,T] \times \mathbb{R}^3$, we define the time-slice energy
\begin{equation}
E_0[\bG](T) = \sup_{0 \leq t \leq T}\intsig{t}\left(\tp[2s]\left(|\al|^2 + |\rho|^2 + |\sigma|^2\right) + \tm[2s]|\ual|^2\right)w \, dx,
\end{equation}
the spacetime energy
\begin{equation}
S_0[\bG](T) = \int_0^T\intsig{t} \left(\tp[2s]|\al|^2 + \tz[1+2{\overline{\mu}}]\left(\tp[2s](|\rho|^2 + |\sigma|^2) + \tm[2s]|\ual|^2\right)\right)w' \, dx \, dt,
\end{equation}
and, defining the region $C_{U, T}$ to be the subset of $[0,T]\times\mathbb{R}^3$ where $\us = U$, the weak conical energy
\begin{equation}
C_0[\bG](T) = \sup_{U}\int_{C_{U, T}}-\left(\tp[2s]\,T[\bG](\Ls, \nabla \us) + \tm[2s]\,T[\bG](\uLs, \nabla\us)\right)w \, dV_C, 
\end{equation}
where $T[\bG]$ is defined in \eqref{def:EMF}. We consequently define the weighted energy of a 2-form with zero charge,
\begin{equation}
\mathcal{E}_0[\bG](T) = E_0[\bG](T) + S_0[\bG](T) + C_0[\bG](T),
\end{equation}
as well as the iterated energy
\begin{equation}
\mathcal{E}_k[\bG](T) = \sum_{\substack{|I| \leq k}} \mathcal{E}_0[\Lie{\wX}^I\bG](T).
\end{equation}

\begin{Theorem}\label{thm:L2F}
Let $\bG$ be a 2-form defined on $[0,T]\times\mathbb{R}^3$ which satisfies $\nabla^\beta \bG_{\alpha\beta} = \bJ$. Then, $\bG$ satisfies the estimate
\begin{equation}
\label{FEnergyEstimate}
E_0[\bG](T) + S_0[\bG](T) + C_0[\bG](T) \lesssim E_0[\bG](0)  + \lnorm \bJ[\bG]\rnorm^2_{L^2[w]}.
\end{equation}
provided that the quantities on the right hand side are bounded.
\end{Theorem}
\begin{Remark}
We say $C_0$ is weak because it is not positive definite. In particular, we do not know that the vectors $\Ls, \nabla u$ are causal, so $C_0$ could have an (unsigned) error of size $\varepsilon_g|\ual|^2$. Later we will bound this error by higher order energies using weighted Klainerman-Sobolev-type inequalities applied to the energy $E_0$. For now we note that $C_0 \geq 0$ by taking $U=T$ in the definition.
\end{Remark}
\begin{proof}
We take
\begin{equation}
P[\bG]_\alpha = -T[\bG]_{\alpha\beta}\okos^\beta \widetilde{w}.
\end{equation}
which has divergence
\begin{equation}
\nabla^\alpha P[\bG]_\alpha = \bG_{\beta\gamma}\okos^\beta J[\bG]^\gamma \widetilde{w} - T[\bG]_{\alpha\beta}\okos^{\beta;\alpha}\widetilde{w} - T[\bG]_{\alpha\beta}\okos^\beta \nabla^\alpha(\widetilde{w}).
\end{equation}
We apply the divergence theorem on $[0,t]\times\mathbb{R}^3$ (with $t\in [0,T]$) to get
\begin{equation}
\intsig{T}-P[\bG]^0 +  \intsig{0}P[\bG]^0 = \int_0^t \intsig{t}\left(\bG(\okos, \bJ[\bG]) \widetilde{w} - T[\bG]_{\alpha\beta}\okos^{\beta;\alpha}\widetilde{w} - T[\bG]_{\alpha\beta}\okos^\beta \nabla^\alpha(\widetilde{w})\right),
\end{equation}
where all integrals are evaluated with respect to the volume element of $g$.

We briefly detail our strategy. The left hand side will give $E_0(0)$ and $E_0(T)$. We will bound the term containing $(\bG(\okos, \bJ)$ using our current norm combined with Hölder's inequality. The approximate conformality of $\okos$ will allow us to discard the term containing $Q[\bG]_{\alpha\beta}\okos^{\beta;\alpha}$. The term containing $\nabla\widetilde{w}$ will give $S_0$. Additionally, we may repeat this estimate over regions exterior to the cone $\us = c$ to get $C_0$.

Recalling the identity \eqref{DT}, we have the commutators
\begin{equation}
[\okos, \Ls] = -\tfrac12\Ls(\uls[2s])\Ls, \qquad
[\okos, \uLs] = -\tfrac12\uLs(\abus[2s])\uLs, \qquad
[\okos, \Bs{j}] = -\tfrac12\tfrac{\uls[2s] - \abus[2s]}{r^*}\Bs{j}.
\end{equation}
The identity \eqref{LieTwoForm} gives the null decomposition
\begin{subequations}
\label{KDefLower}
\begin{align}
(\Lie{\okos}g)_{\Ls\Ls} &= \okos(g_{\Ls\Ls}) + \Ls(\uls[2s])g_{\Ls\Ls},\\
 (\Lie{\okos}g)_{\uLs\uLs} &=  \okos(g_{\uLs\uLs}) + \uLs(\us[2s])g_{\uLs\uLs},\\
(\Lie{\okos}g)_{\Ls\uLs} & = \okos(g_{\uLs\Ls}) + 2s(\uls[2s-1] + \text{sgn}(\us)\abus[2s-1])(g_{\uLs\Ls}),\\
 (\Lie{\okos}g)_{\Ls \Bs{j}} &= \okos(g_{\Ls\Bs{j}}) + \left(2s\uls[2s-1] + \tfrac12\tfrac{\uls[2s] - \abus[2s]}{r^*}\right)g_{\Ls\Bs{j}}, \\
 (\Lie{\okos}g)_{\uLs \Bs{j}} &= \okos(g_{\uLs\Bs{j}}) + \left(2s\cdot\text{sgn}(\us)\abus[2s-1] + \tfrac12\tfrac{\uls[2s] - \abus[2s]}{r^*}\right)g_{\uLs\Bs{j}}, \\
(\Lie{\okos}g)_{\Bs{i}\Bs{j}} &= \okos(g_{\Bs{i}\Bs{j}}) + \tfrac{\uls[2s] - \abus[2s]}{r^*}g_{\Bs{i}\Bs{j}},
\end{align}
\end{subequations}
with other components following from symmetry. Additionally,
\begin{equation}
\tfrac{\uls[2s] - \abus[2s]}{r^*} = t^{2s-1}\tfrac{(1+r^*/t)^{2s} - |1-r^*/t|^{2s}}{r^*/t} = r^{*2s-1}((1+t/{r^*})^{2s} - |1-t/r^*|^{2s}),
\end{equation}
and take the $t^{2s-1}$ and $r^{*2s-1}$ terms for $r^*/t \to 0$, $t/r^* \to 0$ respectively, which gives us
\[
\tfrac12\tfrac{\uls[2s] - \abus[2s]}{r^*} \lesssim \tp[2s-1]
\]
Combining \eqref{MetLI} with \eqref{KDefLower} and the relation $2s + 2\delta < 1+\gamma$ gives
\begin{subequations}\label{DefTenEstimates}
\begin{align}
\left|(\Lie{\okos}g)_{\Ls\Ls}\right| &\lesssim \varepsilon_g \tp[-1-\delta]\tm[\gamma]\\
\left|(\Lie{\okos}g)_{\uLs\uLs}\right| &\lesssim  \varepsilon_g\tp[2s-2+\delta],\\
\left|(\Lie{\okos}g)_{\Ls\uLs} + 2\left(2s(\uls[2s-1] + \text{sgn}(\us)\abus[2s-1])\right)\right| &\lesssim \varepsilon_g\tp[2s-2+\delta],\\
\left|(\Lie{\okos}g)_{\Ls \Bs{j}} \right|&\lesssim \varepsilon_g\tp[-1-\delta]\tm[\gamma]\\
\left|(\Lie{\okos}g)_{\uLs \Bs{j}} \right|&\lesssim  \varepsilon_g\tp[2s-2+\delta], \\
\left|(\Lie{\okos}g)_{\Bs{i}\Bs{j}} - \delta_{ij}\tfrac{\uls[2s] - \abus[2s]}{r^*}\right| & \lesssim  \varepsilon_g\tp[2s-2+\delta],
\end{align}
\end{subequations}
The right hand side vanishes if we replace $g$ with $\mhat$.We now look at terms which will appear in the energy momentum tensor. First, recalling the definition \eqref{def:GNorms}, a null decomposition on $\bG$ combined with \eqref{MetricApprox2} and \eqref{MetLI} gives
\begin{equation}
\left|\bG_{\gamma\delta}\bG^{\gamma\delta} - 2|\sigma|^2 + 2|\rho|^2 + 2\alpha\cdot\ual\right| \lesssim \varepsilon_g(\tp[-1+\delta]|\bG_N|^2 + \tm[\gamma]\tp[-1-\gamma+\delta]|\bG|^2).
\end{equation}
Similar reasoning gives
\begin{subequations}
\label{QFest}
\begin{align}
\left|T_{\Ls\Ls}[\bG] - |\al|^2\right| &\lesssim \varepsilon_g(\tp[-1+\delta]|\al|^2 + \tz[\gamma]\tp[-1+\delta]|\bG_N|^2 + \tz[2\gamma]\tp[-2+2\delta]|\bG|^2), \\
\left|T_{\Ls\uLs}[\bG] - (|\sigma|^2 + |\rho|^2)\right| &\lesssim \varepsilon_g(\tp[-1+\delta]|\bG_N|^2 + \tz[\gamma]\tp[-1+\delta]|\bG|^2), \\
\left|T_{\uLs\uLs}[\bG] - |\ual|^2\right| &\lesssim \varepsilon_g\tp[-1+\delta]|\bG|^2, \\
\left|T_{\Bs{1}\Bs{1}} + T_{\Bs{2}\Bs{2}}- (|\sigma|^2 + |\rho|^2)\right| &\lesssim \varepsilon_g(\tp[-1+\delta]|\bG_N|^2 + \tz[\gamma]\tp[-1+\delta]|\bG|^2), \\
\left|T_{\Bs{j}\Bs{j}}\right| &\lesssim |\bG_N|^2 + \varepsilon_g\tz[\gamma]\tp[-1+\delta]|\bG|^2,\\
\left|T_{\Bs{1}\Bs{2}}\right| &\lesssim |\bG_N|^2 + \varepsilon_g\tz[\gamma]\tp[-1+\delta]|\bG|^2,\\
\left|T_{\Ls \Bs{j}}\right| &\lesssim |\alpha||\overline{\bG}| + \varepsilon_g(\tz[\gamma]\tp[-1+\delta]|\bG||\overline{\bG}| + \tz[2\gamma]\tp[-2+2\delta]|\bG|^2),\\
\left|T_{\uLs \Bs{j}}\right| &\lesssim |\bG||\overline{\bG}| + \varepsilon_g(\tz[\gamma]\tp[-1+\delta]|\bG|^2).
\end{align}
\end{subequations}

We now look at terms containing
\[
\nabla^\alpha\okos^\beta T_{\alpha\beta}.
\]
We will symmetrize it and apply the identity \eqref{DT}.
\begin{Lemma}
\label{DTLemma}
For a symmetric (0,2)-tensor $T$ and a metric $g$ satisfying \eqref{MetLI},
\begin{equation}
\label{DTDecomp}
\nabla^\alpha\okos^\beta T_{\alpha\beta} =  \tfrac12\left(\tfrac{\uls[2s] - \abus[2s]}{r^*} - \drs\left(\uls[2s] - \abus[2s]\right)\right)T_{\Ls\uLs}  + \tfrac12\tfrac{\uls[2s] - \abus[2s]}{r^*}\mathrm{tr }(T) + R_1[T],
\end{equation}
where $R_1[T]$ is a remainder quantity satisfying
\begin{equation}\label{est:R1Bounds}
\left|R_1[T]\right|\lesssim \varepsilon_g\left(\tfrac{\tm}{\tp[1+2\delta]}|T| + \tp[2s-2+2\delta]|T|_{\mcT\mcU}\right).
\end{equation}
Additionally,
\begin{equation}
\label{Lemma32}
\tfrac12\left(\tfrac{\uls[2s] - \abus[2s]}{r^*} - \drs\left(\uls[2s] - \abus[2s]\right)\right)\geq 0.
\end{equation}
If $T$ is the energy-momentum tensor for some 2-form $\bG$, then
\begin{equation}
\label{ErrorF}
\left|R_1[T]\right| \lesssim \frac{\varepsilon_g}{\tp[1+2\delta]}\left(\tp[2s](|\al[\bG]|^2 + |\rho[\bG]|^2 + |\sigma[\bG]|^2) + \tm[2s]|\ual[\bG]|^2\right).
\end{equation}
\end{Lemma}
\begin{proof}
Symmetrizing and applying the identity \eqref{DT} gives
\begin{equation}\label{id:DTExp}
\nabla^\alpha\okos^\beta T_{\alpha\beta} = \frac12(\mhat + (g-\mhat))^{\alpha\gamma}(\mhat + (g-\mhat))^{\beta\delta}(\Lie{\okos}g)_{\alpha\beta}T_{\gamma\delta}.
\end{equation}
We first examine the terms containing $\mhat^{\alpha\gamma}\mhat^{\beta\delta}$, for which we write
\begin{equation}
\frac12\mhat^{\alpha\gamma}\mhat^{\beta\delta}(\Lie{\okos}g)_{\alpha\beta}T_{\gamma\delta} =\frac12\mhat^{\alpha\gamma}\mhat^{\beta\delta}(\Lie{\okos}\mhat)_{\alpha\beta}T_{\gamma\delta} +  \frac12\mhat^{\alpha\gamma}\mhat^{\beta\delta}(\Lie{\okos}(g-\mhat))_{\alpha\beta}T_{\gamma\delta}.
\end{equation}
We take the null decomposition and apply the estimates \eqref{DefTenEstimates} (recalling that the right hand side vanishes for $g = \mhat$) to get
\begin{equation}\label{id:DefTenMink}
\frac12\mhat^{\alpha\gamma}\mhat^{\beta\delta}(\Lie{\okos}\mhat)_{\alpha\beta}T_{\gamma\delta} = -\frac14 \drs\left(\uls[2s] - \abus[2s]\right)(T_{\Ls\uLs} + T_{\uLs\Ls}) + \frac12\sum_{i}\tfrac{\uls[2s] - \abus[2s]}{r^*}T_{\Bs{i}\Bs{i}}.
\end{equation}
This gives the first two terms on the left hand side of \eqref{DTDecomp} up to the error term
\[
\tfrac12\tfrac{\uls[2s] - \abus[2s]}{r^*}(g - \mhat)^{\alpha\beta}T_{\alpha\beta}.
\]
Expanding in the null frame and applying \eqref{RMetric} allows us to bound this within $R_1$. Similarly, expanding and applying \eqref{DefTenEstimates} gives the bound
\begin{equation}
\frac12\mhat^{\alpha\gamma}\mhat^{\beta\delta}(\Lie{\okos}(g-\mhat))_{\alpha\beta}T_{\gamma\delta} \lesssim \varepsilon_g\Big(\tfrac{\tm}{\tp[1+2\delta]}|T| + \tp[2s-2+\delta]|T|_{\mcT\mcU}\Big).
\end{equation}
We now look at terms in \eqref{id:DTExp} containing at least one factor of $g-\mhat$. We may reduce this to bounding
\[
\nabla^\alpha\okos^\beta T_{\alpha\beta} = \frac12 (g-\mhat)^{\alpha\gamma}\mhat^{\beta\delta}(\Lie{\okos}\mhat)_{\alpha\beta}T_{\gamma\delta},
\]
which follows from symmetry of $T$ and the fact that all terms which are quadratic in $g-\mhat$ (including $\Lie{\okos}(g-\mhat)$) must decay like $\tp[2s-3+2\delta]$ or better, which follows from \eqref{DefTenEstimates} and \eqref{RMetric}. We take the null decomposition. We have only one nonzero term containing $T_{\uLs\uLs}$ which we can bound by
\[
|(g-\mhat)^{\uLs\uLs}\mhat^{\Ls\uLs}(\Lie{\okos}\mhat)_{\uLs\Ls}T_{\uLs\uLs}|\lesssim \varepsilon_g \tm[\gamma]\tp[2s-2-\gamma+\delta]|T| \leq\varepsilon_g \tm\tp[-1-\delta].
\]
All other components may be bounded pointwise by $\varepsilon_g\tp[2s-2+\delta]|T|_{\mcT\mcU}$.

In order to prove \eqref{Lemma32} we fix $t$ and write $U(r^*) = \uls[2s] - \abus[2s]$. Then, $U(0) = 0$, and $U'' < 0$. By the mean value theorem, $\frac{U(r^*)}{r^*} = U'(r^\star)$ for some $r^\star \in (0, r^*)$. Since $U'$ is decreasing, $\frac{U(r^*)}{r^*} - U'(r^*)$ is always positive.

The estimate \eqref{ErrorF} follows almost directly from \eqref{est:R1Bounds} and \eqref{QFest}, noting
\[
\tp[2s-2+\delta]|T|_{\mcT\mcU}\lesssim \tp[2s-2+\delta]|\bG||\overline{\bG}| + \tp[2s-3+2\delta]|\bG|^2 \lesssim \tp[2s-1-\delta]|\overline{\bG}|^2 + \tp[2s-3+3\delta]|\bG|^2.
\]
\end{proof}
Now we look at the terms where the derivative falls on the weight. We decompose
\begin{equation}\label{WeightDecompF}
\nabla^\alpha(w) = g^{\alpha\uLs}\uLs(\widetilde{w}) + g^{\alpha\Ls}\Ls(\widetilde{w}) = -\tfrac12\uLs(\widetilde{w})\Ls^\alpha - \tfrac12\Ls(\widetilde{w})\uLs^\alpha + (g-\mhat)^{\alpha\uLs}\uLs(\widetilde{w})+(g-\mhat)^{\alpha\Ls}\Ls(\widetilde{w}).
\end{equation}
\eqref{QFest}, along with the inequalities $1+2\overline\mu < \min(2s, 2\gamma)$, $1+\gamma-2s \geq 4\delta$, implies
\begin{subequations}
\begin{align}
|T_{\okos \Ls} - (\tp[2s]|\alpha|^2 + \tm[2s](|\sigma|^2 + |\rho|^2)|&\lesssim\varepsilon_g(\tp[2s-1+\delta]|\alpha|^2 + \tm[\gamma]\tp[2s-1-\gamma+\delta]|\bG_N|^2 + \tz[1+2\overline{\mu}]\tm[\gamma]|\bG|^2)\\
|T_{\okos \uLs} - (\tp[2s](|\sigma|^2 + |\rho|^2)+ \tm[2s]|\ual|^2)|&\lesssim \varepsilon_g(\tp[2s-1+\delta]|\bG_N|^2 + \tm[2s]|\bG|^2).
\end{align}
\end{subequations}
The inequality 
\[
\tz[\gamma]\tp[2s-1+\delta]|\bG_N|^2 \lesssim \tp[2s]\tz[1+2\mu]|\overline{\bG}|^2 + \tp[2s]|\alpha|^2 + \tp[2s-2+2\delta]\tz[2\gamma]|\bG|^2
\]
implies
\begin{equation}
-\tfrac12\uLs(\widetilde{w})T_{\okos \Ls} - \tfrac12\Ls(\widetilde{w})T_{\okos \uLs} \gtrsim (1-C\varepsilon_g)S_0[\bG].
\end{equation}
Therefore, for sufficiently small $\varepsilon_1$, we obtain the requisite spacetime integral. To bound other terms, we start with the estimates
\begin{subequations}\label{derwerrors}
\begin{align}
|T_{\okos \alpha}(g-\mhat)^{\alpha\uLs}\uLs(\widetilde{w})|&\lesssim  \varepsilon_g\tp[-1+\delta](|T_{\okos \Ls}|+\tz[\gamma]|T_{\okos \uLs}|+\tz[\gamma]|T_{\okos \Bs{1}}|+\tz[\gamma]|T_{\okos \Bs{2}}|)w'\\
|T_{\okos \alpha}(g-\mhat)^{\alpha\Ls}\Ls(\widetilde{w})|&\lesssim \varepsilon_g\tp[-1+\delta]\tz[1+2\overline\mu](|T_{\okos \Ls}|+|T_{\okos \uLs}|+|T_{\okos \Bs{1}}|+|T_{\okos \Bs{2}}|)w'
\end{align}
\end{subequations}
These may be straightforwardly bounded using \eqref{QFest}. First,
\[
\int_0^T\intsig{t}|\varepsilon_g\tp[-1+\delta]T_{\okos\Ls}w'|\lesssim \varepsilon_gS_0.
\]
We handle metric error terms on the right hand side of \eqref{QFest} (those containing $\varepsilon_g$) by first bounding the correlated terms in \eqref{derwerrors} (pointwise) by $\varepsilon_g\tz[\gamma]\tp[2s-2+2\delta]|\bG|^2w'$. Using $2s-2-\gamma+2\delta < -1-2\overline{\mu}$, these are bounded by $S_0[\bG]$. The remaining integral is bounded by
\[
C\int_0^T\intsig{t}|\varepsilon_g\tp[-1+\delta]\tz[\gamma](\tp[2s](|\alpha|^2+|\rho|^2+|\sigma|^2) + \tm[2s]|\ual|^2)w'|.
\]
The inequality $\tz[\gamma]\tp[-1+\delta]\leq\tz[1+2\overline\mu]$ implies
\begin{equation}\label{WeightBoundF}
\int_0^T\intsig{t}T[\bG]_{\alpha\beta}\okos^\beta \nabla^\alpha(\widetilde{w})\, dx\, dt \geq (1-C\varepsilon_g)S_0[\bG]
\end{equation}
We now bound the boundary terms. Since $|g|\approx 1$, it suffices to show that
\begin{equation}
\label{SliceBDF}
\intsig{t}P[\bG]^0 \approx E_0[\bG](t).
\end{equation}
In order to bound this, we apply \eqref{MetricApprox2} and \eqref{MetLI}, and \eqref{QFest}. 

The bounds for the light cone follow from divergence theorem on regions $t\in (0,T)$, $\us \leq U$. It suffices to bound
\begin{equation}
\int_{C_{U, T}}-\nabla^\alpha(\us)\sqrt{|g|}P[\bG]_\alpha \, dV_{\overline{g}},
\end{equation}
where $V_{\overline{g}}$ is the volume with respect to the induced metric, and which is equivalent to the background Minkowski metric. By \eqref{MetricApprox2} and \eqref{MetLI},
\begin{equation}
\left|-\nabla^\alpha(\us)P[\bG]_\alpha - \frac12 P[\bG]_\Ls\right| \lesssim \varepsilon_g\tp[-1+\delta]|P[\bG]|_\Ls +\varepsilon_g \tz[\gamma]\tp[-1+\delta]\left(|P[\bG]_\Bs{1}| +|P[\bG]_\Bs{2}| + |P[\bG]_\uLs|\right).
\end{equation}
\eqref{QFest} gives the component estimate
\begin{equation}
\label{BadConeEstimates}
\left|-\nabla^\alpha(\us)P[\bG]_\alpha - \frac12 P[\bG]_\Ls\right| \lesssim \varepsilon_g\left(\tp[2s-1+\delta]|\alpha|^2 + \tm[2s]\tp[-1+\delta]|\overline{\bG}|^2 + \tm[2s+\gamma]\tp[-1-\gamma+\delta]|\bG|^2\right)\widetilde{w}.
\end{equation}
This is unfortunately not positive definite in itself. We define
\begin{equation}
C^+_0[\bG](T) =  \sup_{\us}\int_{C_{U, T}}\left(\tp[2s]|\alpha|^2+ \tm[2s](|\sigma|^2+|\rho|^2)\right)w \, dV_C, 
\end{equation}
In the course of our $L^\infty$ estimates we will show
\begin{equation}
\label{C0star}
|C^+_0[\bG](T)|\lesssim |C_0[\bG](T)| + \varepsilon_g|E_2[\bG](T)|.
\end{equation}
 by integrating certain $L^\infty$ estimates which depend only on our time-slice and interior energies.
The result (except for the conical energy terms) follows from an application of the divergence theorem (in Minkowski space) on the quantity $\sqrt{|g|}P[\bG]^\alpha$, over the time slab $[0,T] \times \mathbb{R}^3$. First,
\[
\intsig{T}\sqrt{|g|}P[F]^0 -  \intsig{0}\sqrt{|g|}P[F]^0 = \int_0^T \intsig{t}\sqrt{|g|} \left(\bG(\okos, J) \widetilde{w} - T[\bG]_{\alpha\beta}\okos^{\beta;\alpha}\widetilde{w} - T[\bG]_{\alpha\beta}\okos^\beta \nabla^\alpha(\widetilde{w})\right)\, dx \, dt.
\]
Applying \eqref{SliceBDF} on the left, Lemma \ref{DTLemma} and \eqref{WeightBoundF} on the right, and moving the resulting terms to the left, gives
\begin{equation}
E_0[\bG](T) + S_0[\bG](T) \lesssim E_0[\bG](0) + \int_0^T \intsig{t}\left|\bG(\okos, \bJ)\right|w\, dx\, dt + \int_0^T \tfrac{E_0(t)}{(1+t)^{1+2\delta}} \, dt.
\end{equation}
The last term on the right is bounded by $C\varepsilon_1E_0[\bG](T)$. To bound the current term, we take the null decomposition and obtain the bounds
\begin{subequations}
\begin{align}
|\bJ^{\uLs}\bG_{\okos\uLs}| &\lesssim \tp[2s]|\bJ^{\uLs}||\rho| \\
|\bJ^\Bs{j}\bG_{\okos\Bs{j}}| &\lesssim |\bJ^{\Bs{j}}|\left(\tp[2s]|\alpha| + \tm[2s]|\ual|\right) \\
|\bJ^{\Ls}\bG_{\okos\Ls}| &\lesssim \tm[2s]|\bJ^{\Ls}||\ual|
\end{align}
\end{subequations}
Applying H\"older's inequality with the current norm \eqref{JL2} gives
\begin{equation}
\int_0^T \intsig{t}|\bG(\okos, \bJ)|w \, dx \, dt \lesssim S_0[\bG](T)^{1/2}\left\lVert \bJ \right\rVert_{L^2[w]} \lesssim C^{-1}S_0[\bG](T) + C\left\lVert\bJ\right\rVert^2_{L^2[w]}.
\end{equation}
For some $C$ independent of $\varepsilon, \varepsilon_g$, we can subtract the $S_0[F](T)$ term from the left side.

Finally, in order to include the conical energy, we repeat the divergence theorem on regions of the form
\[
([0,T] \times\mathbb{R}^3) \cap \{\us \leq U\}
\]
and take the maximum integral over the reduced light cones.
\end{proof}
\section{\texorpdfstring{$L^2$ Estimates for $\bphi$}{L2 Estimates for Phi}}
	We now establish a conformal energy estimate for a generic function $\phi$. Our approach is modeled by \cite{BS06}, and largely follows that of Section \ref{ModelMorawetz}.

We recall the conformal Killing field
\[
K_0 = \frac12(\underline{u}^2(\partial_t+\partial_r) + u^2(\partial_t - \partial_r)).
\]
Applying the energy estimate gives a term like $t(\text{tr }Q[\bphi])$, which vanishes for $Q[\bG]$. In the Minkowski space one could apply the conformal transformation
\[
{^I}m = \tfrac{1}{r^2}m,
\]
in which $K_0$ is indeed Killing away from the spatial origin, and combine it with a similar estimate using the fundamental solution of the wave equation (cf. \cite{LS}). However, we will instead use a Hardy-type estimate which seems more robust under perturbations at the cost of requiring greater initial decay. Away from the spatial origin we have the identity
\begin{equation}
\label{rBoxComm}\Box_g\phi = \Box_g^{\mathbb{C}}\tfrac{r^*\phi}{r^*} = \tfrac{1}{r^*}\Box_g^{\mathbb{C}}(r^*\phi) + 2 \nabla^\alpha(\tfrac{1}{r^*})D_\alpha(r^*\phi) + (r^*\phi)\Box_g(\tfrac{1}{r^*})
\end{equation}
The singular behavior near the spatial origin is of concern, as $\frac{1}{r^*}$ is no longer a solution of the wave equation. We instead show take an estimate with a background metric $\mhat$, as defined in \ref{def:hatm}, and deal with the remainder separately. We define the energy
\begin{subequations}
\begin{equation}
\label{TSEnergyPhiD}
E_0[\phi](T) = \sup_{0 \leq t \leq T}\intsig{t}\left(\tp[2s]\left(\left|\tfrac{D_\Ls(r^*\phi)}{r^*}\right|^2 + |\slashed{D}\phi|^2 + \left|\tfrac{\phi}{r^*}\right|^2\right) + \tm[2s]|D_\uLs\phi|^2\right) w \, dx.
\end{equation}
It follows from Lemma \ref{HardyEstimate} that this is equivalent to
\begin{equation}
\label{TSEnergyPhiE}
\sup_{0 \leq t \leq T}\intsig{t}\left(\tp[2s]\left(\left|\tfrac{D_\Ls(r^*\phi)}{r^*}\right|^2 + |\slashed{D}\phi|^2\right) + \tm[2s]\left|\tfrac{D_\uLs(r^*\phi)}{r^*}\right|^2\right) w \, dx.
\end{equation}
\end{subequations}

Likewise, we have the spacetime energy
\begin{equation}
\label{STEnergyPhiD}
S_0[\phi](T) = \int_0^T \intsig{t}\left(\tp[2s]\left|\tfrac{D_\Ls(r^*\phi)}{r^*}\right|^2 + \tz[1+2\overline{\mu}]\left(\tp[2s]\left(|\slashed{D}\phi|^2 +\left|\tfrac{\phi}{r^*}\right|^2 \right) + \tm[2s]\left(|D_\uLs\phi|^2 \right)\right)\right) w' \, dx \, dt.
\end{equation}
and the weak conical energy
\begin{equation}
C_0[\phi](T) = \sup_{U}\int_{C_{U,T}} T(\nabla \us,  \okos) w \, dV_C,
\end{equation}
where the partial energy momentum tensor $T$ is defined in equation \eqref{EMPhi}. We take the opportunity to define $U^\star$ to be the value of $U$ for which $C_0$ is defined, and $C^U_{ext}$ to be the set
\[
C^U_{ext}= \bigcup_{\us\leq U}C_{\us,T}
\]
We then define the strong conical energy
\[
C_0^+[\phi](U,T) = \int_{C_{U,T}}\left(\left|\tp[2s]\tfrac{D_{\Ls}(r^*\phi)}{r^*}\right|^2 + \tm[2s]\sum_i\left|D_{\Bs{i}}\phi\right|^2 + \tp[2s]\tz[2]\left|\tfrac{\phi}{r^*}\right|^2\right)w \, dV_C.
\]
We will show later
\[
|C_0^+[\phi](U,T)| \lesssim |C_0[\phi](T)| + \varepsilon_g |E_2[\phi](T)|.
\]
 We also define the combined energies
\begin{equation}
\mathcal{E}_0[\phi](T) = E_0[\phi](T) + S_0[\phi](T) +  C_0[\phi](T)
\end{equation}
and
\begin{equation}
\mathcal{E}_k[\phi](T) = \sum_{|I| \leq k} \mcE_0[D_{\wX}^I\phi](T).
\end{equation}
If $(\bF, \bphi)$ solves \eqref{MKG} we define the full energy
\begin{equation}
\mathcal{E}_k(T) = \mathcal{E}_k[\bphi] + \mathcal{E}_k[\bF^1] + |q[\bF]|^2.
\end{equation}
When there is no ambiguity we write $\mathcal{E}_k$. We define the analogous quantities $E_k[\phi](T)$, $S_k[\phi](T)$, $C_k[\phi](T)$ similarly. We can now state the main theorem of this section.
\begin{Theorem}
\label{L2PhiMain}
For a sufficiently regular function $\phi$ with sufficient spatial decay on $[0,T]\times\mathbb{R}^3$, we have
\begin{equation}
\mathcal{E}_0[\phi](T) \lesssim \mathcal{E}_0[\phi](0) + \left(|q|+\lnorm {\bF^1} \rnorm^{red}_{L^{\infty}[w]}\right)\mathcal{E}_0[\phi](T) + \lnorm \tp[s]\tm[1/2](\Box_g^{\mathbb{C}}\phi) w_{\delta}^{1/2}\rnorm_2^2,
\end{equation}
where
\begin{equation}
\label{def:Fred}
\lnorm {\bF^1} \rnorm^{red}_{L^{\infty}[w]} = \sup_{[0,T]\times\mathbb{R}^3}\tp\tm[1/2+s]\ual[{\bF^1}]w^{1/2} + \tp[1+s]\tm[1/2](|\rho[{\bF^1}]| + |\sigma[{\bF^1}]|)w^{1/2} + \tp[3/2+s]|\al[{\bF^1}]|w^{1/2}.
\end{equation}
\end{Theorem}
We will use \ref{Lem:FirstEEPhi} to show this for $\mhat$, then we handle the error terms using Lemma \ref{RemainderDivergence}. We start with some preliminary estimates.
\begin{Remark}
For the majority of this section we will raise and lower the metric with respect to ${\mhat}$. To reduce ambiguity we will use the notation 
\[
\nabla_{\mhat}^\alpha = \mhat^{\alpha\beta}\nabla_\alpha, \qquad D_{\mhat}^\alpha = \mhat^{\alpha\beta}D_\alpha.
\]
\end{Remark}

We note the following inequalities which will be used many times in the future.
\begin{Proposition} Given the definitions \eqref{TSEnergyPhiD} and \eqref{STEnergyPhiD}, the following estimates hold:
\begin{subequations}
\label{STEnergy}
\begin{align}
\label{STEA}
\lnorm \tp[-1/2-\delta]\tm[s]|D\phi|w^{1/2} \rnorm_2^2 &\lesssim E_0[\phi](T) \\
\label{STEB}
\lnorm \tp[s-3/2-\delta]|\phi|w^{1/2} \rnorm_2^2 &\lesssim E_0[\phi](T) \\
\label{STEC}
\lnorm \tp[s-1/2-\delta]|\overline{D}\phi|w^{1/2} \rnorm_2^2 &\lesssim E_0[\phi](T) \\
\label{STED}
\lnorm \tp[2s-2+\delta]|\overline{D}\phi||D\phi|w\rnorm_1 &\lesssim E_0[\phi](T) \\
\label{STEE}
\lnorm \tp[-s] \frac{D_{\okos}(r^*\phi)}{r^*}(w')^{1/2}\rnorm_2^2 &\lesssim S_0[\phi](T).
\end{align}
\end{subequations}
\end{Proposition}
\begin{proof}
The inequalities \eqref{STEA}-\eqref{STED} follow from the inequality
\begin{equation}
\int_0^T \frac{E_0[\phi](t)}{(1+t)^{1+2\delta}}\, dt \lesssim E_0[\phi](T).
\end{equation}
 Additionally, \eqref{STED} follows from an application of H\"older's inequality, combined with the inequality $s-3/2+2\delta < -1/2-\delta$. Inequality \eqref{STEE} comes from our spacetime energy norm and the inequality $\frac12 + \overline\mu < s$. 
\end{proof}
 Then it follows from \eqref{MetLI}, \eqref{MetricApprox2}, \eqref{STEA} and \eqref{STED}, as well as the inequality $2s-1-\gamma+4\delta < 0$, that
\begin{equation}
\label{STEF}
\lnorm\tp[2s-1]H_1^{\gamma\delta}D_\gamma\phi\overline{D_\delta\phi} w\rnorm_1 \lesssim \varepsilon_g E_0[\phi](T),
\end{equation}
which will be useful later. We define the modified energy momentum tensor
\begin{equation}
\label{EMPhi}
\widetilde{T}[\phi]_{\alpha\beta} = \frac{1}{r^{*2}}\mathfrak{R}\left(D_\alpha(r^*\phi)\overline{D_\beta(r^*\phi)} - \frac12 \mhat_{\alpha\beta}\mhat^{\gamma\delta}D_\gamma(r^*\phi)\overline{D_\delta(r^*\phi)} \right).
\end{equation}

It follows that
\begin{equation}
\nabla_\mhat^\alpha \widetilde{T}[\phi]_{\alpha\beta} = -\tfrac{2\nabla_\mhat^\alpha(r^*)}{r^{*}}\widetilde{T}[\phi]_{\alpha\beta} + \tfrac{1}{r^{*2}}\mathfrak{R}\left(\Box_\mhat^{\mathbb{C}}(r^*\phi)\overline{D_\beta(r^*\phi)}\right) + \tfrac{1}{r^{*2}}\mathfrak{I}\left(r^*\phi \overline{D_{\mhat}^{\alpha}(r^*\phi)}\right)\bF_{\alpha\beta}.
\end{equation}

The identity \eqref{rBoxComm} gives
\[
\tfrac{1}{r^{*2}}\mathfrak{R}\left(\Box_\mhat^{\mathbb{C}}(r^*\phi)\overline{D_\beta(r^*\phi)}\right) = \mathfrak{R}\left(\left(r^*\Box_\mhat^{\mathbb{C}}\phi - 2\nabla_\mhat^\alpha\tfrac{1}{r^*}D_\alpha(r^*\phi)- r^{*2}\phi\Box_\mhat\left(\tfrac{1}{r^*}\right) \right)\tfrac{\overline{D_\beta(r^*\phi)}}{r^*}\right).
\]
Combining these and taking the necessary cancellations gives us
\begin{equation}
\label{EMTDivPhi}
\nabla_\mhat^\alpha \widetilde{T}[\phi]_{\alpha\beta} = \tfrac{\nabla_\mhat^\alpha r^*}{r^{*3}}\mhat_{\alpha\beta}D_\gamma(r^*\phi)\overline{D_{\mhat}^\gamma(r^*\phi)} +\mathfrak{R}\left(\left(\Box_\mhat^{\mathbb{C}}\phi - r^{*}\phi\Box_\mhat\left(\tfrac{1}{r^*}\right)\right)\overline{\tfrac{D_\beta(r^*\phi)}{r^*}}\right) + \mathfrak{I}\left(\phi \tfrac{\overline{D_{\mhat}^{\alpha}(r^*\phi)}}{r^*}\right)\bF_{\alpha\beta}
\end{equation}
We now contract with $\okos$. Since $1/r^*$ is a fundamental solution of the wave operator with respect to $\mhat$ and $\okos(r^*)$ vanishes at 0, 
\begin{equation}
\mathfrak{R}\left( r^{*}\phi\Box_\mhat\left(\tfrac{1}{r^*}\right)\overline{\tfrac{D_\beta(r^*\phi)}{r^*}} \okos^\beta w\right) = 0.
\end{equation}
To bound the last term on the right of \eqref{EMTDivPhi}, we split $\bF = \bF^0  + \bF^1$. Then, $w \lesssim \tm w'$ in the support of $\bF^0$ so
\begin{equation}
\label{CurrentPhiCharge}
\lnorm\phi\tfrac{\overline{D_{\mhat}^\alpha(r^*\phi)}}{r^*}\bF^0_{\alpha\beta}\okos^\beta w \rnorm_1 \lesssim \lnorm\tp[s-1]\tz[1/2+\overline{\mu}]\phi (w')^{1/2}\rnorm_2\lnorm\tp[2-s]\tz[1/2-\overline{\mu}]\tfrac{D_{\mhat}^\alpha(r^*\phi)}{r^*}\bF^0_{\alpha\beta}\okos^\beta(w')^{1/2}\rnorm_2.
\end{equation}
The first term on the right is easily bounded by $\mathcal{E}[\phi]^{1/2}$. Applying \eqref{ChargeLInfty} and expanding in the null frame gives
\begin{equation}
\left|\tfrac{D_{\mhat}^\alpha(r^*\phi)}{r^*}\bF^0_{\alpha\beta}\okos^\beta\right| \lesssim |q|\left(\left|\tp[2s-2]\tfrac{D_{\Ls}(r^*\phi)}{r^*}\right|  + \left|\tm[2s]\tp[-2]\tfrac{D_{\uLs}(r^*\phi)}{r^*}\right|\right).
\end{equation}
The estimate $\tfrac12 + \overline{\mu} < s$ gives
\begin{equation}\label{EEstPhiF0}
\lnorm\phi\tfrac{\overline{D_{\mhat}^\alpha(r^*\phi)}}{r^*}\bF^0_{\alpha\beta}\okos^\beta w \rnorm_1 \lesssim |q|S_0[\phi](T).
\end{equation}
To handle the $\bF^1$ terms, we first take
\begin{equation}
\label{CurrentPhiRem}
\lnorm\phi\tfrac{\overline{D_{\mhat}^\alpha(r^*\phi)}}{r^*}\bF^1_{\alpha\beta}\okos^\beta w \rnorm_1 \lesssim \lnorm\tp[s-3/2-\overline{\mu}]\phi w^{1/2}\rnorm_2\lnorm\tp[3/2+\overline{\mu}-s]\tfrac{D_{\mhat}^\alpha(r^*\phi)}{r^*}\bF^1_{\alpha\beta}\okos^\beta w^{1/2}\rnorm_2.
\end{equation}
Then,
\begin{equation}
\left|\tfrac{D_{\mhat}^\alpha(r^*\phi)}{r^*}\bF^1_{\alpha\beta}\okos^\beta\right| \lesssim \lVert \bF^1\rVert^{red}_{L^{\infty}[w]}\left(\left|\tp[s-1]\tm[-1/2]\tfrac{D_{\Ls}(r^*\phi)}{r^*}\right| + \left|\tp[s-3/2]\tfrac{\slashed{D}(r^*\phi)}{r^*}\right|+ \left|\tp[-1-s]\tm[2s-1/2]\tfrac{D_{\uLs}(r^*\phi)}{r^*}\right|\right)
\end{equation}
Combining this with \eqref{CurrentPhiRem} and \eqref{EEstPhiF0}, as well as $\tfrac12 + \overline\mu < s$, gives
\begin{equation}\label{EEstPhiF}
\big\lVert\phi\tfrac{\overline{D_{\mhat}^\alpha(r^*\phi)}}{r^*}\bF_{\alpha\beta}\okos^\beta w \big\rVert_1 \lesssim (|q| + \lVert \bF^1\rVert^{red}_{L^{\infty}[w]})S_0[\phi](T).
\end{equation}
We recall the deformation tensor estimate \eqref{id:DefTenMink}
\begin{equation}\label{DTStatement}
\text{symm }\nabla_\mhat^\alpha \okos^\beta = \tfrac12 \tfrac{\uls^{2s} - \abus^{2s}}{r^*}\mhat^{\alpha\beta} + \tfrac14(\tfrac{\uls^{2s} - \abus^{2s}}{r^*}) - 2s(\uls^{2s-1} + \us\abus^{2s-2}))(\Ls^\alpha\uLs^\beta + \Ls^\beta\uLs^\alpha),
\end{equation}
This allows us to establish the estimate in Minkowski space:
\begin{Lemma}\label{Lem:FirstEEPhi}
Define $\mathcal{E}_0^+(U,T) = E_0[\phi](T) + S_0[\phi](T) +  C_0^+[\phi](U,T)$. 
Given a function $\phi$ of sufficient regularity and decay we have the bound
\begin{align}
\label{FirstEEPhi}
\mathcal{E}_0^+[\phi](T) &\lesssim  \mathcal{E}_0[\phi](0) + \left(|q| + \lVert \bF^1 \rVert^{red}_{L^{\infty}[w]}\right)\mathcal{E}_0[\phi](T) + \left|\int_{[0,T]\times \Sigma_t}\Box_\mhat^\mathbb{C}\phi \tfrac{\overline{D_{\okos}(r^*\phi)}}{r^*}w\right| + \\
&+ \left|\int_{C^U_{ext}}\Box_{\mhat}^\mathbb{C}\phi \tfrac{\overline{D_{\okos}(r^*\phi)}}{r^*}w\right|\nonumber
\end{align}
\end{Lemma}

\begin{proof}
This follows from an application of the divergence theorem over the regions $[0,T]\times\mathbb{R}^3$ and $C^U_{ext}$.

We have that
\begin{equation}\label{id:PhiDivergence}
-\nabla_\mhat^\alpha \left(\widetilde{T}[\phi]_{\alpha\beta}\okos^\beta \widetilde{w}\right) = - (\nabla_\mhat^{\alpha}\okos^\beta) \widetilde{T}_{\alpha\beta}\wt{w} - \okos^\beta\nabla_\mhat^\alpha\wt{T}_{\alpha\beta}\wt{w} - (\nabla_\mhat \wt{w})^\alpha \okos^\beta \wt{T}_{\alpha\beta}
\end{equation}
The boundary terms on regions like $[0,T]\times\mathbb{R}^3$ give $E_0$, expressed as \eqref{TSEnergyPhiE}, and the additional boundary term on the $C^U_{ext}$ integral gives $C_0$.

Symmetrizing the first term on the right hand side of \eqref{id:PhiDivergence}, applying \eqref{DTStatement} and \eqref{EMTDivPhi} along with positivity of $\widetilde{T}_{\Ls\uLs}$ and the identities
\[
\tfrac{\nabla_\mhat^\alpha r^*}{r^{*}}\okos^\beta\mhat_{\alpha\beta} = \tfrac12\tfrac{\uls^{2s} - \abus^{2s}}{r^*}, \qquad \text{tr}_{\mhat}\wt{T} = -\tfrac{1}{r^{*2}}D_\alpha(r^*\phi)D^\alpha(r^*\phi).
\]
gives
\begin{equation}
-\nabla_\mhat^\alpha \left(\widetilde{T}[\phi]_{\alpha\beta}\okos^\beta \widetilde{w}\right) \leq -\tfrac{1}{r^{*2}}\mathfrak{R}\left(\Box_\mhat^{\mathbb{C}}(r^*\phi)\overline{D_{\okos}(r^*\phi)}\right) - \tfrac{1}{r^{*2}}\mathfrak{I}\left(r^*\phi \overline{D_{\mhat}^{\alpha}(r^*\phi)}\right)\bF_{\alpha\beta}\okos^\beta- \nabla^\alpha_\mhat \wt{w} \okos^\beta \wt{T}_{\alpha\beta}.
\end{equation}
The first term appears in the right hand side of \eqref{FirstEEPhi}, and the second term can be bounded using \eqref{EEstPhiF}. To bound $S_0$, we take the null decomposition
\begin{equation}
\nabla^\alpha(\widetilde{w}) = \Ls^{\alpha}\mhat^{\Ls\uLs}\uLs(\widetilde{w}) + \uLs^\alpha \mhat^{\uLs\Ls}\Ls(\widetilde{w})
\end{equation}
The first term is equivalent to $\Ls^{\alpha}w'$, and the second term is equivalent to  
$\uLs^\alpha\tz[1+2\overline{\mu}]w'$. Expanding out $\wt{T}$ in the null decomposition and integrating gives an integral equivalent to $-S_0$ on the right (after applying the Hardy estimate \eqref{HardyEstimate}), which we may move over to the left hand side.
 \end{proof}
We now bound the perturbations on this energy estimate which arise from the metric. We first define the remainder momentum density tensor
\begin{equation}
\label{RTensorF}
P[\phi]_{rem}^\beta = \mathfrak{R}\Big(\tfrac{\overline{D_{\okos}(r^*\phi)}}{r^*} (g^{\beta\gamma} - \mhat^{\beta\gamma})D_\gamma\phi - \frac12 \okos^\beta (g^{\gamma\delta} - \mhat^{\gamma\delta})D_\gamma\phi\overline{D_\delta\phi}\Big)
\end{equation}
and the conical remainder energy
\begin{equation}
C^U_{err}[\phi](T)= \int_{C_{U,T}}|\partial_\beta(\us)P[\phi]_{rem}^\beta w|\, dV_C,
\end{equation}
\begin{Lemma}
\label{RemainderDivergence}
Given \eqref{MetLI} and a sufficiently regular function $\phi$, the following inequality holds for all $U$:
\begin{align}
\left|\int_{[0,T] \times \Sigma_t}\mathfrak{R}\left(\left(\Box_g^{\mathbb{C}}\phi - \Box_\mhat^{\mathbb{C}}\phi\right)\tfrac{\overline{D_{\okos}(r^*\phi)}}{r^*} w\right)\, dx \, dt\right| + & \\
\left|\int_{C^U_{ext}}\mathfrak{R}\left(\left(\Box_g^{\mathbb{C}}\phi - \Box_\mhat^{\mathbb{C}}\phi\right)\tfrac{\overline{D_{\okos}(r^*\phi)}}{r^*} w\right)\, dx \, dt\right|& \lesssim \varepsilon_g (E_0[\phi]+S_0[\phi])(1+|q|+\lnorm \bF \rnorm^{red}_{L^\infty[w]}) + C^U_{err}[\phi](T).\nonumber
\end{align}
\end{Lemma}
\begin{proof}
Defining the reduced derivative and wave operator
\begin{equation}
\wt{D}_\alpha = \partial_\alpha + iA_\alpha, \qquad\wt\Box^{\mathbb{C}}_\mhat\phi  = \mhat^{\alpha\beta}\wt{D}_\alpha\wt{D}_\beta\phi,
\end{equation}
we have
\begin{equation}
\Box^{\mathbb{C}}_g\phi  - \Box^{\mathbb{C}}_\mhat\phi = (\Box^{\mathbb{C}}_g\phi  - \wt\Box^{\mathbb{C}}_\mhat\phi) + (\wt\Box^{\mathbb{C}}_\mhat\phi  - \Box^{\mathbb{C}}_\mhat\phi)
\end{equation}
We will use the harmonic coordinate condition to bound the first term on the right hand side. The second term can be written in terms of the rapidly decaying Christoffel symbols of $\mhat$ with respect to the background Minkowski metric. By the inequalities \eqref{MetLI}, \eqref{est:ABounds} and \eqref{STEA}, and \eqref{STEE}, we have the pointwise bound
\[
|\mhat^{ab}\wt{D}_a\wt{D}_b\phi - \Box_{\mhat}^{\mathbb{C}}\phi| \lesssim |\partial A^{\alpha}_ a||\partial\phi|
\]
and consequently
\begin{align}
\Big\lVert\left(\mhat^{ab}\wt{D}_a\wt{D}_b\phi - \Box_{\mhat}^{\mathbb{C}}\phi\right)\tfrac{\overline{D_{\okos}(r^*\phi)}}{r^*} w\Big\rVert_1 &\lesssim\varepsilon_g \Big\lVert\tp[s-3/2]\tm[s]|\partial\phi|w^{1/2}\Big\rVert_2\Big\lVert\tp[-s]\tm[-1/2-\overline\mu]\tfrac{\overline{D_{\okos}(r^*\phi)}}{r^*} w^{1/2}\Big\rVert_2 \\
&\lesssim \varepsilon_g(E_0[\phi]+S_0[\phi])\nonumber
\end{align}

Next we want to bound
\begin{equation}
\label{BoxRemainder}
\int_0^T\intsig{t}(g^{\alpha\beta} - \mhat^{\alpha\beta})\wt{D}_\alpha\wt{D}_\beta\phi\tfrac{\overline{D_{\okos}(r^*\phi)}}{r^*}w \, dx \, dt
\end{equation}
in magnitude. Again, we consider $r > \frac{t}{2}$, as proving the results in the far interior is easier. By the divergence theorem,
\begin{equation}\label{DTRemainder}
\int_0^T\intsig{t}\partial_\beta\left(P[\phi]_{rem}^\beta\widetilde{w}\right) \, dx \, dt = \intsig{T}P[\phi]_{rem}^0 \widetilde{w}\, dx - \intsig{0}P[\phi]_{rem}^0 \widetilde{w} \, dx
\end{equation}
We seek to isolate \eqref{BoxRemainder} and bound all other terms using the energy.
First we bound energies of the form
\begin{equation}\label{PhiFixedError}
\Big|\intsig{t}P[\phi]_{rem}^0 \widetilde{w}\Big|\, dx\, dt \lesssim \Big\lVert\tfrac{\overline{D_{\okos}(r^*\phi)}}{r^*}\tp[-s]w^{1/2}\Big\rVert_2 \Big\lVert\varepsilon_g \tp[s-1+\delta] |\wt{D}\phi|w^{1/2}\Big\rVert_2\lesssim \varepsilon_g\mcE_0[\phi](t).
\end{equation}
 When the derivative falls on $\phi$ or $D\phi$, we get
\begin{equation}\label{PhiNiceTerm1}
\mathfrak{R}\left(\tfrac{\overline{D_{\okos}(r^*\phi)}}{r^*}(g-\mhat)^{\beta\gamma}\wt{D}_\beta \wt{D}_\gamma\phi\right) + \tfrac{\okos(r^*)}{r^*}(g-\mhat)^{\beta\gamma}\wt{D}_\gamma\phi\overline{\wt{D}_\beta\phi} + \mathfrak{R}\left(\bF_{\beta \okos}\phi (g-\mhat)^{\beta\gamma}\overline{\wt{D}_\gamma\phi}\right).
\end{equation}
The first term appears in \eqref{BoxRemainder}, and the second and third terms in the corresponding energy integral can be bounded using \eqref{STEF}, \eqref{MetricApprox2}, and (a virtually identical argument to) \eqref{EEstPhiF}, so in particular
\begin{equation}\label{PhiNiceTermErrors}
\Big\lVert\Big(\Big|\tfrac{\okos(r^*)}{r^*}(g-\mhat)^{\beta\gamma}\wt{D}_\gamma\phi\overline{\wt{D}_\beta\phi}\Big| + \Big|\bF_{\beta \okos}\phi (g-\mhat)^{\beta\gamma}\overline{\wt{D}_\gamma\phi}\Big|\Big)w\Big\rVert_1 \lesssim \varepsilon_g\mcE_0[\phi](t)
\end{equation}
When the derivative falls on metric terms of \eqref{RTensorF}, a null decomposition gives
\begin{subequations}
\begin{align*}
\left|(\partial_\beta (g-\mhat)^{\beta\gamma})D_\gamma\phi\right| &\lesssim \varepsilon_g\left(\tp[-1-\gamma+\delta]\tm[\gamma-1]|D\phi| + \tp[-1+\delta]\tm[-1]|\overline{D}\phi|\right), \\
\left|\okos((g-\mhat)^{\gamma\delta})D_\gamma\phi\overline{D_\delta\phi}\right| &\lesssim \varepsilon_g\left(\tp[2s-2-\gamma+\delta]\tm[\gamma]|D\phi|^2 + \tp[2s-2+\delta]|D\phi||\overline{D}\phi|\right),
\end{align*}
\end{subequations}
and consequently, using \eqref{STEE}, as well as $2s-1-\gamma < 0$, $\gamma < s$,
\begin{subequations}
\begin{align*}
\Big\lVert\tfrac{\overline{D_{\okos}(r^*\phi)}}{r^*} (\partial_\beta(g-\mhat)^{\beta\gamma})D_\gamma\phi w\Big\rVert_1 &\lesssim \varepsilon_g\mcE_0^{1/2}\left(\Big\lVert\tp[s-1-\gamma+\delta]\tm[\gamma-1/2+\overline\mu]|D\phi|w^{1/2}\Big\rVert_2 + \Big\lVert\tp[s-1+\delta]\tm[-1/2+\overline{\mu}]|\overline{D}\phi|w^{1/2}\Big\rVert_2\right), \\
\Big\lVert\okos((g-\mhat)^{\gamma\delta})D_\gamma\phi\overline{D_\delta\phi}w\Big\rVert_1&\lesssim \varepsilon_g \Big(\mathcal{E}_0 + \Big\lVert\tm[-s]\tp[s-1/2-\overline{\mu}]|\overline{D}\phi| w^{1/2}\Big\rVert_2 \Big\lVert\tm[s]\tp[-1/2-\delta]|D\phi|w^{1/2}\Big\rVert_2\Big).
\end{align*}
\end{subequations}
Therefore,
\begin{equation}\label{DerivMetricErrorPhi}
\Big\lVert\tfrac{\overline{D_{\okos}(r^*\phi)}}{r^*} (\partial_\beta(g-\mhat)^{\beta\gamma})D_\gamma\phi w\Big\rVert_1 + \Big\lVert\okos((g-\mhat)^{\gamma\delta})D_\gamma\phi\overline{D_\delta\phi}w\Big\rVert_1 \lesssim \varepsilon_g\mathcal{E}_0[\phi].
\end{equation}
When the derivative falls on $\tfrac{1}{r^*}$ or components of $\okos$, we have the terms
\begin{equation}\label{Dokos}
\mathfrak{R}\left((g-\mhat)^{\beta\gamma}\partial_\beta\left(\tfrac{\okos(r^*)}{r^*}\right)\overline{\phi}D_\gamma\phi - \frac12(\partial_\beta\okos^\beta)(g-\mhat)^{\gamma\delta}D_\gamma\phi\overline{D_\delta\phi} + (\partial_\beta\okos^\alpha)\overline{D_\alpha\phi}(g-\mhat)^{\alpha\beta}D_\beta\phi\right).
\end{equation}
In each a null decomposition along with \eqref{MetricApprox2} and \eqref{MetLI} give
\begin{subequations}
\begin{align}
\left\lVert(g-\mhat)^{\beta\gamma}\partial_\beta\left(\tfrac{\okos(r^*)}{r^*}\right)\overline{\phi}D_\gamma\phi w\right\rVert_1 &\lesssim \varepsilon_g\lnorm\tp[2s-3+\delta]|\phi||D\phi| w\rnorm_1 \\
&\lesssim \varepsilon_g \lnorm\tp[s-3/2-\delta]|\phi| w^{1/2}\rnorm_2\lnorm\tp[s-3/2+2\delta]|D\phi|w^{1/2}\rnorm_2,\nonumber \\
\label{ErrorEstPhi}\lnorm(\partial_\beta\okos^\beta)H^{\gamma\delta}D_\gamma\phi\overline{D_\delta\phi} w \rnorm_1 &\lesssim \varepsilon_g\left(\lnorm\tp[2s-2+\delta]|D\phi||\overline{D}\phi|w\rnorm_1 + \lnorm\tp[2s-2-\gamma+\delta]\tm[\gamma]|D\phi|^2w\rnorm_1\right) \\
\lVert\partial_\beta\okos^\alpha)\overline{D_\alpha\phi}(g-\mhat)^{\alpha\beta}D_\beta\phi w\rVert_1&\lesssim \varepsilon_g\left(\lnorm\tp[2s-2+\delta]|D\phi||\overline{D}\phi|w\rnorm_1 + \lnorm\tp[2s-2-\gamma+\delta]\tm[\gamma]|D\phi|^2w\rnorm_1\right).
\end{align}
\end{subequations}
Therefore,
\begin{equation}\label{ErrorEstKPhi}
\big\lVert\big(\big|(g-\mhat)^{\beta\gamma}\partial_\beta\big(\tfrac{\okos(r^*)}{r^*}\big)\overline{\phi}D_\gamma\phi\big| \!+\!  |(\partial_\beta\okos^\beta)(g-\mhat)^{\gamma\delta}D_\gamma\phi\overline{D_\delta\phi}| \!+\! |(\partial_\beta\okos^\gamma)\overline{D_\gamma\phi}(g-\mhat)^{\alpha\beta}D_\beta\phi\big)w\big\rVert_1 \lesssim \varepsilon_g\mathcal{E}_0[\phi]
\end{equation}
When the derivative falls on $\widetilde{w}$, we use the pointwise estimates
\begin{subequations}\label{ErrorEstwPhi}
\begin{align}
|\okos(\wt{w})(g-\mhat)^{\alpha\beta} D_\beta\phi\overline{D_\alpha\phi}| & \lesssim \varepsilon_g(\tp[2s-1+\delta]\tz[1+2\overline\mu] |D\phi||\overline{D}\phi| + \tp[2s-1-\gamma+\delta]\tz[1+2\overline\mu]\tm[\gamma] |D\phi|^2)w', \\
\Big|\tfrac{D_{\okos}(r^*\phi)}{r^*}(g-\mhat)^{\alpha\beta}\overline{D_\alpha\phi}\partial_\beta w\Big| &\lesssim \varepsilon_g\Big|\tp[-s]\tfrac{D_{\okos}(r^*\phi)}{r^*}(w')^{1/2}\Big|(\tp[s-1+\delta]\tz[\gamma]|D\phi|(w')^{1/2} + \tp[s-1+\delta]|\overline{D}\phi|(w')^{1/2})
\end{align}
\end{subequations}
Therefore,
\begin{equation}
\Big\lVert |\okos(\wt{w})(g-\mhat)^{\alpha\beta} D_\beta\phi\overline{D_\alpha\phi}| + \Big|\tfrac{D_{\okos}(r^*\phi)}{r^*}(g-\mhat)^{\alpha\beta}\overline{D_\alpha\phi}\partial_\beta w\Big|\Big\rVert_1 \lesssim \varepsilon_g\mathcal{E}_0[\phi].
\end{equation}
Combining \eqref{DTRemainder}, \eqref{PhiFixedError}, \eqref{DerivMetricErrorPhi}, \eqref{ErrorEstKPhi}, \eqref{ErrorEstwPhi}, \eqref{PhiNiceTerm1}, and \eqref{PhiNiceTermErrors} gives our result.
\end{proof}
Finally, Lemmas \ref{Lem:FirstEEPhi} and \ref{RemainderDivergence} together give Theorem \ref{L2PhiMain}.

\section{\texorpdfstring{$L^\infty$ Estimates for $\bF$}{L Infinity Estimates for F}}
	We now establish $L^\infty$ estimates on Lie derivatives of the charge-modified field $\bF^1$. These for the most part are identical to estimates in Minkowski space (cf. \cite{LS}), modulo our modified frame, as $d\wtx \sim dx$. Our primary tools are Lemma \ref{Sobolev}, which we combine with commutator estimates to convert from derivatives to Lie derivatives. First, we take an estimate which holds for bad components in the extended exterior.
\begin{Lemma}
\label{WorstFLInf}
We have the following uniform estimate on all components of $\bG$:
\begin{equation}
\tp\tm[s+1/2]|\chi\bG|w^{1/2}\lesssim E_2^{1/2}[\bG]
\end{equation}
\end{Lemma}
\begin{proof}
This follows from \eqref{TSSob} with $\delta_+ = 0$ and $\delta_- = s$, noting
\[
|\tm\drs\phi| \lesssim |\Ss\phi| + |\Os{0r}\phi| + |\drs\phi|.
\]
then expanding each of these with respect to the Lorentz fields. We have in particular the estimate
\begin{equation}
\lnorm\tp\tm[1/2+s]\chi |\bG(\dxsi{\alpha}, \dxsi{\beta})|w^{1/2}\rnorm_{L^\infty(\mathbb{R}^3)} \lesssim \sum_{\substack{|I| \leq 2}}\lnorm\tm[s]\wX^I (\chi \bG(\dxsi{\alpha}, \dxsi{\beta}))w^{1/2}\rnorm_{L^2(\mathbb{R}^3)}.
\end{equation}
Then, noting $\wX^I\chi \lesssim 1$, \eqref{FieldComms} combined with \eqref{LieTwoForm} gives our result.
\end{proof}
\begin{Lemma}\label{est:FInteriorLI}
For all components of $\bG$, we have the following estimate:
\begin{equation}
\tp[s + 3/2]|(1-\chi)\bG|w^{1/2} \lesssim E_2^{1/2}[\bG].
\end{equation}
\end{Lemma}
The proof is similar, using \eqref{ISob} and noting that $\tm \approx \tp$ in the support of $(1-\chi)$.
\begin{Lemma}
For the nice components $\alpha, \rho, \sigma$ of $\overline{\bG}$ we have the estimate
\begin{equation}
\tp[s+1]\tm[1/2]|\overline{\bG}|w^{1/2}\lesssim E_2^{1/2}[\bG].
\end{equation}
\end{Lemma}
\begin{proof}
By \eqref{est:FInteriorLI} it suffices to show this for $\chi\overline{\bG}$. We use Lemma \ref{NullComms}. We show this for $\sigma$, as the proof for the other components is identical.
\begin{equation}
\tp[s+1]\tm[1/2]|\chi\sigma[\bG]|w^{1/2}\lesssim \sum_{\substack{|I| \leq 2}}\lnorm \tp[s] \wX^I(\chi \sigma[\bG]) w^{1/2}\rnorm_{L^2(\mathbb{R}^3)}.
\end{equation}
Repeated application of Lemma \ref{NullComms} gives the desired estimate.
\end{proof}
We now look at the $L^2(L^\infty)$ estimate for $\alpha$.
\begin{Lemma}\label{lem:alphamixednorm}
Given a form $\bG$ with sufficient decay, we have the bound
\begin{equation}
\lnorm\tp[s+1]\tm[1/2]\alpha[\bG](w')^{1/2}\rnorm_{L^2(L^\infty)} \lesssim S_2[\bG](T).
\end{equation}
\end{Lemma}
\begin{proof}
We as usual take our Sobolev estimate in the extended exterior
\begin{equation}
\lnorm\tp[s+1]\tm[1/2]\chi\alpha[\bG](w')^{1/2}\rnorm^2_{L^\infty(x)} \lesssim  \sum_{\substack{|I| \leq 2}}\lnorm\tp[s]\wX^I(\chi\alpha[\bG])(w')^{1/2}\rnorm^2_{L^2(x)}.
\end{equation}
We can use the usual expansion in terms of Lie derivatives, followed by the commutator estimate \ref{NullComms}, and it follows that this is contained in $S_2[\bG](T)$. We combine this with the interior estimate
\begin{equation}
\lnorm \tp[3/2+s]|\bG_{\alpha\beta}|(w')^{1/2}\rnorm_{L^\infty(x)}^2\lesssim \sum_{|I| \leq 2}\lnorm\tp[s]\wX^I((1-\chi)\bG_{\alpha\beta})(w')^{1/2}\rnorm_{L^2(x)}^2
\end{equation}
coming from \eqref{ISob}. Integrating in time gives our result.
\end{proof}
We can combine these estimates as follows:
\begin{Lemma}
For any two-form $\bG$ with zero charge and sufficient decay, the following estimate holds on $[0,T]$:
\begin{equation}
\label{LInfFPart1}
\tp\tm[s+1/2]|\ual[\bG]| + \tp[s+1]\tm[1/2](|\alpha[\bG]| + |\rho[\bG]|+|\sigma[\bG]|) + \lnorm\tp[s+1]\tm[1/2]\alpha[\bG](w')^{1/2}\rnorm_{L^2(t)L^\infty(x)} \lesssim \mathcal{E}_2[\bG](T).
\end{equation} 
\end{Lemma}
We will use this to bound components of $\bF^1$ and its Lie derivatives.

Before we proceed, we mention one auxiliary estimate which will give us more precise bounds for the component $\alpha$. We recall the conical energy
\[
C_0[\bG](T) = \sup_{\us}\int_{\{C(\us)\} \cap \{t \in [0,T]\} }\left(\tfrac12 \tp[2s]T(\nabla u^*, \Ls) + \tfrac12\tm[2s]T[\bG](\nabla u^*, \uLs)\right)w \, dVC(\us).
\]
Then, \eqref{QFest} combined with the null decomposition of $g$ gives
\begin{subequations}
\begin{align}
\left|T[\bG](\nabla u^*, \Ls)-|\alpha|^2\right| &\lesssim \varepsilon_g|\al|^2 + \varepsilon_g\tp[-1-\gamma+\delta]\tm[\gamma]|\overline{\bG}|^2 + \varepsilon_g^2\tm[2\gamma]\tp[-2-2\gamma+2\delta]|\bG|^2, \\
\left|T[\bG](\nabla u^*, \uLs)-(|\rho|^2+|\sigma|^2)\right| &\lesssim \varepsilon_g|\overline{\bG}|^2 + \varepsilon_g\tm[\gamma]\tp[-1-\gamma+\delta]|\bG|^2.
\end{align}
\end{subequations}
Therefore, for sufficiently small $\varepsilon_g$, and for $C$ independent of $\varepsilon_g$, the inequality $2s + \delta < 1+\gamma$ gives
\begin{equation}
\label{LCDecomp}
\frac12\tp[2s]T[\bG](\nabla u^*, \Ls) + \frac12\tm[2s]T[\bG](\nabla u^*, \uLs) \geq 1/2(\tp[2s]|\alpha|^2 + \tm[2s](|\rho|^2+|\sigma|^2)) - C\varepsilon_g\tm[2s+\gamma]\tp[-1-\gamma+\delta]|\bG|^2.
\end{equation}
We may bound the last term by
\begin{equation}
\varepsilon_g\tm[2s+\gamma]\tp[-1-\gamma+\delta]|\bG|^2 \lesssim \varepsilon_g \tm[\gamma-1]\tp[-3-\gamma+\delta]\mathcal{E}_2[\bG](T).
\end{equation}
It follows that
\begin{equation}\label{LCDecomp2}
\sup_{\us}\int_{\{C(\us)\} \cap \{t \in [0,T]\} }\varepsilon_g\tm[2s+\gamma]\tp[-1-\gamma+\delta]|\bG|^2 \lesssim \varepsilon_g\mathcal{E}_2[\bG](T).
\end{equation}
The estimate \eqref{C0star} follows. In order to further bound nice components we take the estimate
\begin{equation}
\lnorm\tp[3/2+s]\alpha[\bG] w\rnorm_{L^\infty(C_{\us})} \lesssim  \sum_{\substack{|I|, |J| \leq 1\\ \wX \in\{\uls\Ls, \mathbb{O}\}, \wY \in \mathbb{O}}}\lnorm\tp[2s] \left(\alpha[\Lie{\wX}^I\Lie{\wY}^J \bG] + \tp[2s]\tz\left(\rho[\Lie{\wX}^I\Lie{\wY}^J \bG] + \sigma[\Lie{\wX}^I\Lie{\wY}^J \bG]\right)\right)w\rnorm_{L^2(C_{u^*})}.
\end{equation}
Applying equation \eqref{LCDecomp2} gives us the estimate
\begin{equation}
|\tp[3/2+s]\alpha[\bG] w^{1/2}| \lesssim \mathcal{E}^{1/2}_4[\bG](T).
\end{equation}
We can combine these to get the following:
\begin{Theorem}
\label{LInftyF}
For any two-form $\bG$ with zero associated charge and sufficient decay, the following estimates hold:
\begin{subequations}
\label{est:FTLinfbasic}
\begin{align}
\label{est:FTLinf1}\tp\tm[s+1/2]|\ual[\bG]|w + \tp[s+1]\tm[1/2](|\alpha[\bG]| + |\rho[\bG]|+|\sigma[\bG]|)w + \lnorm\tp[s+1]\tm[1/2]\alpha[\bG](w')^{1/2}\rnorm_{L^2(L^\infty)} &\lesssim \mathcal{E}^{1/2}_2[\bG](T), \\
\label{est:FTLinf2}|\tp[3/2+s]\alpha[\bG] w^{1/2}| &\lesssim \mathcal{E}^{1/2}_4[\bG](T).
\end{align}
\end{subequations}
We have the $L^\infty$ estimates on derivatives:
\begin{subequations}
\label{est:FTLinfty}
\begin{align}
\tp\tm[s+1/2]|\ual[\Lie{\wX}^I\bG]| + \tp[s+1]\tm[1/2](|\rho[\Lie{\wX}^I\bG]|+|\sigma[\Lie{\wX}^I\bG]|) + |\tp[3/2+s]\alpha[\Lie{\wX}^I\bG] w^{1/2}| &\lesssim \mathcal{E}^{1/2}_{|I|+4}[\bG](T) \\
\lnorm\tp[s+1]\tm[1/2]\alpha[\Lie{\wX}^I\bG](w')^{1/2}\rnorm_{L^2(L^\infty)} &\lesssim \mathcal{E}^{1/2}_{|I|+2}[\bG](T) 
\end{align}
\end{subequations}
\end{Theorem}
\section{\texorpdfstring{$L^\infty$ Estimates for $\bphi$}{L Infinity Estimates for Phi}}
	We now establish analogous $L^\infty$ estimates on $\phi$. Before doing so, we show some auxiliary bounds on quantities relating to $\bF$, which will also be useful later.
For an arbitrary two-form $\bG$ we can expand $\wX$ in our null frame to get the pointwise bounds
\begin{subequations}\label{est:FDecomp}
\begin{align}
|\bG(\wX, \Ls)| &\lesssim \tp|\alpha[\bG]| + \tm|\rho[\bG]|, \\
|\bG(\wX, \Bs{i})| &\lesssim \tp(|\alpha[\bG]|+|\sigma[\bG]|) + \tm|\ual[\bG]|, \\
|\bG(\wX, \uLs)| &\lesssim \tp(|\rho[\bG]|+|\ual[\bG]|).
\end{align}
\end{subequations}
The identity \eqref{Comm1} gives
\begin{equation}
|(\Lie{\wX}^ID\phi)_a|\lesssim |\wt{D}_a D_{\wX}^I\phi| + \sum_{|I_1|+|I_2|+1\leq |I|}|D_{\wX}^{I_1}\phi||(\Lie{\wX}^{I_2}\bF)(\wpa_a, \wX_1)|.
\end{equation}
We now show $L^2$ and $L^\infty$ spacetime norms on components of $\Lie{\wX}^ID\phi$, as these crop up naturally in the current norm. We hope to be able to approximate them in the same way as $|DD_X^I\phi|$. We first split $\bF = \bF^0+\bF^1$. If $k-4$ or fewer derivatives fall on $\bF^1$, or if any number of derivatives fall on $\bF^0$, \eqref{ChargeLInfty} and \eqref{est:FTLinfbasic} give
\begin{subequations}
\begin{align*}
\sum_{|I|\leq k}|\alpha[\Lie{\wX}^I\bF^0]| + \sum_{|J|\leq k-4}|\alpha[\Lie{\wX}^J\bF^1]| &\lesssim \mcE_k(T)^{1/2}\tp[-3/2-s]\lrangle{(r^*-t)_+}^{s-1/2}\\
\sum_{|I|\leq k}(|\rho[\Lie{\wX}^I\bF^0]| + |\sigma[\Lie{\wX}^I\bF^0]|) + \sum_{|J|\leq k-4}(|\rho[\Lie{\wX}^J\bF^1]| + |\sigma[\Lie{\wX}^J\bF^1]|) &\lesssim \mcE_k(T)^{1/2}\tp[-1-s]\tm[-1/2]\lrangle{(r^*-t)_+}^{s-1/2}\\
\sum_{|I|\leq k}|\ual[\Lie{\wX}^I\bF^0]| + \sum_{|J|\leq k-4}|\ual[\Lie{\wX}^I\bF^1]| &\lesssim \mcE_k(T)^{1/2}\tp[-1]\tm[-1/2-s]\lrangle{(r^*-t)_+}^{s-1/2}
\end{align*}
\end{subequations}
This combined with \eqref{est:FDecomp} gives
\begin{subequations}
\label{est:FLorentz}
\begin{align}
\sum_{|I|\leq k-1}|(\Lie{\wX}^I\bF^0)_{{\wX}_1\Ls}| + \sum_{|J|\leq k-5}|(\Lie{\wX}^J\bF^1)_{{\wX}_1\Ls}| &\lesssim \mcE_k(T)^{1/2}\tp[-1/2-s]\lrangle{(r^*-t)_+}^{s-1/2}\\
\sum_{|I|\leq k-1}|(\Lie{\wX}^I\bF^0)_{{\wX}_1\Bs{i}}| + \sum_{|J|\leq k-5}(\Lie{\wX}^J\bF^1)_{{\wX}_1\Bs{i}}| &\lesssim \mcE_k(T)^{1/2}\tp[-s]\tm[-1/2]\lrangle{(r^*-t)_+}^{s-1/2}\\
\sum_{|I|\leq k-1}|(\Lie{\wX}^I\bF^0)_{{\wX}_1\uLs}| + \sum_{|J|\leq k-5}|(\Lie{\wX}^J\bF^1)_{{\wX}_1\uLs}| &\lesssim \mcE_k(T)^{1/2}\tm[-1/2-s]\lrangle{(r^*-t)_+}^{s-1/2}
\end{align}
\end{subequations}
We now establish our estimates on $\phi$, which follow from Theorem \ref{Sobolev}, with consideration for commutators, and will be conducted in the $\{\wt{\partial}_a\}$ frame. We first have
\begin{equation}
\tp[2]\tm[2s+1]|D\phi|^2w + \tp[2s+2]\tm\sum_{|I|\leq 1}\left|\tfrac{D_{\wX}\phi}{\tp}\right|^2w
\lesssim \sum_{|I|\leq 2}\lnorm\tm[s]D_{\wX}^I D\phi w^{1/2}\rnorm^2_{L^2(x)} + \sum_{|I|\leq 3}\lnorm\tp[s-1]D_{\wX}^I\phi w^{1/2}\rnorm^2_{L^2(x)}.
\end{equation}
Expanding $\wX$ in our null decomposition allows us to bound the second norm on the right by $E_2[\phi](T)$, so we only need to deal with the first term. We reduce this to the commutator, noting
\[
\sum_{|I|\leq 2 }\lnorm\tm[s]DD_{\wX}^I\phi w^{1/2}\rnorm_{L^2(x)} \lesssim (E_2[\phi](T))^{1/2}.
\] 
\label{D2Dphi}We look at the case $|I| = 2$, as other cases are easier. First, we have the identity
\begin{equation}
[D_{\wX_1}D_{\wX_2},  \Dxsi{a}]\phi = D_{\wX_1}(i\bF(\wX_2, \dxsi{a})\phi)+ i\bF(\wX_1, \dxsi{a})D_{\wX_2}\phi+ D_{[\wX_1, \dxsi{a}]}D_{\wX_2}\phi  + D_{\wX_1}D_{[\wX_2, \wpa_a]}\phi .
\end{equation}
The third and fourth terms are bounded by the energy, and a reduction of order, respectively. We show the bound for the first term, as the second term follows from an easier argument. The identity \eqref{LieTwoForm} gives
\begin{equation}
\label{DFphi}
D_{{\wX}_1}(i\bF({\wX}_2, \dxsi{\alpha})\phi) = i\bF({\wX}_2, \dxsi{\alpha})D_{{\wX}_1}\phi + i(\Lie{{\wX}_1}\bF)({\wX}_2, \dxsi{\alpha})\phi + i\bF([{\wX}_1, {\wX}_2], \dxsi{\alpha})\phi + i\bF({\wX}_2, [{\wX}_1, \dxsi{\alpha}])\phi.
\end{equation}
Therefore,
\begin{equation}
\sum\lnorm \tm[s]D_{\wX_1}(i\bF(\wX_2, \dxsi{\alpha})\phi)w^{1/2}\rnorm_{L^2(x)}\lesssim \sum_{|I|\leq 1}|\tm\Lie{\wX_1}\bF(\wX_2, \dxsi{a})|\sum_{|I|\leq 1 }\lVert\tm[s-1]D_{\wY}\phi w^{1/2}\rVert_{L^2(x)}.
\end{equation}
An application of Lemma \ref{PoincareMain} with $p = 2s-2$, $q = 0$ gives us
\begin{equation}
\lnorm\tm[s-1]D_{\wY}\phi w^{1/2}\rnorm_{L^2(x)} \lesssim (E_1[\phi](T))^{1/2}.
\end{equation}
Therefore, 
\begin{equation}
\sum_{|I|\leq 2}\lnorm\tm[s]D_{\wX}^I D\phi w^{1/2}\rnorm^2_{L^2(x)} + \sum_{|I|\leq 3}\lnorm\tm[s]\tp[-1]D_{\wX}^I\phi w^{1/2}\rnorm^2_{L^2(x)}  \lesssim \Big(1+\tm\sum_{|I|\leq 1}| \Lie{\wX}\bF|\Big)^2E_2[\phi](T).
\end{equation}
We have our first Lemma:
\begin{Lemma}\label{Lemma:FirstLInftyPhi}
For a function $\phi$ defined on $[0,T] \times \Sigma_t$ with suitable regularity,
\begin{equation}
\tp\tm[s+1/2]|D\phi|w^{1/2} + \tp[s+1]\tm[1/2]\sum_{|I|\leq 1 }\left|\tfrac{D_{\wX}\phi}{\tp}\right|w^{1/2}\lesssim \Big(1+\tm\sum_{|I|\leq 1}| \Lie{\wX}\bF|\Big)(E_2[\phi](T))^{1/2}.
\end{equation}
\end{Lemma}
This estimate will be used to bound $D_\uLs\phi$ for $r^* \geq t/2$, and all derivatives when $r^* \leq t/2$, and will be a base for later $L^\infty$ estimates. We now take a look at our better derivatives. We will prove this for a general vector field $U$, noting that the estimate will not necessarily be finite unless $U\in\{\Ls, \Bs{1}, \Bs{2}\}$. Again, Theorem \ref{Sobolev} gives
\begin{equation}
\lnorm\tp[2+2s]\tm[1]|D_{U}\phi|^2w \rnorm_{L^\infty(x)} \lesssim \sum_{|I|\leq 2}\lnorm\tp[s]D_{\wX}^I D_{U}\phi w^{1/2}\rnorm^2_{L^2(x)}.
\end{equation}
We recall the identities
\begin{equation}
\label{D2DSphi}
[D_{\wX_1}D_{\wX_2},  D_{U}]\phi = D_{\wX_1}(i\bF(\wX_2, {U})\phi)+ D_{\wX_1}D_{[\wX_2, {U}]}\phi + i\bF(\wX_1, {U})D_{\wX_2}\phi + D_{[\wX_1, U]}D_{\wX_2}\phi,
\end{equation}
for which a straightforward argument (using \eqref{FieldComms} and \eqref{FrameComms}) again allows us to reduce to bounding the first term, and
\begin{equation}\label{CommDFPhi}
D_{\wX_1}(i\bF(\wX_2, U)\phi) = i\bF(\wX_2, {U})D_{\wX_1}\phi + i(\Lie{\wX_1}\bF)(\wX_2, {U})\phi + i\bF([\wX_1, \wX_2], {U})\phi + i\bF(\wX_2, [\wX_1, {U}])\phi.
\end{equation}
For $U\in\{\Ls, \Bs{1}, \Bs{2}\}$, the commutator identities \eqref{FieldComms} and \eqref{FrameComms} and a null decomposition give
\begin{equation}
|D_{\wX_1}(i\bF(\wX_2, U)\phi)| \lesssim \sum_{|I|\leq 1 }\left(\tp(|\al[\Lie{\wX}\bF]| + |\rho[\Lie{\wX}\bF]| + |\sigma[\Lie{\wX}\bF]|) + \tm|\ual[\Lie{\wX}\bF]|\right)\sum_{|I|\leq 1}|D_{\wY}\phi|
\end{equation}
Recalling the definition \eqref{def:Fred}, we have the inequality
\begin{equation}
\tp[s]\tm[1-s]\left(\tp(|\al[\Lie{\wX}\bF]| + |\rho[\Lie{\wX}\bF]| + |\sigma[\Lie{\wX}\bF]| + \tm|\ual[\Lie{\wX}\bF]|\right)\lesssim |q|+\lnorm \Lie{\wX}\bF \rnorm^{red}_{L^\infty[w]}.
\end{equation}
It follows that
\begin{equation}
\sum_{|I|\leq 2 }\lnorm\tp[s]D_{\wX_1}(i\bF(\wX_2, D_{U})\phi w^{1/2}\rnorm^2_{L^2(x)}\lesssim \Big(1+\sum_{|I|\leq 1}\lnorm \Lie{\wX}\bF \rnorm^{red}_{L^\infty[w]}\Big)^2\sum_{|I|\leq 1 }\lnorm\tm[s-1]D_{\wY}\phi w^{1/2}\rnorm^2.
\end{equation}
Combining this with Lemma \ref{PoincareMain}, gives us our second estimate:
\begin{Lemma}\label{lem:PhiOkay}
For a function $\phi$ defined on $[0,T] \times \Sigma_t$ with suitable regularity, and for $U\in\{\Ls, \Bs{1}, \Bs{2}\}$,
\begin{equation}
\tp[s+1]\tm[1/2]|D_{U}\phi|w^{1/2}\lesssim \Big(1+|q|+\sum_{|I|\leq 1}\lnorm \Lie{\wX}\bF \rnorm^{red}_{L^\infty[w]}\Big)(E_2[\phi](T))^{1/2}.
\end{equation}
\end{Lemma}
We can now take a nicer estimate on $\phi$. Our time-slice Sobolev estimate, \eqref{TSDSob}, gives us
\begin{equation}
\tp\tm[s-1/2]|\phi|w^{1/2} \lesssim \sum_{|I|\leq 1 }\lnorm \tm[s-1]D_{\wX}^I\phi w^{1/2}\rnorm_{L^2(x)}.
\end{equation}
Lemma \ref{PoincareMain} gives the following estimate:
\begin{Lemma}\label{lem:PhiBadComp}
For a function $\phi$ defined on $[0,T] \times \Sigma_t$ with suitable regularity,
\begin{equation}
\tp\tm[s-1/2]|\phi|w^{1/2} \lesssim E_2^{1/2}[\phi](T).
\end{equation}
\end{Lemma}
We now take an $L^2(t)L^\infty(x)$ estimate for the nice component $r^{*-1}D_{\Ls}(r^*\phi)$ in the extended exterior region $r^* > t/2$. We can rewrite Theorem \ref{Sobolev} and integrate to get
\begin{equation}
\label{ASobEst}
\lnorm\tp[s+1]\tm[1/2] \tfrac{D_{\Ls}(r^*\phi)}{r^*}(w')^{1/2}\rnorm^2_{L^2(t)L^\infty(x)} \lesssim \sum_{\substack{|I|, |J| \leq 1 \\ \wt{V}\in \{\drs, \Ss, \Os{0r^*}, \Os{ij}\}\wt{U} \in \{\Os{ij}\}}}\lnorm\tp[s]D_{\wt{V}}^ID_{\wt{U}}^J\Big(\tfrac{D_{\Ls}(r^*\phi)}{r^*}\Big)(w')^{1/2}\rnorm^2_2,
\end{equation}
where we recall the bound
\[
|\tm\drs\phi| \lesssim |\drs\phi| + |\Os{0r^*}\phi| + |\Ss\phi|.
\]
For a vector field $\wt{W}$, we have with the commutator relation
\begin{equation}
\left[D_{\wt{W}}, \frac{1}{r^*}D_\Ls(r^*\cdot)\right]\phi = D_{[\wt{W}, \Ls]}\phi + i\bF(\wt{W}, \Ls)\phi + \wt{W}\left(\tfrac{1}{r^*}\right)\phi
\end{equation}
from which it follows that, for $\wt{V}$ as in \eqref{ASobEst},
\begin{align}
\left[D_{{\wt{V}}}D_{{\wt{U}}}, \tfrac{1}{r^*}D_\Ls(r^*\cdot)\right]\phi &= D_{{\wt{V}}}\left(D_{[{\wt{U}}, \Ls]}\phi + i\bF({\wt{U}}, \Ls)\phi + {\wt{U}}\left(\tfrac{1}{r^*}\right)\phi\right) + \\
&\qquad + D_{[{\wt{V}}, \Ls]}D_{{\wt{U}}}\phi + i(\bF({\wt{V}}, \Ls)D_{{\wt{U}}}\phi) + {\wt{V}}\left(\tfrac{1}{r^*}\right)D_{{\wt{U}}}\phi. \nonumber
\end{align}
Recalling the identities \eqref{FrameComms}, we may simplify
\begin{equation}
\label{ACommFull}
\left[D_{\wt{V}}D_{\Os{ij}}, \tfrac{1}{r^*}D_\Ls(r^*\cdot)\right]\phi = iD_{\wt{V}}(\bF(\Os{ij}, \Ls)\phi) + D_{[\wt{V}, \Ls]}D_{\Os{ij}}\phi +i\bF(\wt{V}, \Ls)D_{\Os{ij}}\phi + \wt{V}\Big(\tfrac{1}{r^*}\Big)D_{\Os{ij}}\phi.
\end{equation}
We bound everything in groups. First, setting $\wt{V} = \Os{0r^*}$, expanding in terms of $\Os{0i}$ (using the linearity of the Lie derivative), and taking \eqref{CommDFPhi}, \eqref{FieldComms} and \eqref{FrameComms}, gives
\begin{equation}
|D_{\wt{V}}(\bF(\Os{ij}, \Ls)\phi)|\lesssim \sum_{|I|+|J|\leq 1}\tp\alpha[\Lie{{\wX}}^I\bF]|D_{\wY}^J\phi|
\end{equation}
Additionally, a null decomposition gives
\begin{equation}
|\bF(\wt{V}, \Ls)|\lesssim \tm|\rho|.
\end{equation}
For all $\wt{V}$, we use \eqref{FieldComms} and \eqref{FrameComms}, and the null decomposition \eqref{RadBoost}, to get
\begin{equation}
D_{[\wt{V}, \Ls]}D_{\Os{ij}}\phi= C_{\wt{V}} D_{\Ls}D_{\Os{ij}}\phi, 
\end{equation}
where $C_{\wt{V}} = -1$ for $\wt{V} = \Ss$ or $\Os{0r^*}$ and 0 for all other fields. Consequently, in the domain of $\chi$ we have
\begin{equation}
\Big|\wt{V}\Big(\tfrac{1}{r^*}\Big) - \tfrac{C_{\wt{V}}}{r^*}\Big|\lesssim \tfrac{\tm}{r^{*2}}, \qquad \Big|D_{[\wt{V}, \Ls]}D_{\Os{ij}}\phi + \wt{V}\Big(\tfrac{1}{r^*}\Big)D_{\Os{ij}}\phi -\tfrac{C_{\wt{V}}\Ls(r^*\phi)}
{r^*}\Big|\lesssim \tfrac{\tm}{r^{*2}}\end{equation}
Therefore, we have the pointwise bounds
\begin{align}
\label{PhiComms}
\sum_{\substack{|I|, |J| \leq 1 \\ \wt{V}\in \{\drs, \Ss, \Os{0r^*}, \Os{ij}\}\wt{U} \in \{\Os{ij}\}}}\Big|\tp[s]D_{\wt{V}}^ID_{\wt{U}}^J\Big(\tfrac{D_{\Ls}(r^*\phi)}{r^*}\Big)\Big| &\lesssim \sum_{|I| \leq 2}\left|\tp[s]\tfrac{D_\Ls(r^*D_{\wX}^I\phi)}{r^*}\right| +  \sum_{|I| \leq 1}\tp[s]\tz\left|\tfrac{D_{\wX}^I\phi}{r^*}\right| +\\
+\sum_{|I| + |J|\leq 1} &\tp[s+1]|\alpha[\Lie{\wX}^I\bF]||D_{\wY}^J\phi| + \sum_{|I| + |J|\leq 1 } \tp[s]\tm|\rho[\Lie{\wX}^I\bF]||D_{\wY}^J\phi| \nonumber.
\end{align}
It suffices to show that all terms appearing on the right hand side of \eqref{PhiComms} can be bounded by the energy when inserted into \eqref{ASobEst}. The first two terms on the right hand side appear almost immediately in $S_2[\phi]$, using the estimate $\tz \leq\tz[1/2+\overline{\mu}]$ for the second term.

In order to bound $\tp|\alpha(\Lie{\wX}\bF)||D_{\wY}\phi|$ the bounds \eqref{ChargeLInfty} give us
\begin{subequations}\label{Est:NiceFieldError1}
\begin{align}
\lnorm\tp[s+1]|\alpha(\Lie{X}{\bF^1})||D_Y\phi|(w')^{1/2}\rnorm_2 &\lesssim |\tp D_Y\phi|\lnorm\tp[s]\alpha[\Lie{X}\bF^1](w')^{1/2}\rnorm_2\lesssim \mathcal{E}_3[\phi](T)^{1/2}\mathcal{E}_1[\bF^1](T)^{1/2},  \\
\lnorm\tp[s+1]|\alpha[\Lie{\wX}{\bF^0}]||D_{\wY}\phi|(w')^{1/2}\rnorm_2 &\lesssim |q|\lnorm\tz\tp[s]\Big|\tfrac{D_{\wY}\phi}{r^*}\Big|(w')^{1/2}\rnorm_2 \lesssim |q|\mathcal{E}_2[\phi](T)^{1/2}
\end{align}
\end{subequations}
Similarly,
\begin{subequations}\label{Est:NiceFieldError2}
\begin{align}
\lnorm\tp[s]\tm|\rho(\Lie{\wX}\bF^1)||D_{\wY}\phi|(w')^{1/2}\rnorm_2 &\lesssim |\tp D_{\wY}\phi|\lnorm\tp[s]\tz\rho(\Lie{\wX}\bF^1)(w')^{1/2}\rnorm_2\lesssim \mathcal{E}_3[\phi](T)^{1/2}\mathcal{E}_1[\bF^1](T)^{1/2}, \\
\lnorm\tp[s]\tm|\rho(\Lie{\wX}\bF^0)||D_{\wY}\phi|(w')^{1/2}\rnorm_2 &\lesssim |q|\lnorm\tz\tp[s]\left|\tfrac{D_{\wY}\phi}{r^*}\right|(w')^{1/2}\rnorm_2\lesssim |q|\mathcal{E}_2[\phi](T)^{1/2}.
\end{align}
\end{subequations}
Therefore,
\begin{equation}
\lnorm\tp[s+1]\tm[1/2] \tfrac{D_{\Ls}(r^*\phi)}{r^*}(w')^{1/2}\rnorm_{L^2(t)L^\infty(x)} \lesssim (1+|q|+\mathcal{E}_1[\bF^1](T)^{1/2})\mathcal{E}_3[\phi](T)^{1/2}.
\end{equation}

We consider the far interior, where derivatives satisfy the same bounds. By \eqref{IDSob}
\[
\lnorm\tp[s+1-\overline\mu]|D\phi|\rnorm^2_{L^2(t)L^\infty(x)\{r > t/2\}} \lesssim  \sum_{|I| \leq 2}\lnorm\tp[s-1/2-\overline\mu]D_{\wX}^ID\phi\rnorm^2_{L^2(t)L^2(x)\{r < 3t/4\}}.
\]
We commute through and bound commutators as in the proof of Lemma \ref{Lemma:FirstLInftyPhi} to get
\begin{equation}
\lnorm\tp[s+1-\overline\mu]|D\phi|\rnorm_{L^2(t)L^\infty(x)\{r > t/2\}} \lesssim \Big(1+\sum_{|I|\leq 1}\lnorm \Lie{\wX}\bF \rnorm^{red}_{L^\infty[w]}\Big)(\mathcal{E}_2[\phi](T))^{1/2}.
\end{equation}
We can combine these to get our main results:
\begin{Lemma}
For a suitably regular function $\phi$, we have the estimate
\begin{equation}
\Big\lVert \tp[s+1]\tm[1/2] \Big(\tfrac{D_\Ls(r^*\phi)}{r^*}\chi_{\{r^* \geq t/2\}} + \Ls(\phi)\chi_{\{r^* \leq t/2\}}\Big)(w')^{1/2}\Big\rVert_{L^2(t)L^\infty(x)} \lesssim \Big(1+\sum_{|I|\leq 1}\lnorm \Lie{\wX}\bF \rnorm^{red}_{L^\infty[w]}\Big)^{1/2}(\mathcal{E}_3[\phi](T)).
\end{equation}
\end{Lemma}
Finally, we turn our attention to the pure $L^\infty$ bound on the nice terms, for which we will use the following approach. Recalling Lemmas \ref{Lem:FirstEEPhi} and \ref{RemainderDivergence}, we will first show the bound
\begin{equation}\label{ConeFirstBound}
C_0^+[\phi](U,T) \lesssim C_0[\phi](T) + \varepsilon_g\Big(1+\sum_{|I|\leq 1}\lnorm \Lie{\wX}\bF \rnorm^{red}_{L^\infty[w]}\Big)^2(E_2[\phi](T)),
\end{equation}
for some uniform constant then apply\eqref{LCDSob}, bounding the commutator terms. By the reverse triangle inequality $C_0^+[\phi](U,T) \geq C_0[\phi](T) - C_{err}^U[\phi](T)$, it suffices to show
\begin{equation}\label{ConeSecondBound}
C^U_{err}[\phi](T)\lesssim \varepsilon_gC_0^+[\phi](U,T) + \varepsilon_g \Big(1+\sum_{|I|\leq 1}\lnorm \Lie{\wX}\bF\rnorm^{red}_{L^\infty[w]}\Big)^2(E_2[\phi](T)).
\end{equation}
This follows first from the estimates
\begin{align*}
|(g^{\alpha\beta}-\mhat^{\alpha\beta})\partial_\alpha\us D_\beta \phi| &\lesssim \varepsilon_g(\tp[-1+\delta]|\overline{D}\phi| + \tz[\gamma]\tp[-1+\delta]|D\phi|),\\
|\okos(\us) (g^{\gamma\delta} - \mhat^{\gamma\delta})D_\gamma\phi\overline{D_\delta\phi}|&\lesssim \varepsilon\tp[-1+\delta]\tm[2s-1](\tz[\gamma]|D\phi|^2 + |D\phi||\overline{D}\phi|)
\end{align*}
Hölder's inequality gives
\begin{equation}
\int_{C_{U,T}}\Big|\tfrac{\tp[2s]\Ls(r^*\phi)}{r^*} \Big| |(g^{\beta\gamma} - \mhat^{\beta\gamma})D_\gamma\phi\partial_\beta(\us)| w\, dV_C\lesssim \varepsilon_g (C_0^+)^{1/2}\Big(\int_{C_{U,T}}\tp[2s]|(g^{\beta\gamma} - \mhat^{\beta\gamma})D_\gamma\phi\partial_\beta(\us)|^2 w\, dV_C\Big)^{1/2}
\end{equation}
Then, Lemmas \ref{lem:PhiBadComp} and \ref{lem:PhiOkay} allow us to bound the right by 
\[
 \Big(1+\sum_{|I|\leq 1}\lnorm \Lie{\wX}\bF \rnorm^{red}_{L^\infty[w]}\Big)(E_2[\phi](T))^{1/2}.
\]
We start with the conical estimate
\begin{equation}
\lnorm\tp[3/2+s]\chi\tfrac{D_{\Ls}(r^*\phi)}{r^*}w^{1/2}\rnorm^2_{L^\infty(C_{U,T})} \lesssim \sum_{\substack{|I|, |J| \leq 1 \\ \wt{V}\in \{\drs, \Ss, \Os{0r^*}, \Os{ij}\}\wt{U} \in \{\Os{ij}\}}} \lnorm\tp[s]D_{\wt{V}}^ID_{\wt{U}}^J\left(\chi\tfrac{D_{\Ls}(r^*\phi)}{r^*}\right)w^{1/2}\rnorm^2_{L^2(C_{U,T})}.
\end{equation}
We can move $\chi$ outside the derivatives on the right hand side with no issue, as when derivatives fall on it we get a uniformly bounded quantity. We again use the bound \eqref{PhiComms} and bound the right hand side term-by-term. The first term on the right appears directly in the energy. For the second term, Lemma \ref{lem:PhiBadComp} gives
\begin{equation}
\tp[s]\tz\chi\left|\tfrac{D_X^I\phi}{r^*}\right|w^{1/2} \lesssim \tp[s-3]\tm[3/2-s]\mathcal{E}_3[\phi]^{1/2}.
\end{equation}
and directly integrate the corresponding $L^2$ integral. The arguments for the third and fourth terms follow nearly exactly those for the bounds \eqref{Est:NiceFieldError1}, \eqref{Est:NiceFieldError2}, with the caveat that we use Lemma \ref{lem:PhiBadComp} to bound the $L^2$ norms containing $\phi$ by $\mcE_3[\phi]$.
Therefore, we can say
\begin{equation}
\left|\tp[3/2+s]\left|\chi\tfrac{D_{\Ls}(r^*\phi)}{r^*}\right|w^{1/2}\right|\lesssim \Big(1+\sum_{|I|\leq 1}\lnorm \Lie{\wX}\bF \rnorm^{red}_{L^\infty[w]}\Big)(\mathcal{E}_3[\phi](T) + C_0^+[\phi](T)) \lesssim \Big(1+\sum_{|I|\leq 1}\lnorm \Lie{\wX}\bF \rnorm^{red}_{L^\infty[w]}\Big)^2\mathcal{E}_4[\phi(T)
\end{equation}
\begin{Proposition}
For any function $\phi$ with sufficient decay, we have the bound
\begin{equation}
\left|\tp[3/2+s]\left|\chi\tfrac{D_{\Ls}(r^*\phi)}{r^*}\right|w^{1/2}\right| \lesssim \Big(1+\sum_{|I|\leq 1}\lnorm \Lie{\wX}\bF \rnorm^{red}_{L^\infty[w]}\Big)^2\mathcal{E}_4^{1/2}[\phi].
\end{equation}
\end{Proposition}
Combining this with Theorem \eqref{est:FTLinfty} and the bounds \eqref{ChargeLInfty}, we have
\begin{Theorem}
\label{LinftyPhi}
Given a function $\phi$, and a two-form $\bF$ as in \eqref{MKG}, we have the estimates 
\begin{subequations}\label{est:LinftyDPhi}
\begin{align}
\left(\tp\tm[s-1/2]|\phi| + \tp\tm[1/2+s]|D_{\uLs}\phi| + \tp[s+1]\tm[1/2](|\overline{D}\phi|)\right) w^{1/2} &\lesssim \mathcal{E}_2^{1/2}[\phi](T)\Big(1+|q[\bF]| + \mcE_4^{1/2}[\bF^1]\Big),\\
\Big\lVert \tp[s+1]\tm[1/2] \Big(\chi\tfrac{D_\Ls(r^*\phi)}{r^*} + (1-\chi)D_{\Ls}(\phi)\Big)(w')^{1/2}\Big\rVert_{L^2(L^\infty)} &\lesssim \mathcal{E}^{1/2}_3[\phi](T)\Big(1+|q[\bF]| + \mcE_4^{1/2}[\bF^1]\Big), \\
\tp[s+3/2]\left|\chi\tfrac{D_{\Ls}(r^*\phi)}{r^*}\right|w^{1/2} &\lesssim \mathcal{E}^{1/2}_4[\phi](T)\Big(1+|q[\bF]| + \mcE_4^{1/2}[\bF^1]\Big)^2,\\
\tp[s]\tm[1/2]|\phi|w^{1/2}&\lesssim \mathcal{E}^{1/2}_1[\phi](T)
\end{align}
\end{subequations}
\end{Theorem}
\begin{Remark}
This is for a generic function $\phi$, so we may find bounds on higher derivatives of $\bphi$ without requiring higher derivatives of $\bF$.
\end{Remark}
	\section{The Bootstrap Estimate}
	We are now able to begin the proof of Theorem \eqref{Main} in earnest. We recall the energy norms
\begin{align*}
\mathcal{E}_0[\bG](T) &= E_0[\bG](T) + S_0[\bG](T)+C_0[\bG](T), \\
\mathcal{E}_0[\phi](T) &=  E_0[\phi](T) + S_0[\phi](T)+C_0[\phi](T), \\
\mathcal{E}_k(T) =\mathcal{E}_k[\bF, \bphi](T) &= \sum_{\substack{|I| \leq k,\wX \in \mbL}} \mathcal{E}_0[\Lie{\wX}^I\bF^1](T) + \mathcal{E}_0[D_{\wX}^I\bphi] + |q[\bF]|^2.
\end{align*}
We fix a constant $k \geq 11$, and take the bootstrap assumption
\begin{equation}\label{bootstrap}
\mathcal{E}_k(T) \leq \varepsilon_b^2.
\end{equation}
For now we assume $\varepsilon_b \leq 1$, and $\varepsilon_1$ is as bounded in the previous sections. Then, Theorem \ref{thm:L2F} implies
\begin{equation}
\label{FE}
\mathcal{E}_k[\bF^1](T) \lesssim \mathcal{E}_k[\bF^1](0) + \sum_{|I| \leq k}\lnorm \bJ[\Lie{\wX}^I \bF^1]\rnorm^2_{L^2[w]}.
\end{equation}
Theorem \ref{LInftyF} and \eqref{ChargeLInfty} imply that for $|I_1| \leq k-2, |I_2|\leq k-4$, and for all $I$,
\begin{subequations}\label{est:FBootstrap}
\begin{align}
\tp\tm[s+1/2]|\ual[\Lie{\wX}^{I_1}\bF^1]|w^{1/2} + \tp[s+1]\tm[1/2](|\rho[\Lie{\wX}^{I_1}\bF^1]|+|\sigma[\Lie{\wX}^{I_1}\bF^1]|)w^{1/2}  &\lesssim \mathcal{E}^{1/2}_{k}(T) \\
 \lVert\tp[s+1]\tm[1/2]\alpha[\Lie{\wX}^{I_1}\bF^1](w')^{1/2}\rVert_{L^2(L^\infty)} + |\tp[3/2+s]\alpha[\Lie{\wX}^{I_2}\bF^1] w^{1/2}| &\lesssim \mathcal{E}^{1/2}_{k}(T), \\
\tp[3]\tm[-1]|\alpha[\Lie{\wX}^I\bF^0]| + \tp[2](|\rho[\Lie{\wX}^I\bF^0]| +|\sigma[\Lie{\wX}^I\bF^0]|  +|\ual[\Lie{\wX}^I\bF^0]|) &\lesssim  \mathcal{E}^{1/2}_{k}(T).
\end{align}
\end{subequations}
It follows that 
\begin{equation}
\lnorm {\bF^1} \rnorm^{red}_{L^{\infty}[w]} \lesssim \mathcal{E}^{1/2}_{4}(T)
\end{equation}
so Theorem \ref{L2PhiMain} implies
\begin{equation}
\label{PhiE}
\mathcal{E}_k[\bphi](T) \lesssim \mathcal{E}_k[\bphi](0) + \mathcal{E}_k^{3/2}(T) + \sum_{|I|\leq k}\lnorm \tp[s]\tm[1/2]\Box_g^{\mathbb{C}}(D_{\wX}^I\bphi)w_\omu^{1/2}\rnorm^2_{L^2[w]}.
\end{equation}
For sufficiently small $\varepsilon_b$, we may say
\begin{equation}
C\mathcal{E}_k^{3/2}(T) \leq \tfrac12 \mathcal{E}_k(T).
\end{equation}
Similarly, for $|I_1| \leq k-2$, $|I_2|\leq k-4$ and $|I_3|\leq k-1$, Theorem \eqref{LinftyPhi} implies
\begin{subequations}\label{est:PhiBootstrap}
\begin{align}
\left(\tp\tm[s-1/2]|D_{\wX}^{I_1}\bphi| + \tp\tm[1/2+s]|D_{\uLs}D_{\wX}^{I_1}\bphi| + \tp[s+1]\tm[1/2](|\overline{D}D_{\wX}^{I_1}\bphi|)\right) w^{1/2} &\lesssim \mathcal{E}_k^{1/2}(T),\\
\Big\lVert \tp[s+1]\tm[1/2] \Big(\chi\tfrac{D_\Ls(r^*D_{\wX}^{I_1}\bphi)}{r^*} + (1-\chi)D_{\Ls}(D_{\wX}^{I_1}\bphi)\Big)(w')^{1/2}\Big\rVert_{L^2(L^\infty)} &\lesssim \mathcal{E}_k^{1/2}(T), \\
\tp[s+3/2]\left|\chi\tfrac{D_{\Ls}(r^*D_{\wX}^{I_2}\bphi)}{r^*}\right|w^{1/2} &\lesssim \mathcal{E}_k^{1/2}(T),\\
\tp[s]\tm[1/2]|D_{\wX}^{I_3}\bphi|w^{1/2}&\lesssim \mathcal{E}_k^{1/2}(T)
\end{align}
\end{subequations}
 Consequently, adding \eqref{FE} and \eqref{PhiE} gives
 \begin{Theorem} \label{thm:MainEnergy}
Given $(\bphi,\bF)$ which solves the MKG system \eqref{MKG} for time $[0,T]$ and satisfies the bootstrap assumption \eqref{bootstrap}, the following estimate holds for sufficiently small $\varepsilon_g, \varepsilon_b$:
\begin{equation}
\mathcal{E}_k(T) \lesssim \mathcal{E}_k(0) + \sum_{|I| \leq k}\lnorm \bJ[\Lie{\wX}^I\bF^1]\rnorm_{L^2[w]}^2 + \sum_{|I| \leq k}\lnorm(\Box_g^{\mathbb{C}}D_{\wX}^I\bphi)\tp[s]\tm[1/2]w_{\overline{\mu}}^{1/2}\rnorm_2^2.
\end{equation}
\end{Theorem}
We restate Theorem \ref{Main} in a form which will  be useful for the full system.
\begin{Theorem}
\label{Bootstrap}
Take constants $s, s_0, \gamma, \mu, \delta, {\overline{\mu}}$ such that $\frac12 < s < 1 < s_0 < 3/2$, $\mu < 1/2$, $\gamma>1/2$, $2s < 1+\gamma$, and $0 < 4\delta, 4\overline{\mu} < \min(1-2\mu, s-\frac12, 1-s, \gamma-\frac12, 1-\gamma, 1+\gamma-2s)$. Additionally, let $k_0$ be an integer with $k_0 \geq 11$.

There exists constants $\varepsilon_0$, $\varepsilon_1> 0$ such that if the metric, in the decomposition \eqref{Metric}, satisfies \eqref{MetLIint} and \eqref{MetL2int} for $\varepsilon_g < \varepsilon_1$, and 
\[
\mathcal{E}_k(0)\leq \varepsilon^2.
\]
for $\varepsilon < \varepsilon_0$, then there exists a constant $C$  independent of $T$ such that
\begin{equation}
\mathcal{E}_k(t)\leq C\varepsilon^2.
\end{equation}
Additionally, the bounds \eqref{est:FBootstrap}, \eqref{est:PhiBootstrap} hold with right hand sides replaced by $C\varepsilon$.
\end{Theorem}
To prove Theorem \ref{Bootstrap}, it suffices to prove the bound
\begin{equation}
\sum_{|I| \leq k}\lnorm \bJ[\Lie{\wX}^I\bF^1]\rnorm_{L^2[w]}^2 + \sum_{|I| \leq k}\lnorm(\Box_g^{\mathbb{C}}D_{\wX}^I\bphi)\tp[s]\tm[1/2]w_{\overline{\mu}}^{1/2}\rnorm_2^2 \lesssim |q|^2+\mathcal{E}_k^2(T) + \varepsilon_g^2\mathcal{E}_k(T).
\end{equation}
We will put this off until the following sections, where it follows from Lemmas \ref{Theorem:CurrentBound} and \ref{lem:PhiCommutators}.
To conclude the proof of Theorem \ref{Main}, we combine Theorem \ref{Bootstrap} with Proposition \ref{thm:InitialData}.
\subsection{The Initial Data Bounds}
Before we conclude this, we must show that the initial conditions of Theorem \ref{thm:MainEnergy} are compatible with those for Theorem \ref{Main}. To be precise, we need to show the following estimate:
\begin{Proposition}\label{thm:InitialData} Under the constraints of Theorem \eqref{Main}, we have the bound
\begin{equation}
\mathcal{E}_k(0) \lesssim \lnorm \mathbf{E}_{0df}\rnorm^2_{H^{k, s_0}(\mathbb{R}^3)} + \lnorm \mathbf{B}_0\rnorm^2_{H^{k, s_0}(\mathbb{R}^3)} + \lnorm D\bphi_0\rnorm^2_{H^{k, s_0}(\mathbb{R}^3)} + \lVert \,\dot{\!\bphi_0}\rVert^2_{H^{k, s_0}(\mathbb{R}^3)}.
\end{equation}
\end{Proposition}
This largely follows the proof in the Minkowski case, with some adaptations made to account for the metric.
From the definition \eqref{def:q}, Hölder's inequality and the condition $s>\tfrac12$,
\begin{equation}
|q|\leq \lnorm(1+r^*)^{s}\bJ^0\rnorm_{L^{6/5}(x)}\lVert(1+r^*)^{-s}\rVert_{L^6(x)} \leq C\lnorm(1+r^*)^{s}\bJ^0\rnorm_{L^{6/5}(x)}
\end{equation}
We now restate a technical elliptic estimate, Lemma 10.1 (and its generalization (3.53))in \cite{LS}:
\begin{Lemma}
Let $\psi$ be a smooth test function on $\mathbb{R}^3$, and define
\begin{equation}
q[\phi] = \int_{\mathbb{R}^3}\psi \, dx.
\end{equation}
Then, for $\tfrac12 < s_0 < \tfrac32$,
\begin{equation}
\int_{\mathbb{R}^3}r^{2s_0}\Big|\nabla\Big(\Delta^{-1}\psi + \tfrac{q[\phi]}{4\pi r}\Big)\Big|^2 \, dx\lesssim \lVert r^{s_0}\psi\rVert^2_{L^{6/5}(\mathbb{R}^3)}
\end{equation}
Additionally,
\begin{equation}
\sum_{|I| \leq k}\int_{\mathbb{R}^3}r^{2s_0+2|I|}\left|\nabla\nabla_x^I\left(\Delta^{-1}\psi + \tfrac{q[\psi]}{4\pi r}\right)\right|^2\,dx \lesssim \sum_{|I| \leq k}\lnorm r^{s_0+|I|}\nabla_x^I \psi\rnorm^2_{L^{6/5}(x)}.
\end{equation}
\end{Lemma}
Applying these estimates to $\psi = -\sqrt{|\wt{g}|}\bJ^0$ in the $\wt{x}$ coordinates gives
\begin{subequations}
\begin{align}
\lnorm r^{*s_0} \left(\mathbf{E}_{cf}^i[\bF] - \bF^0_{0i}\right)\rnorm_{L^2(\mathbb{R}^3)} &\lesssim \lnorm r^{*s_0} (\sqrt{|\wt{g}|}\bJ^0) \rnorm_{L^{6/5}(\mathbb{R}^3)}, \\
\lnorm r^{*s_0+|I|} \left(\mathbf{E}_{cf}^i[\nabla_{\wt{x}}^I\bF] - \nabla_{\wt{x}}^I\bF^0_{0i}\right)\rnorm_{L^2(\mathbb{R}^3)} &\lesssim \lnorm r^{*s_0+|I|}\nabla_{\wt{x}}^I (\sqrt{|\wt{g}|}\bJ^0) \rnorm_{L^{6/5}(\mathbb{R}^3)}.
\end{align}
\end{subequations}
Consequently,
\begin{equation}
\intsig{0}(1+r^*)^{2s_0}|\mathbf{E}[\bF^1]|^2\, dx \lesssim \intsig{0}(1+r^*)^{2s_0}|\mathbf{E}_{df}|^2\, dx + \lnorm(1+r^*)^{s_0}\bJ^0\rnorm^2_{L^{6/5}(x)},
\end{equation}
We bound $\bJ^0$ with the following general estimate on $\mathbb{R}^3$ for suitably regular $f, h$:
\begin{equation}\label{CurrBoundEst}
\lVert \langle r^*\rangle^{s_0}fh\rVert_{L^{6/5} }\leq \lVert\langle r^*\rangle^{s_0}f\rVert_{L^2 }\lVert h\rVert_{L^3 }\leq \lVert\langle r^*\rangle^{s_0}f\rVert_{L^2 }\lVert \partial h\rVert_{L^{3/2} } \leq \lVert\langle r^*\rangle^{s_0}f\rVert_{L^2 }\lVert \langle r^*\rangle^{s_0}\partial h\rVert_{L^{2} }\lVert \langle r^*\rangle^{-s_0}\rVert_{L^{6} }
\end{equation}
To see this, we expand $\widetilde{\bJ}^0 = \wt{g}^{0a}\mathfrak{I}(\bphi\overline{\wt{D}_a\bphi})$.
If $k-6$ or fewer derivatives fall on metric terms $\wt{g}$, we bound it uniformly by 1 and apply \eqref{CurrBoundEst} with $f = \overline{\wt{D}_a\bphi}), h = \bphi$, and applying \eqref{est:Kato}. Otherwise, set $h=g\bphi$, and use Lemma \ref{Lemma:FirstLInftyPhi} to bound the $\bphi$ terms. This gives the result
\begin{equation}
\sum_{|I|\leq k}\lnorm r^{*s_0+|I|}\nabla_{\wt{x}}^I (\sqrt{|\wt{g}|}\bJ^0) \rnorm_{L^{6/5}(\mathbb{R}^3)}\lesssim (1+\varepsilon_g)\mathcal{E}_k(0).
\end{equation}

Next we seek to bound 
\[
\lVert \langle r^*\rangle^{s_0+|I|} \nabla_{\wt{x}}^I\bF^1\rVert_{L^2(\mathbb{R}^3)}
\]
For this we use the identities
\[
d\bF = 0, \qquad \nabla^\beta \bF_{\alpha\beta} = \bJ_\alpha.
\]
which will allow us to bound terms where the derivative falls on the magnetic and electric field components respectively. For an initially split metric $g$, the harmonic coordinate condition implies
\begin{equation}
\wpa_0\bG_{i0}= {\wt{g}}_{00}(\bJ[\bG]_i- \wt{g}^{jk}\wpa_j\bG_{ik}-\widetilde{g}^{ab} \widetilde{\Gamma}^c_{a i}\bG_{cb}).
\end{equation}
We now show the initial bound for $\bF^1$. Given $\Lie{\wt{X}}^I\bF^1$, we decompose, commute time derivatives to the right, and use the $L^\infty$ bounds
\[
\sum_{1 \leq|a|\leq k-6}(1+r^*)^a\wpa^a \wt{g}^{ij}\lesssim 1, \qquad \sum_{1 \leq|a|\leq k}(1+r^*)^a\wpa^a {\wt{g}}_{00}\lesssim 1
\]
whenever applicable to get, for multiindices $a_i$,
\begin{align}
\sum_{|I|\leq k }|\Lie{\wt{X}}^I\bF^1| &\lesssim \sum_{|a_1|+a_t\leq k }(1+r^{*})^{s_0+|a_1|+a_t}|\wpa_{\wt{x}}^{a_1}\wpa_0^{a_t} \bF^1|\\
&\lesssim\sum_{|a_1|+|a_2|\lesssim k}(1+r^{*})^{s_0+|a_1|+|a_2|}|\wpa^{a_1}\wt{g}||\wpa_{\wt{x}}^{a_2}\bF^1|+ \sum_{|a_1|+|a_2|\lesssim k-1}(1+r^{*})^{s_0+|a_1|+|a_2|+1}|\wpa^{a_1}\wt{g}||\wpa^{a_2}\bJ[\bF^1]|. \nonumber
\end{align}
In order to bound the first term on the right we take standard $L^2-L^\infty$ bounds, and to bound the second term we decompose $\bJ[\bF^1] = \bJ[\bF] - \bJ[\bF^0]$, and note that $\bJ^0$ rapidly decays away from the origin (cf. Section \ref{ChargeCurrentBounds}), so for small $\varepsilon_g$ we may write
\begin{equation}
\slk\lVert \langle r^*\rangle^{s_0} \Lie{\wt{X}}^I\bF^1(0, \cdot)\rVert_{L^2(\mathbb{R}^3)}\lesssim\slk\lVert \langle r^*\rangle^{s_0+|I|}\wpa_{\wt{x}}^I\bF^1(0, \cdot)\rVert_{L^2(\mathbb{R}^3)} + \sum_{|I| \leq k-1}\lnorm|\langle r^*\rangle^{s_0+|I|+1}\wpa^I\bJ(0, \cdot)\rnorm_{L^2(\mathbb{R}^3)}.
\end{equation}
Next, we look at
\[
\sum_{|I| \leq k-1}\lnorm|\langle r^*\rangle^{s_0+|I|+1}\wpa^I\bJ(0, \cdot)\rnorm_{L^2(\mathbb{R}^3)}
\]
In order to do this, we write $\bJ$ as $\mathfrak{I}(\bphi\overline{\wt{D}\bphi})$. It suffices to bound
\[
\sum_{|a_1|+|a_2|\leq k-1}\lnorm|(1+r^*)^{s_0+|a_1|+|a_2|+1}|\wt{D}^{a_1}\bphi ||\wt{D}^{a_2}\wt{D}\bphi|\rnorm_2^2.
\]

In order to take care of this, we first take our time-slice Sobolev estimate
\begin{equation}
\label{TSZero}
|\langle r^*\rangle^{1+|a|}\wt{D}^a\bphi| \lesssim \sum_{|a_1|\leq 2}\langle r^*\rangle^{s_0 + |a|+|a_1|-1}|\wt{D}_{\wt{x}}^{a_1}\wt{D}^a\bphi|.
\end{equation}
which follows from the estimate $s_0> 1/2$. We finally bound
\begin{equation}
\sum_{|a|\leq k}\lVert\langle r^*\rangle^{s_0+|a|}\wt{D}\wt{D}^a\bphi\rVert.
\end{equation}
To bound this we commute time derivatives to the right and recall
\begin{equation}
D_t^2\bphi =-\wt{g}_{00}\wt{g}^{ij}\wt{D}_i\wt{D}_j\bphi.
\end{equation}
Iterating this and applying $L^\infty$ bounds whenever possible (and noting that we may always apply $L^\infty$ bounds on $g$ when they appear, as metric components are paired with two spatial derivatives of $\bphi$) gives
\begin{equation}
\sum_{|a|\leq k}\lVert\langle r^*\rangle^{s_0+|a|}\wpa\wpa^a\bphi\rVert \lesssim \sum_{|a|\leq k}\lVert\langle r^*\rangle^{s_0+|a|}|\wt{D}\wt{D}_{\wt{x}}^a\bphi|\rVert_{L^2(\mathbb{R}^3)} + \sum_{|a|\leq k-1}\lVert\langle r^*\rangle^{s_0+|a|+1}|\wt{D}\wt{D}_{\wt{x}}^a\wt{D}_0\bphi|\rVert_{L^2(\mathbb{R}^3)}
\end{equation}
We may combine everything to get the the initial decay bounds
\[
\mathcal{E}_k(0)^{1/2} \lesssim \left\lVert \mathbf{E}_{0df} \right\rVert_{H^{k_0, s_0}(\mathbb{R}^3)} + \left\lVert \mathbf{B}_0 \right\rVert_{H^{k_0, s_0}(\mathbb{R}^3)}+ \left\lVert D\bphi_0 \right\rVert_{H^{k_0, s_0}(\mathbb{R}^3)} + \left\lVert\,\,\dot{\!\!\bphi}_0 \right\rVert_{H^{k_0, s_0}(\mathbb{R}^3)}
\]
as long as the right hand side is sufficiently small.

\section{\texorpdfstring{Commutator Estimates for $\bF$}{Commutator Estimates for F}}
	Before proving Lemma \ref{Theorem:CurrentBound}, we will prove some $L^2$ and $L^\infty$ estimates which will prove useful.
\subsection{Pointwise and Energy bounds}
By Theorem \ref{LinftyPhi},
\[
\sum_{|I_1|\leq k-4}|D_{\wX}^{I_1}\bphi| \lesssim \mcE_k(T)^{1/2}\tp[-1]\tm[1/2-s]w^{-1/2}.
\]
Combining this with \eqref{est:FLorentz} gives
\begin{subequations}
\label{est:FPhiLInfty}
\begin{align}
\sum_{|I_1|+|I_2|\leq k-4}|D_{\wX}^{I_1}\bphi||(\Lie{\wX}^{I_2}\bF)_{{\wX}_1\Ls}| &\lesssim \mcE_k(T)\tp[-3/2-s]\lrangle{(r^*-t)_-}^{1/2-s}w^{-1/2}\\
\sum_{|I_1|+|I_2|\leq k-4}|D_{\wX}^{I_1}\bphi||(\Lie{\wX}^{I_2}\bF)_{{\wX}_1\Bs{i}}| &\lesssim \mcE_k(T)\tp[-1-s]\tm[-1/2]\lrangle{(r^*-t)_-}^{1/2-s}w^{-1/2}\\
\sum_{|I_1|+|I_2|\leq k-4}|D_{\wX}^{I_1}\bphi||(\Lie{\wX}^{I_2}\bF)_{{\wX}_1\uLs}| &\lesssim \mcE_k(T)\tp[-1]\tm[-1/2-s]\lrangle{(r^*-t)_-}^{1/2-s}w^{-1/2}
\end{align}
\end{subequations}
The bound $\omu < s-\frac12$ implies $w_\omu^{1/2}\leq w^{1/2}\lrangle{(r^*-t)_-}^{s-1/2}$, so for $|I|\leq k-4$
\begin{equation}
\big(\tp\tm|(\Lie{\wX}^{I}D\bphi)_{\uLs})| +\tp[s+1]\tm[1-s](|(\Lie{\wX}^{I}D\bphi)_{\Ls}| +  |(\Lie{\wX}^{I}D\bphi)_{\Bs{i}}|)\big)\langle(t-r^*)_+\rangle^{s-1/2} w^{1/2}\lesssim \mcE_k^{1/2}(T).
\end{equation}
Additionally, we have the slightly better norm
\begin{equation}
\tp[s+3/2]\tm[1/2-s]\Big|\tfrac{(\Lie{\wX}^I D(r^*\bphi))_{\Ls}}{r^*}\Big|\langle(t-r^*)_+\rangle^{s-1/2} w^{1/2}\lesssim \mcE_k^{1/2}(T).
\end{equation}
In order to show the $L^2$ bound, we first decompose $\bF = \bF^1+\bF^0$. Using \eqref{est:FDecomp} on \eqref{ChargeLInfty} and applying that to the $D_{\wX}^I\bphi$ term in $S_0[D_{\wX}^I\bphi]$ gives
\begin{subequations}
\label{OLFPhiL2}
\begin{align}
\sum_{|I_1|+|I_2| +1 \leq k}\lrnorm{\tp[s+1]\tz[1/2+{\omu}]\tm[-1]D_{\wX}^{I_1}\bphi(\Lie{\wX}^{I_2}\bF^0)_{\wX\Ls}(w')^{1/2}}_2 & \lesssim \mcE_k(T),\\
\sum_{|I_1|+|I_2|+1 \leq k}\lrnorm{\tp[s]\tz[1/2+{\omu}]D_{\wX}^{I_1}\bphi(\Lie{\wX}^{I_2}\bF^0)_{\wX\Bs{i}}(w')^{1/2}}_2 & \lesssim \mcE_k(T),\\
\sum_{|I_1|+|I_2| +1\leq k}\lrnorm{\tp[s]\tz[1/2+{\omu}]D_{\wX}^{I_1}\bphi(\Lie{\wX}^{I_2}\bF^0)_{\wX\uLs}(w')^{1/2}}_2 & \lesssim \mcE_k(T).
\end{align}
\end{subequations}
We now prove similar bounds for $\bF^1$. First,
\begin{subequations}
\begin{align}
\label{alphaF1Comm}\sum_{|I_1|+|I_2| +1\leq k}\big\lVert\tp[2s]\tm[1/2]|\alpha[\Lie{\wX}^{I_2}\bF^1]||D_{\wX}^{I_1}\bphi|w^{1/2}(w')^{1/2}\big\rVert_2&\lesssim \mathcal{E}_k(T),\\
\label{rhosigF1Comm}\sum_{|I_1|+|I_2| +1\leq k}\big\lVert\tp[2s]\tz[1/2+\omu]\tm[1/2](|\sigma[\Lie{\wX}^{I_2}\bF^1]|+|\rho[\Lie{\wX}^{I_2}\bF^1]|)|D_{\wX}^{I_1}\bphi| w^{1/2}(w')^{1/2}\big\rVert_2&\lesssim \mathcal{E}_k(T),\\
\label{BadualF1Comm}\sum_{|I_1|+|I_2| +1\leq k}\big\lVert\tp[s]\tz[1/2+\omu]\tm[s+1/2]|\ual[\Lie{\wX}^{I_2}\bF^1]||D_{\wX}^{I_1}\bphi|w^{1/2}(w')^{1/2}\big\rVert_2&\lesssim \mathcal{E}_k(T),\\
\label{GoodualF1Comm}\sum_{|I_1|+|I_2| +2\leq k}\big\lVert\tp\tz[1/2+\omu]\tm[2s-1/2]|\ual[\Lie{\wX}^{I_2}\bF^1]||D_{\wX}^{I_1}\bphi|w^{1/2}(w')^{1/2}\big\rVert_2&\lesssim \mathcal{E}_k(T),
\end{align}
\end{subequations}
To prove \eqref{alphaF1Comm} we use the $L^2(L^\infty)$ bound on $\alpha$ in Theorem \ref{LInftyF} and the spatial $L^2$ norm on $\bphi$, or the $L^\infty$ bound on $\bphi$ and the spacetime norm on $\alpha$. To bound \eqref{rhosigF1Comm} we combine the $L^\infty(L^\infty)$ bound with an $L^2(L^2)$ bound. To bound \eqref{BadualF1Comm} and \eqref{GoodualF1Comm} we take $L^\infty$ estimates on $\bphi$ terms using Lemmas \ref{Lemma:FirstLInftyPhi} and \ref{lem:PhiBadComp} respectively and bound $\ual$ in $L^2(L^2)$.
Then, the decomposition \eqref{est:FDecomp} gives
\begin{subequations}
\label{WTFPhiL2}
\begin{align}
\label{F1CommNice}\sum_{\substack{|I_1|+|I_2|+1 \leq k }}\lrnorm{\tp[2s-1]\tm[1/2]D_{\wX}^{I_1}\bphi(\Lie{\wX}^{I_2}\bF^1)_{{\wX}\Ls}w^{1/2}(w')^{1/2}}_2 &\lesssim \mcE_k(T), \\
\sum_{\substack{|I_1|+|I_2|+1 \leq k }}\lrnorm{\tp[2s-1]\tz[1/2+{\omu}]\tm[1/2]D_{\wX}^{I_1}\bphi(\Lie{\wX}^{I_2}\bF^1)_{{\wX}\Bs{i}}w^{1/2}(w')^{1/2}}_2 &\lesssim \mcE_k(T), \\
\label{F1CommBad}\sum_{\substack{|I_1|+|I_2|+1 \leq k }}\lrnorm{\tp[s-1]\tz[1/2+{\omu}]\tm[s+1/2]D_{\wX}^{I_1}\bphi(\Lie{\wX}^{I_2}\bF^1)_{{\wX}\uLs}w^{1/2}(w')^{1/2}}_2 &\lesssim \mcE_k(T), \\
\label{F1CommOK}\sum_{\substack{|I_1|+|I_2|+2 \leq k }}\lrnorm{\tm[2s-1/2]\tz[1/2+{\omu}]D_{\wX}^{I_1}\bphi(\Lie{\wX}^{I_2}\bF^1)_{{\wX}\uLs}w^{1/2}(w')^{1/2}}_2 &\lesssim \mcE_k(T).
\end{align}
\end{subequations}
 Combining  \eqref{OLFPhiL2} and \eqref{WTFPhiL2}, and recalling the spacetime norms of derivatives of $\bphi$ appearing in $\mcE_k(T)$ gives us the estimates
\begin{subequations}
\label{est:DIDPhi}
\begin{align}
\sum_{|I|\leq k}\lrnorm{\tp[s]\tz[1-s]\left|\tfrac{(\LieC[I]{\wX}(r^*\bphi))_{\Ls}}{r^*}\right|(w')^{1/2}}_2 &\lesssim \mcE_k(T)^{1/2}, \\
\sum_{|I|\leq k}\lrnorm{\tp[s]\tz[3/2-s+{\omu}]|(\LieC[I]{\wX} D\bphi)_{\Bs{i}}|(w')^{1/2}}_2 &\lesssim \mcE_k(T)^{1/2}, \\
\label{est:DIDPhi3}\sum_{|I|\leq k}\lrnorm{\tm[s]\tz[3/2-s+{\omu}]|(\LieC[I]{\wX} D\bphi)_{\uLs}|(w')^{1/2}}_2 &\lesssim \mcE_k(T)^{1/2}, \\
\label{est:DIDPhi4}\sum_{|I|\leq k-1}\lrnorm{\tm[s]\tz[1/2+{\omu}]|(\LieC[I]{\wX} D\bphi)_{\uLs}|(w')^{1/2}}_2 &\lesssim \mcE_k(T)^{1/2}.
\end{align}
\end{subequations}

\subsection{Bounding the Current Norm}

We restate the current norm \eqref{JL2} with raised indices
\begin{equation}
\label{CurrentNorm2}
\left\lVert \bJ\right\rVert_{L^2[w]} = \left\lVert \tp[s]\tz[-1/2-{\overline{\mu}}]\tm[1/2] \bJ^{\uLs} w^{1/2}_{\overline{\mu}}\right\rVert_2 + \left\lVert\tp[s]\tm[1/2]|\bJ^{\Bs{i}}|w_{\overline{\mu}}^{1/2}\right\rVert_2 + \left\lVert\tz[s-1/2-{\overline{\mu}}]\tm[s+1/2]|\bJ^{\Ls}| w_{\overline{\mu}}^{1/2}\right\rVert_2.
\end{equation}
If we write $\bJ = \bJ[\bG]$, where $\bG$ is a two-form, and $\bG^0$ and $\bG^1$ are the charged and charge-free part of $\bG$ respectively, then we define the decomposition
\[
\bJ^0[\bG] = \bJ[{\bG^0}],\qquad \bJ^1[\bG] = \bJ[{\bG^1}].
\]
\begin{Lemma}\label{Theorem:CurrentBound}
If $(\bF, \bphi)$ solve the system \eqref{MKG} on an interval $[0,T]\times\mathbb{R}^3$ and satisfy the bootstrap assumption \eqref{bootstrap}, then
\begin{equation}
\label{CFEst}
\sum_{\substack{|I| \leq k}}\lnorm \bJ[\Lie{\wX}^I \bF^1] \rnorm_{L^2[w]} \lesssim |q|+\mathcal{E}_k(T) + \varepsilon_g\mathcal{E}^{1/2}_k(T).
\end{equation}
\end{Lemma}
\begin{proof}
It suffices to show \eqref{CFEst} holds if $\bF^1$ is replaced by $\bF$ or $\bF^0$.
We recall the identity \eqref{FComms}, which we iterate and use \eqref{DT} to get
\begin{equation}
\label{JComms}
|\bJ[\Lie{\wX}^I \bF^1]^{\alpha} - (\Lie{\wX}^I\bJ^1)^\alpha| \lesssim \sum_{|I_1|+|I_2|+1 = |I|}\Big|\Lie{\wX}^{I_1}\Big((\nabla_\beta(g^{\gamma\delta}(\Lie{\wX}g)_{\gamma\delta})(\Lie{\wX}^{I_2}(\bF^{1\sharp}))^{\alpha\beta}\Big)\Big|,
\end{equation}
On the right hand side of equation \eqref{JComms} we can replace $(\Lie{\wX}g)_{\gamma\delta}$ with $\TDT{X}_{\gamma\delta}$, as the difference is the derivative of a constant.
\subsubsection{\texorpdfstring{The estimate on $\bJ$}{The estimate on J}}

 We can expand the first term out as
\begin{equation}
(\Lie{\wX}^I\bJ)^\alpha = \Lie{\wX}^I\left(g^{\alpha\beta}\mathfrak{I}\left(\bphi\overline{D_\beta\bphi}\right)\right).
\end{equation}
When derivatives fall on $g$, at each step we write $(\Lie{\wX}g)^{\alpha\beta} =( \OLie{\wX}g)^{\alpha\beta} - c_{\wX}g^{\alpha\beta}$. Using \eqref{eq:LeibnizRule} and \eqref{Comm1}
\begin{equation}
\Lie{\wX}(\phi \overline{D\psi})_{\alpha} = D_{\wX}\phi\overline{D_\alpha\psi} + \phi(\overline{\LieC{\wX}D\phi})_{\alpha}.
\end{equation}
Iterating and symmetrizing this gives us, for any vector $\wt{U}$,
\begin{align}
\label{LieCurrent}
\left|\mathfrak{I}(\Lie{\wX}^I(\bphi\overline{D\bphi}))_{\wt{U}}\right| &\lesssim \sum_{\substack{|I_1|+|I_2|\leq |I| }}\left|\mathfrak{I}\left(D_{\wX}^{I_1}\bphi\overline{D_{\wt{U}} D_{\wX}^{I_2}\bphi}+ D_{\wX}^{I_2}\bphi\overline{D_{\wt{U}} D_{\wX}^{I_1}\bphi}\right) \right| + \\
&\quad+ \sum_{\substack{|I_1|+|I_2|+|I_3|+1 \leq |I|}}\left|D_{\wX}^{I_1}\bphi\overline{D_{\wX}^{I_2}\bphi}|(\Lie{\wX}^{I_3}\bF)_{\wX {\wt{U}}}\right|\nonumber.
\end{align}
We note the identity
\begin{equation}
\mathfrak{I}\left(D_{\wX}^{I_1}\bphi\overline{D_U D_{\wX}^{I_2}\bphi}+ D_{\wX}^{I_2}\bphi\overline{D_U D_{\wX}^{I_1}\bphi}\right) = \mathfrak{I}\left(D_{\wX}^{I_1}\bphi\tfrac{\overline{D_U (r^*D_{\wX}^{I_2}\bphi)}}{r^*}+ D_{\wX}^{I_2}\bphi\tfrac{\overline{D_U (r^*D_{\wX}^{I_1}\bphi)}}{r^*}\right),
\end{equation}
as at each step the difference is the imaginary part of a real quantity.
We are ready to show energy and decay bounds, with the latter necessary to control energy terms coming from the metric. We combine \eqref{est:LinftyDPhi} and \eqref{est:FPhiLInfty} to get, for $|I|\leq k-4$,
\begin{subequations}
\begin{align}
|\Lie{\wX}^I\mathfrak{I}\left(\bphi\overline{D\bphi}\right)_{\Ls}| &\lesssim \mathcal{E}_k(T)\tp[-5/2-s]\tm[1/2-s]w^{-1}, \\
|\Lie{\wX}^I\mathfrak{I}\left(\bphi\overline{D\bphi}\right)_{\Bs{i}}|&\lesssim \mathcal{E}_k(T)\tp[-2-s]\tm[-s]w^{-1},		\\
|\Lie{\wX}^I\mathfrak{I}\left(\bphi\overline{D\bphi}\right)_{\uLs}|&\lesssim \mathcal{E}_k(T)\tp[-2]\tm[-2s]w^{-1}.
\end{align}
\end{subequations}
We recall the decomposition of Lie derivatives of the inverse metric, $\Lie{\wX}g^{-1} = \OLie{X}g^{-1} - c_Xg^{-1}$, as well as the term $\TDTD{X^I} = \OLie{X}^Ig^{-1}$. We can then take the pointwise estimate
\begin{equation}
(|\Lie{\wX}^I\bJ)^\alpha| \lesssim \sum_{|I'|\leq|I|}|g^{\alpha\beta}(\Lie{\wX}^{I'}\mathfrak{I}(\bphi\overline{D\bphi}))_\beta| + \sum_{|I_1|+|I_2|\leq|I|}|\TDTD{I_1}^{\alpha\beta}(\Lie{\wX}^{I_2}\mathfrak{I}(\bphi\overline{D\bphi}))_\beta|.
\end{equation}
To bound the second term, we first take
\begin{equation}
\sum_{\substack{|I_1| + |I_2| \leq |I| \\ |I_2|\leq k-6}}\lnorm |\TDTD{I_1}^{\alpha\beta}||\Lie{\wX}^{I_2}\mathfrak{I}(\bphi\overline{D\bphi})_\beta|\tp[s+1/2+{\overline{\mu}}]\tm[-{\overline{\mu}}]w_{\overline{\mu}}^{1/2}\rnorm_2 \lesssim \mathcal{E}_k(T)\sum_{|I|\leq k}\lnorm\tp[s+1/2+{\overline{\mu}}-2]\tm[-2s]\TDTD{\wX^I}\rnorm_2.
\end{equation}
Using $s+1/2+{\overline{\mu}}-2 < -1/2-\frac{\delta}{2}$, $-2s < -1$, we can easily bound this by $C\mathcal{E}_k(T)\varepsilon_g$ using \eqref{L21}.

When we can take $L^\infty$ estimates on the metric, we can lower indices easily, as all error terms correspond with metric terms with decay faster than $\tz[2s]$, the difference in weights between the highest and lowest weights in the norm \eqref{CurrentNorm2}.
 We first consider the inequality
\[
w_{\overline{\mu}}^{1/2}\lesssim \tm[1/2+2{\overline{\mu}}](w')^{1/2},
\]
which we combine with \eqref{CurrentNorm2} to get
\[
\lnorm \bJ \rnorm_{L^2[w]} \lesssim \lnorm\tp[s+1/2+{\overline{\mu}}]\tm[1/2+{\overline{\mu}}]|\bJ_\Ls|(w')^{1/2}\rnorm_2 + \lnorm\tz[s-1/2-{\overline{\mu}}]\tm[s+1+2{\overline{\mu}}]|\bJ_{\uLs}|(w')^{1/2}\rnorm_2 + \lnorm\tp[s]\tm[1+2{\overline{\mu}}]|\bJ_{\Bs{i}}|(w')^{1/2}\rnorm_2.
\]
We then combine the identity \eqref{LieCurrent} with the $L^2$ norm \eqref{est:DIDPhi} and the $\Lie{\wX}^I\bphi$ term in the $L^2$ norm $S_0[D_{\wX}^I\bphi](T)$ as well as the $L^\infty$ norms \eqref{est:FPhiLInfty} and \eqref{est:LinftyDPhi} to get our desired estimate. Note that for each component estimate we use the inequalities $w^{-1/2}\lesssim 1, \quad 1/2+2\mu-s < 0,\quad\tm[\delta]\leq\tp[\delta], \quad\tm[s-1/2]\leq\tp[s-1/2]$. We take the $\Bs{i}$ components as an example; other cases follow similarly. We first take, for $|I|\leq k$,
\begin{align}
\sum_{|I|\leq k}&\lnorm\tp[s]\tm[1+2{\overline{\mu}}]|(\Lie{\wX}^I\bJ)_{\Bs{i}}|(w')^{1/2}\rnorm_2 \lesssim \\
&\lesssim \sum_{\substack{|I_2|\leq k\\|I_1|\leq k-5}}\lrnorm{\tp \tm[s-1/2]D_{\wX}^{I_1}\bphi}_\infty\lrnorm{\tp[s-1]\tm[3/2-s+2{\overline{\mu}}](\Lie{\wX}^{I_2}D\bphi)_{\Bs{i}}(w')^{1/2}}_2 + \nonumber\\
&\quad+\sum_{\substack{|I_1|\leq k\\|I_2|\leq k-5}}\lrnorm{\tp[-1] \tm[1/2+2{\overline{\mu}}]D_{\wX}^{I_1}\bphi(w')^{1/2}}_2\lrnorm{\tp[1+s]\tm[1/2](\Lie{\wX}^{I_2}D\bphi)_{\Bs{i}}}_\infty \nonumber \\
&\lesssim \mcE_k(T)\nonumber,
\end{align}
using the inequalities
\[\tp[s-1]\tm[3/2-s+2{\overline{\mu}}] = \tp[s]\tz\tm[1/2+2{\overline{\mu}}-s]\leq\tp[s]\tz[1/2+{\overline{\mu}}]\]
 and 
\[\tp[-1]\tm[1/2+2{\overline{\mu}}] = \tp[s-1]\tz[1/2+{\overline{\mu}}]\tm[{\overline{\mu}}]\tp[1/2+{\overline{\mu}}-s]\leq \tp[s-1]\tz[1/2+{\overline{\mu}}].\]
Repeating this for each component gives the estimate
\begin{equation}\label{est:JBound}
\sum_{|I|\leq k}\big\lVert\Lie{\wX}^I\bJ\big\rVert_{L^2[w]} \lesssim \mathcal{E}_k(T).
\end{equation}
\subsubsection{\texorpdfstring{The estimate on $\bJ^0$}{The estimate on J^0}}
\label{ChargeCurrentBounds}We now look at $\bJ^{0a}$. It is easiest to calculate this with respect to the modified frame $\{\wpa_a\}$. Defining $\wt{g}$ and $\wt{\Gamma}$ to be the metric coefficients and Christoffel symbols of $g$ with respect to this frame, the identities
\[\wt\Gamma^b_{bc} = \tfrac12\wpa_c\ln|{\wt{g}}|, \qquad\wt \Gamma^\alpha_{\beta\gamma}{\bF^0}^{\beta\gamma}=0.\]
imply
\begin{equation}
\bJ^{0a} = \nabla_b {\bF^0}^{ab} = \tfrac12\wpa_b{(\ln |\wt{g}|)}{\bF^0}^{ab} + \wpa_b{\bF^0}^{ab}
\end{equation}
We define
\begin{equation}
{\widehat{\bF}^{ab}} = \mhat^{ac}\mhat^{bd}\bF^0_{cd}, \qquad \widehat{\bJ}^a = \wpa_b{\widehat{\bF}^{ab}}
\end{equation}
Direct computation gives
\begin{equation}
\widehat{\bJ}^a = \Big(\tfrac{q\ochi'}{4\pi r^{*2}}\Big)\Ls^a.
\end{equation}
Iterating \eqref{FrameComms}, noting that $\widehat{\bJ}$ is compactly supported in $t-r^*$, and taking the null decomposition gives the estimate
\begin{equation}
\tp[2]|(\Lie{\wX}^I\widehat{\bJ})^\uLs| + \tp(|(\Lie{\wX}^I\widehat{\bJ})^{\Bs{1}}| + |(\Lie{\wX}^I\widehat{\bJ})^{\Bs{2}}|) + |(\Lie{\wX}^I\widehat{\bJ})^\uLs| \lesssim |q|\tp[-2].
\end{equation}
It follows from direct integration (again noting that $\widehat{\bJ}$ is supported close to the light cone) that
\begin{equation}
\big\lVert(\Lie{\wX}^I\widehat{\bJ})\big\rVert_{L^2[w]}\lesssim |q|.
\end{equation}
We note generally that, for all tensorial quantities $\bf{K}$ which satisfy
\begin{equation}
|\mathbf{K}^a|\lesssim \tp[-3+\delta]\tm[-1]
\end{equation}
in all components, direct integration gives the bound
\begin{equation}
\big\lVert\mathbf{K}\big\rVert_{L^2[w]} \leq C.
\end{equation}
We may use this to bound all terms in
\begin{equation}
\bJ_{rem}^a = \wpa_b\bF^{0ab} - \widehat{\bJ}^a
\end{equation}
where the most derivatives fall on $\bF^0$ or $m_0 - \mhat$. Therefore, a straightforward calculation using \eqref{MetLI} and \eqref{MetricApprox2} gives
\[
\sum_{|I|\leq k}\big\lVert\Lie{X}^I(\bJ_{rem})^a\big\rVert_{L^2[w]}\lesssim \varepsilon_g|q| + |q|\lVert\tp[s+1/2+\omu]\tm[-\omu + \gamma]\tp[-2](|\wpa\Lie{X}^I H_0| + \tm[-1]\tp[-1+\delta]|H_0|)\rVert_2 \lesssim \varepsilon_g|q|.
\]
The bound on $\wpa_b{(\ln |\wt{g}|)}{\bF^0}^{ab}$ follows similarly (using \eqref{MetricApprox1} and fast decay of derivatives of the coefficient transformation matrix). For sufficiently small $\varepsilon_1$, we therefore have
\begin{equation}\label{est:J0Bound}
\sum_{|I|\leq k}\big\lVert\Lie{\wX}^I\bJ^0\big\rVert_{L^2[w]}\lesssim |q|
\end{equation}
\subsubsection{Bounding the Commutator terms}
To conclude the proof of Theorem \ref{Theorem:CurrentBound} it suffices to bound the right hand side of \eqref{JComms}. We recall the identity \eqref{id:DivX2}, and the consequent pointwise bound
\begin{equation}
\Lie{\wX}^I\nabla_\beta(g^{\gamma\delta}(\OLie{{\wX}}g)_{\gamma\delta}) \lesssim \sum_{|J| \leq |I|+1}\nabla_\beta(g^{\gamma\delta}(\OLie{{\wX}}^J g)_{\gamma\delta}) + \sum_{|J| + |K| \leq |I|+1}\nabla_\beta((\OLie{{\wX}}^Jg)^{\gamma\delta}(\OLie{{\wX}}^Kg)_{\gamma\delta}).
\end{equation}
Therefore, for $|I|+1 \leq k - 6$,
\begin{subequations}
\label{metricdivI}
\begin{align}
|\Ls^\alpha\Lie{\wX}^I\nabla_\alpha(g^{\gamma\delta}(\OLie{X}g)_{\gamma\delta})| &\lesssim \varepsilon_g\tp[-2+\delta], \\
|\Bs{i}^{\alpha}\Lie{\wX}^I\nabla_\alpha(g^{\gamma\delta}(\OLie{X}g)_{\gamma\delta})| &\lesssim \varepsilon_g\tp[-2+\delta], \\
|\uLs^{\alpha}\Lie{\wX}^I\nabla_\alpha(g^{\gamma\delta}(\OLie{X}g)_{\gamma\delta})| &\lesssim \varepsilon_g\tm[-1]\tp[-1+\delta].
\end{align}
\end{subequations}
If $|I|\leq k$, a repeated application of \eqref{id:DivX2} combined with the estimates \eqref{MetL2} and \eqref{MetLI} imply
\begin{subequations}
\label{metricdiv2}
\begin{align}
\label{metricdiv21}\lnorm \tp[-\delta]\Ls^\alpha\Lie{\wX}^I\nabla_\alpha(g^{\gamma\delta}(\OLie{X}g)_{\gamma\delta})(w_g')^{1/2}\rnorm_2 \lesssim \varepsilon_g, \\
\label{metricdiv22}\lnorm \tp[-\delta]\Bs{i}^\alpha\Lie{\wX}^I\nabla_\alpha(g^{\gamma\delta}(\OLie{X}g)_{\gamma\delta})(w_g')^{1/2}\rnorm_2 \lesssim \varepsilon_g,\\
\label{metricdiv23}\lnorm \tp[-1/2-\delta]\uLs^\alpha\Lie{\wX}^I\nabla_\alpha(g^{\gamma\delta}(\OLie{X}g)_{\gamma\delta})\rnorm_2 \lesssim \varepsilon_g.
\end{align}
\end{subequations}
In all cases, if Lie derivatives fall on $g^{\gamma\delta}$ we have terms which are quadratic and easy to bound using \eqref{MetL2} and \eqref{MetLI}. Otherwise, for \eqref{metricdiv21} and \eqref{metricdiv22} this follows from the spacetime bound \eqref{L22} combined with the estimates $\tm[-1/2-\mu]w < (w_g')^{1/2}$ and a dyadic decomposition in $1+t$ and for \eqref{metricdiv23} it follows from bounding the square of the $L^2$ norm by $\varepsilon_g(1+t)^{-1+\delta}$ and integrating.

Now we look at the case when Lie derivatives fall on $\bF^{1\sharp}$. Taking a null decomposition and combining the bounds \eqref{MetLI} and \eqref{est:FTLinfty}, as well as $\tm[-\omu]w_\omu \lesssim w$, gives the following:
\begin{subequations}\label{FSharpNiceLI}
\begin{align}
\sum_{|J| \leq k-6}\big|(\Lie{\wX}^{J}\bF^{1\sharp})^{\Ls\uLs}\big| &\lesssim \tp[-s-1]\tm[-1/2+\omu]w_\omu^{-1/2}\mathcal{E}_k[\bF](T)^{1/2}, \\
\sum_{|J| \leq k-6}\big|(\Lie{\wX}^{J}\bF^{1\sharp})^{\Ls\Bs{}}\big| &\lesssim \tp[-1]\tm[-1/2-s+\omu]w_\omu^{-1/2}\mathcal{E}_k[\bF](T)^{1/2}, \\
\sum_{|J| \leq k-6}\big|(\Lie{\wX}^{J}\bF^{1\sharp})^{\uLs\Bs{}}\big| &\lesssim \tp[-3/2-s]\tm[\omu]w_\omu^{-1/2}\mathcal{E}_k[\bF](T)^{1/2}, \\
\sum_{|J| \leq k-6}\big|(\Lie{\wX}^{J}\bF^{1\sharp})^{\Bs{1}\Bs{2}}\big| &\lesssim \tp[-s-1]\tm[1/2+\omu]w_\omu^{-1/2}\mathcal{E}_k[\bF](T)^{1/2},
\end{align}
\end{subequations}
We also have the $L^2$ norms
\begin{subequations}\label{FSharpNiceL2}
\begin{align}
\label{FSharpNice1}\sum_{|J| \leq k}\lnorm \tp[s]\tz[1/2+{\overline{\mu}}]\tm[-1/2-2{\overline{\mu}}](\Lie{\wX}^{J}\bF^{1\sharp})^{\Ls\uLs}w_{\overline{\mu}}^{1/2}\rnorm_2 &\lesssim \mathcal{E}_k^{1/2}, \\
\sum_{|J| \leq k}\lnorm \tm[s-1/2-2{\overline{\mu}}]\tz[1/2+{\overline{\mu}}](\Lie{\wX}^{J}\bF^{1\sharp})^{\Ls\Bs{i}}w_{\overline{\mu}}^{1/2}\rnorm_2	&\lesssim \mathcal{E}_k^{1/2}, \\
\label{FSharpNice3}\sum_{|J| \leq k}\lnorm \tp[s]\tz[1/2+{\overline{\mu}}]\tm[-1/2-2{\overline{\mu}}](\Lie{\wX}^{J}\bF^{1\sharp})^{\Bs{1}\Bs{2}}w_{\overline{\mu}}^{1/2}\rnorm_2	&\lesssim \mathcal{E}_k^{1/2}, \\
\label{FSharpNice4}\sum_{|J| \leq k}\lnorm \tp[s]\tm[-1/2-2{\overline{\mu}}](\Lie{\wX}^{J}\bF^{1\sharp})^{\uLs\Bs{i}}w_{\overline{\mu}}^{1/2}\rnorm_2	&\lesssim \mathcal{E}_k^{1/2}.
\end{align}
\end{subequations}
In order to obtain these bounds, we first decompose $g^{\alpha\beta} = m_0^{\alpha\beta}+H_1^{\alpha\beta}$. Then, if more derivatives fall on $\bF^{1\sharp}$, we use the estimates \eqref{MetLI} and \eqref{MetricApprox2} and a null decomposition, noting the intermediate bound
\begin{equation}
|\bG^{\uLs\Bs{i}}|\lesssim |\al[\bG]| + \varepsilon_g\tp[-1+\delta](|\rho[\bG]|+|\sigma[\bG]|) + \varepsilon_g\tz[\gamma]\tp[-1+\delta]|\ual[\bG]|.
\end{equation}
When more derivatives fall on $H_1$, then \eqref{FSharpNice1}-\eqref{FSharpNice3} follow from \eqref{est:FTLinfty} and \eqref{L21}, with no null decomposition necessary (using the bounds $s-\tfrac12-\overline\mu - 1 < -\tfrac12-\delta$, squaring, and integrating in time). For the estimate \eqref{FSharpNice4}, we use a null decomposition as well as the bound \eqref{L23} to bound terms containing $|\ual||H_0^{\uLs\uLs}|$.

In order to bound the right hand side of \eqref{JComms}, we take a null decomposition and apply the estimates \eqref{metricdivI}, \eqref{metricdiv2}, \eqref{FSharpNiceLI}, and \eqref{FSharpNiceL2}. Combining this with \eqref{est:JBound} and \eqref{est:J0Bound} completes the proof of Theorem \ref{Theorem:CurrentBound}.

\end{proof}

\section{\texorpdfstring{Commutator Estimates for $\bphi$}{Commutator Estimates for Phi}}
	Here we will bound the remaining term in Theorem \ref{thm:MainEnergy}.
\begin{Lemma}\label{lem:PhiCommutators}
For a pair $(\bphi, \bF)$ satisfying \eqref{MKG}, as well as the restrictions of Theorem \ref{Main} and the bootstrap assumption \eqref{bootstrap}, then for sufficiently small $\varepsilon_b$, and $\varepsilon_g \leq 1$, the estimate
\begin{align}\label{est:BoxDPhi}
\sum_{\substack{|I| \leq k}}\left\lVert\left(\Box^{\mathbb{C}}_gD_{\wX}^I\bphi\right) \tp[s]\tm[1/2]w_{\overline{\mu}}^{1/2}\right\rVert_2  \lesssim \mathcal{E}_k(T) + \varepsilon_g\mathcal{E}_k(T)^{1/2}.
\end{align}
\end{Lemma}
For ease of notation we say
\[
\slk \,T[\wX^I, \wX^{I_1},\hdots ,\wX^{I_m},\,\wY_1,\hdots ,\wY_n] = \sum_{|I|+|I_1|+\hdots+|I_m|+n\leq k}\,T[\wX^{I},\wX^{I_1},\hdots \wX^{I_m},\,\wY_1,\hdots,\wY_n],
\]
where all $\wX, \wY \in \{\wpa_\alpha, \Os{\alpha\beta}, \Ss\}$, and $I, \{I_k\}$ are multiindices of vector fields in this set. Since $(\bF, \bphi)$ solves \eqref{MKG}, it suffices to prove

\begin{equation}
\left\lVert\tp[s]\tm[1/2]\left(\left[\Box^{\mathbb{C}}_g, D_{\wX}^I\right]\bphi\right) w_{\overline{\mu}}^{1/2}\right\rVert_2 \lesssim \mathcal{E}_k(T) + \varepsilon_g\mathcal{E}_k(T)^{1/2}.
\end{equation}
We may iterate the identity \eqref{CommBox} to obtain
\begin{equation}
\label{ThreeParts}
\left|\left[\Box_g^\mathbb{C}, D_{\wX}^I\right]\bphi\right| \lesssim A_{|I|}+B_{|I|}+C_{|I|},
\end{equation}
where
\begin{subequations}
\begin{align}
\label{AjDef}A_j &:= \sum_{|I_1| + |I_2| \leq j-1}\left|D_{\wX}^{I_1}\left(D^\alpha{D_{\wX}^{I_2}}\bphi\nabla_\alpha(\nabla\cdot \wY)\right)\right|\\
\label{BjDef}B_j &:= \sum_{|I_1| + |I_2| \leq j-1} \left|D_{\wX}^{I_1}\left(D_\alpha\left(\DTD{\wY}^{\alpha\beta}D_\beta{D_{\wX}^{I_2}}\bphi\right) \right)\right|\\
\label{CjDef}C_j &:= \sum_{|I_1| + |I_2| \leq j-1}\left|D_{\wX}^{I_1}\left( i(\nabla^\alpha F_{\wY\alpha}{D_{\wX}^{I_2}}\bphi+ 2F_{\wY\alpha}D^\alpha{D_{\wX}^{I_2}}\bphi)\right)\right|
\end{align}
\end{subequations}
We briefly explain how we bound each term in our weighted $L^2$ norm, as this will be the focus of the remainder of this section. First, $A_{|I|}$ is quadratic. In order to bound $B_{|I|}$, we subtract off a lower-order term (which we can bound by $A_{|I-1|}, B_{|I-1|}, C_{|I-1|}$), then bound the remainder using the improved decay estimates coming from the modified coordinates. Finally, to bound $C_{|I|}$ we must use a special structure of the MKG system which shows up also in the Minkowski case. The result follows from induction, using $A_0 = B_0 = C_0 = 0$.

\subsection{Bounding $A_{|I|}$}
In order to bound $A_{|I|}$, we iterate \eqref{Comm2} to get:
\begin{equation}
\slj\left|\Lie{{\wX}_1}^{I_1}\left(D^\alpha{D_{{\wX}_2}^{I_2}}\bphi\nabla_\alpha(\nabla\cdot {\wY})\right)\right| \lesssim A^1_{j} + A^2_{j} + A^3_{j},
\end{equation}
where
\begin{subequations}
\begin{align}
A^1_{j} &= \slj \left|g^{\alpha\beta}D_{\alpha}D_{{\wX}_1}^{I_1}\bphi\nabla_\beta\Lie{{\wX}_2}^{I_2}(g^{\gamma\delta}(\OLie{{\wY}}g)_{\gamma\delta}) \right|\\
A^2_{j} &= \slj \left|\TDTD{{\wX}_1^{I_1}}^{\alpha\beta}D_\alpha D_{{\wX}_2}^{I_2}\bphi \nabla_{\beta}\Lie{{\wX}_3}^{I_3}(g^{\gamma\delta}(\OLie{{\wY}}g)_{\gamma\delta}) \right|\\
A^3_{j} &=\slj \left|\DTD{{\wX}_1^{I_1}}^{\alpha\beta} (\Lie{{\wX}_2}^{I_2}\bF)_{{\wY}_1\alpha} D_{{\wX}_3}^{I_3}\bphi \nabla_{\beta}\Lie{{\wX}_4}^{I_4}(g^{\gamma\delta}(\OLie{{\wY}_2}g)_{\gamma\delta})\right|.
\end{align}
\end{subequations}
$A_j^1$ consists of terms where the derivative commutes through $D^\alpha$, and $A_j^2$ and $A_j^3$ consist of the commutator terms. To bound $A^1_j$, we take the null decomposition
\[
\slj \left|g^{\alpha\beta}D_{\alpha}D_{{\wX}}^{I_1}\bphi\nabla_\beta\Lie{{\wX}}^{I_2}(g^{\gamma\delta}(\OLie{Y}g)_{\gamma\delta}) \right| \lesssim \slj|DD_{{\wX}}^{I_1}\bphi||\overline\partial \OLie{{\wX}}^{I_2}g| + |\overline{D}D_{{\wX}}^{I_1}\bphi||\partial  \OLie{{\wX}}^{I_2}g| + |g^{\uLs\uLs}||DD_{{\wX}}^{I_1}\bphi||\partial  \OLie{{\wX}}^{I_2}g|,
\]
We can take our first weighted estimate. If $|I_1| \leq k-6$, we use \eqref{est:LinftyDPhi} and \eqref{RMetric} to get
\begin{equation}\label{A1jEstimate}
\slkI{|I_1| \leq k-6}\!\!\!\! \lnorm \tp[s]\tm[1/2]g^{\alpha\beta}D_{\alpha}D_{{\wX}_1}^{I_1}\phi\nabla_\beta\Lie{{\wX}_2}^{I_2}(g^{\gamma\delta}(\OLie{Y}g)_{\gamma\delta}) w_{\overline{\mu}}^{1/2}\rnorm_2 \lesssim \mathcal{E}_k^{1/2}\slk\lnorm (\tp[s-1]\tm[-s]|\overline\partial \Lie{{\wX}_2}^{I_2}g| + \tp[-1]|\partial  \Lie{{\wX}_2}^{I_2}g|)(\tfrac{w_{\overline{\mu}}}{w})^{1/2}\rnorm_2,
\end{equation}
which is bounded by $\varepsilon_g\mathcal{E}_k^{1/2}$. If $|I_2| \leq k-6$, this can be bounded by
\begin{equation}\label{est:CommPhiL2DPhi}
\varepsilon_g \slk\lnorm(\tp[s-1+\delta]\tm[-1/2]|\overline{D}D_{{\wX}}^{I_1}\bphi| + \tp[s-3/2-\gamma/2]\tm[1/2]|{D}D_{{\wX}}^{I_1}\bphi|)w_{\overline{\mu}}^{1/2}\rnorm_2
\end{equation}
Then, $\tp[-1+\delta]<\tp[-1/2-2\delta-\overline\mu]$, $\tp[s-3/2-\gamma/2] < \tp[-1/2-2\delta-\overline\mu]$, so by \eqref{STEA}, \eqref{STEC}, this is finite, so
\begin{equation}
\slk\lnorm\tp[s]\tm[1/2]A_1^jw_\omu^{1/2}\rnorm_2 \lesssim\varepsilon_g\mathcal{E}_k^{1/2}
\end{equation}

We consider $A^2_j$. Proposition \ref{prop:ModifiedDTBounds} allows us to replace $\TDTD{\wX^I}$ with $\OLie{\wX}^I H_1$ and reduce to the case $|I_1| > k-6$. We can take our worst $L^\infty$ estimates on the $D\bphi$ and $\partial\nabla\cdot Y$ terms. The remaining quantities in the second term are therefore bounded by
\begin{equation}\label{A2jFirstEst}
\slk\mathcal{E}_k^{1/2}\varepsilon_g\lnorm\tp[s-2+\delta]\tm[-1-s]|\OLie{\wX}^{I_1}H_1|w_{\overline{\mu}}^{1/2}\rnorm_2,
\end{equation}
which is easily bounded by $\mathcal{E}_k^{1/2}\varepsilon_g$. Therefore,
\begin{equation}\label{A2jBound}
\slk\lnorm\tp[s]\tm[1/2]A_2^jw_{\overline{\mu}}^{1/2}\rnorm_2 \lesssim\varepsilon_g\mathcal{E}_k^{1/2}
\end{equation}
We now look at $A_j^3$. If $k-6$ vector fields or fewer appear in 
\[
(\Lie{{\wX}_2}^{I_2}\bF)_{\wY_1\alpha} D_{{\wX}_3}^{I_3}\bphi
\]
For these terms, we split $\DTD{{\wX}_1^{I_1}}$ into $g$ and $\TDTD{{\wX}_1^{I_1}}$. and bound the resultant terms the same way as \eqref{A1jEstimate} and \eqref{A2jFirstEst}respectively.

Otherwise, we use a null decomposition combined with the estimates \eqref{est:DIDPhi} and \eqref{MetLI}.
\begin{equation}\label{A3jBound}
\slk\lnorm\tp[s]\tm[1/2]A_3^jw_{\overline{\mu}}^{1/2}\rnorm_2 \lesssim\varepsilon_g\mathcal{E}_k^{1/2}
\end{equation}
We can state our first subresult:
\begin{equation}
\slk\lnorm\tp[s]\tm[1/2]\Lie{{\wX}_1}^{I_1}\left(-D^\alpha{D_{{\wX}_2}^{I_2}}\bphi\nabla_\alpha(\nabla\cdot Y)\right)w_{\overline{\mu}}^{1/2}\rnorm_2 \lesssim \varepsilon_g\mathcal{E}_k^{1/2}
\end{equation}
\subsection{Bounding $B_{|I|}$}
\textbf{The Second Term.} 
We now look at \eqref{BjDef}, which we may reduce to bounding
\begin{equation}
\label{SecondTerm}
\slk\lnorm\tp[s]\tm[1/2]\LieC[I_1]{\wX}\left(D_\alpha\left(\TDTD{\wY}^{\alpha\beta}D_\beta D_{\wX}^{I_2}\bphi\right)\right)w_{\overline{\mu}}^{1/2}\rnorm_2,
\end{equation}
as the difference features a scalar multiple of the metric, and thus we may bound it using $A_{|I|-1}, B_{|I|-1}, C_{|I|-1}$.
We will commute $\LieC[I_1]{\wX}$ through $D_\alpha$, noting the commutator
\begin{equation}\label{PhiCommComm}
\left|[\LieC[I_1]{\wX}, D_\alpha]\left(\TDTD{Y}^{\alpha\beta}D_\beta D_{\wX}^{I_2}\bphi\right)\right| \lesssim \sum_{|I_3| + 1 + |I_4| = |I_1|} \left|\LieC[I_3]{\wX}[\LieC{\wY_1}, D_\alpha]\LieC[I_4]{\wX}\left(\TDTD{X_2}^{\alpha\beta}D_\beta D_{\wX}^{I_2}\bphi\right)\right|.
\end{equation}
We first bound
\[
\slkI{|I_1|\geq 1}\lnorm \tp[s]\tm[1/2]D_\alpha\left(\TDTD{X^{I_1}}^{\alpha\beta} \LieC[I_2]{\wX}(DD_{\wX}^{I_3}\bphi)_\beta\right)w_{\overline{\mu}}^{1/2}\rnorm_2.
\]
We will prove this for $r^* > (t+1)/2$, as the interior case is easier. Taking the null decomposition in $\beta$ and applying the product rule gives
\begin{subequations}\label{DTDTD}
\begin{align}
\slk&\lVert \tp[s]\tm[1/2]D_\alpha\left(\TDTD{X^{I_1}}^{\alpha\beta} \LieC[I_2]{\wX}(DD_{\wX}^{I_3}\bphi)_\beta\right)w_{\overline{\mu}}^{1/2}\rVert_2 \lesssim \nonumber\\ 
&\label{term:DTDTD1}\lesssim \slk\lVert \tp[s]\tm[1/2]\nabla_\alpha\left(\TDTD{\wX^{I_1}}^{\alpha\uLs}\right) \LieC[I_2]{\wX}(DD_{\wX}^{I_3}\bphi)_\uLs w_{\overline{\mu}}^{1/2}\rVert_2 + \\
&\label{term:DTDTD2}\qquad+ \slk\lVert \tp[s]\tm[1/2]\nabla_\alpha\left(\TDTD{\wX^{I_1}}^{\alpha\mcT}\right) \LieC[I_2]{\wX}(DD_{\wX}^{I_3}\bphi)_\mcT w_{\overline{\mu}}^{1/2}\rVert_2 + \\
&\label{term:DTDTD3}\qquad+ \slk\lVert \tp[s]\tm[1/2]\TDTD{\wX^{I_1}}^{\alpha\uLs} D_\alpha\left(\LieC[I_2]{\wX}(DD_{\wX}^{I_3}\bphi)_\uLs\right)w_{\overline{\mu}}^{1/2}\rVert_2 + \\
&\label{term:DTDTD4}\qquad+ \slk\lVert \tp[s]\tm[1/2]\TDTD{\wX^{I_1}}^{\alpha\mcT} D_\alpha\left(\LieC[I_2]{\wX}(DD_{\wX}^{I_3}\bphi)_\mcT\right)w_{\overline{\mu}}^{1/2}\rVert_2
\end{align}
\end{subequations}
We first bound \eqref{term:DTDTD1} and \eqref{term:DTDTD2}. First, by \eqref{est:FPhiLInfty}, and recalling \eqref{DTLIm},
\begin{subequations}
\begin{align}
\slkI{|I_2|+|I_3|\leq k-6}\left\lVert \tp[s]\tm[1/2] \nabla_\alpha(\OLie{\wX}^{I_1}H_1)^{\alpha\mcT}\LieC[I_2]{\wX}(DD_{\wX}^{I_3}\bphi)_\mcT w_{\overline{\mu}}^{1/2}\right\rVert_2 &\lesssim \mathcal{E}_k(T)^{1/2}\slk\left\lVert\tp[-1]\tm[{\overline{\mu}}]\nabla_\alpha(\OLie{\wX}^{I_1}H_1)^{\alpha\mcT}\right\rVert_2, \\
\slkI{|I_2|+|I_3|\leq k-6}\left\lVert \tp[s]\tm[1/2] \nabla_\alpha(\OLie{\wX}^{I_1}H_1)^{\alpha\uLs}\LieC[I_2]{\wX}(DD_{\wX}^{I_3}\bphi)_\uLs w_{\overline{\mu}}^{1/2}\right\rVert_2 &\lesssim \mathcal{E}_k(T)^{1/2}\slk\left\lVert\tp[s-1]\tm[\overline{\mu}-s]\nabla_\alpha(\OLie{\wX}^{I_1}H_1)^{\alpha\uLs}\right\rVert_2.
\end{align}
\end{subequations}
All of these are bounded by $\varepsilon_g\mathcal{E}_k(T)^{1/2}$. Next, \eqref{MetLI} gives
\begin{align*}
\sum_{\substack{\leq k \\1\leq|I_1| \leq k-6}}\left\lVert \tp[s]\tm[1/2+{\overline{\mu}}] \nabla_\alpha(\OLie{\wX}^{I_1}H_1)^{\alpha\mcT}\LieC[I_2]{\wX}(DD_{\wX}^{I_3}\bphi)_\mcT w^{1/2}\right\rVert_2 &\lesssim \varepsilon_g\slk\left\lVert\tp[s-1+\delta]\tm[-1/2+{\overline{\mu}}]\LieC[I_2]{\wX}(DD_{\wX}^{I_3}\bphi)_\mcT w^{1/2}\right\rVert_2, \\
\sum_{\substack{\leq k \\1\leq|I_1| \leq k-6}}\left\lVert \tp[s]\tm[1/2+{\overline{\mu}}] \nabla_\alpha(\OLie{\wX}^{I_1}H_1)^{\alpha\uLs}\LieC[I_2]{\wX}(DD_{\wX}^{I_3}\bphi)_\uLs w^{1/2}\right\rVert_2 &\lesssim \varepsilon_g\slk\left\lVert\tz[s]\tm[s-3/2+2{\overline{\mu}}]\LieC[I_2]{\wX}(DD_{\wX}^{I_3}\bphi)_\uLs w^{1/2}\right\rVert_2.
\end{align*}
Both of these are bounded by $\varepsilon_g\mathcal{E}_k(T)^{1/2}$, using \eqref{est:DIDPhi} (including \eqref{est:DIDPhi4}). A similar argument using \eqref{DTLIm} allows us to replace $(\OLie{\wX}^{I_1}H_1)$ with $\TDTD{\wX^{I_1}}$.

Now we bound \eqref{term:DTDTD3} and \eqref{term:DTDTD4}. First, we consider the case where the derivative falls on $\bF$. We recall the decompositions
\[
|\Ls(\psi)| + \sum_i|\Bs{i}(\psi)| \lesssim \tp[-1]\sum_{|I|=1}|\wX^I\psi|,\qquad |\uLs(\psi)| \lesssim \tm[-1]\sum_{|I|=1}|\wX^I\psi|.
\]
We combine this with Lemma \ref{NullComms} to get
\begin{subequations}
\begin{align}
\sum_{|I|\leq k-2}|\mcU((\Lie{\wX}^I \bF)_{\wX\mcU})| &\lesssim \sum_{|J|\leq k-1}\tz[-1]|\Lie{\wX}^J\bF| \\
\sum_{|I|\leq k-2}|\mcT((\Lie{\wX}^I \bF)_{\wX\mcU})| &\lesssim \sum_{|J|\leq k-1}|\Lie{\wX}^J\bF| \\
\sum_{|I|\leq k-2}\mcU((\Lie{\wX}^I \bF)_{\wX\mcT}) &\lesssim \sum_{|J|\leq k-1}|\Lie{\wX}^J\bF| + \tz[-1](|\alpha[\Lie{\wX}^J\bF]| + |\rho[\Lie{\wX}^J\bF]|+|\sigma[\Lie{\wX}^J\bF]|)
\end{align}
\end{subequations}
When the derivative falls on derivatives of $\bphi$, we use the estimates
\begin{subequations}
\label{TwoDers}
\begin{align}
\sum_{|I|\leq k-1}|D_{\Ls} D_{\T}D_{\wX}^I\bphi| &\lesssim \sum_{|J|\leq k}\left|\tp[-1]D_{\Ls} D_{\wX}^J\bphi\right| + \left|\tp[-2]D_{\wX}^J\bphi\right| \\
\sum_{|I|\leq k-1}|D_{\Bs{j}}D_{\T}D_{\wX}^I\bphi| &\lesssim \sum_{|J|\leq k}\left|\tp[-1]D_{\Bs{j}}D_{\wX}^J\bphi\right| + \left|\tp[-2]D_{\wX}^J{\bphi}\right|\\
\sum_{|I|\leq k-1}|D_{\uLs}D_{\T}D_{\wX}^I\bphi| &\lesssim \sum_{|J|\leq k}\left|\tp[-1]D_{\uLs}D_{\wX}^J\bphi\right| + \left|\tp[-2]D_{\wX}^J\bphi\right|\\
\sum_{|I|\leq k-1}|D_{\uLs}D_{\uLs}D_{\wX}^I\bphi| &\lesssim \sum_{|J|\leq k}\left|\tm[-1]D_{\uLs}D_{\wX}^J\bphi\right|
\end{align}
\end{subequations}
For all other derivatives we use the identity
\[
D_{U}D_{\uLs}\psi = D_{\uLs}D_{U}\psi + D_{[U, \uLs]}\psi + i\bF_{U\uLs}\psi.
\]
For $U\in\mcT$, $[U, \uLs]$ is either 0 or $\frac{1}{r^*}\Bs{j}$. Combining these gives
\begin{subequations}
\begin{equation}
\sum_{\substack{|I|\leq k-1 \\ T\in\mcT}}|D_TD_\uLs D_{\wX}^I\bphi|\lesssim \sum_{|J|\leq k}\tp[-1]|DD_{\wX}^J\bphi|+\tp[-2]|D_{\wX}^J\bphi| + \sum_{|I|\leq k-1 }|\bF||D_{\wX}^{I}\bphi|.
\end{equation}
\end{subequations}
We now put everything together. Again, we may replace $\TDTD{\wX^J}$ with $\OLie{\wX}^J H$ and bound the difference using the estimates \eqref{DTLIm}. In the cases where we can bound $|\OLie{\wX}^J H|$ in $L^\infty$ our estimates follow from \eqref{MetLI} and \eqref{est:DIDPhi} (with the worse weights of \label{est:DIDPhi3}). Otherwise the bounds \eqref{est:FPhiLInfty} and \eqref{est:LinftyDPhi} imply
\begin{align}
\label{HSpacetimeNeed}\slkI{|I_2|+|I_3|\leq k-6}&\lnorm \tp[s]\tm[1/2]\left((\OLie{\wX}^IH_1)^{\alpha U} D_\alpha\LieC[I_2]{\wX}(DD_{\wX}^{I_3}\phi)_U\right)w_{\overline{\mu}}^{1/2}\rnorm_2 \lesssim \\
&\lesssim \sum_{|I|\leq k}\lnorm \tp[s-2]\tm[-s]|(\OLie{\wX}^IH_1)|\rnorm_2 + \sum_{|I|\leq k}\lnorm \tp[s-1]\tm[-s-1]|(\OLie{\wX}^IH_1)^{\uLs\uLs}|\rnorm_2. \nonumber
\end{align}
These can be bounded by \eqref{L21} and \eqref{L22} without issue. Therefore,
\begin{equation}
\slkI{|I_1|\geq 1}\lnorm \tp[s]\tm[1/2]D_\alpha\left(\TDTD{X^{I_1}}^{\alpha\beta} \LieC[I_2]{\wX}(DD_{\wX}^{I_3}\bphi)_\beta\right)w_{\overline{\mu}}^{1/2}\rnorm_2 \lesssim \varepsilon_g\mcE_k^{1/2}.
\end{equation}
Finally, we bound the right hand side of \eqref{PhiCommComm}. We have four types of terms we need to bound here:
\begin{subequations}
\begin{align}
\sum_{|I_3| + 1 + |I_4| = |I_1|} &\Big\lVert\tp[s]\tm[1/2]\LieC[I_3]{\wX}[\LieC{\wY_1}, D_\alpha]\LieC[I_4]{\wX}\Big(\TDTD{{\wX}_2}^{\alpha\beta}D_\beta D_{\wX}^{I_2}\bphi\Big)w_\omu^{1/2}\Big\rVert_2 \lesssim\nonumber \\
\label{FirstTermSpecial}
&\lesssim\big\lVert\tp[s]\tm[1/2](\Lie{{\wX}}^{I_1}\bF)_{\wY_1\alpha}(\Lie{{\wX}}^{I_2}\bF)_{\wY_2\beta}\TDTD{{\wX}^{I_3}}^{\alpha\beta}D_{{\wX}}^{I_4}\bphi w_{\omu}^{1/2}\big\rVert_2 \\
\label{SecondTermSpecial}
&\lesssim\big\lVert\tp[s]\tm[1/2](\Lie{{\wX}}^{I_1}\bF)_{\wY_1\alpha}\TDTD{{\wX}^{I_3}}^{\alpha\beta}D_\beta D_{{\wX}}^{I_4}\bphi w_{\omu}^{1/2}\big\rVert_2 \\
\label{ThirdTermSpecial}&\lesssim\big\lVert\tp[s]\tm[1/2]\nabla_\alpha({\wX}^I(\nabla\cdot \wY) )(\Lie{{\wX}}^{I_2}\bF)_{\wY_2\beta}\TDTD{{\wX}^{I_3}}^{\alpha\beta}D_{{\wX}}^{I_4}\bphi w_{\omu}^{1/2}\big\rVert_2 \\
\label{FourthTermSpecial}&\lesssim\big\lVert\tp[s]\tm[1/2]\nabla_\alpha({\wX}^I(\nabla\cdot \wY))\TDTD{{\wX}^{I_3}}^{\alpha\beta}D_\beta D_{{\wX}}^{I_4}\bphi w_{\omu}^{1/2}\big\rVert_2
\end{align}
\end{subequations}
We first deal with \eqref{FirstTermSpecial}. We prove this with $\TDTD{\wX^{I_3}}$ replaced by $\OLie{\wX^{I_3}}$, and use \eqref{DTLIm} to bound the remainder. Additionally, we decompose $\bF = \bF^0 + \bF^1$. Then, if $\bF^0$ appears twice and if $|I_3| \leq k-6$, a null decomposition and \eqref{MetLI} give
\[
\sum_{\substack{\leq k \\ |I_3|\leq k-6}}\big\lVert\tp[s]\tm[1/2](\Lie{{\wX}}^{I_1}\bF)_{\wY_1\alpha}(\Lie{{\wX}}^{I_2}\bF)_{\wY_2\beta}\TDTD{{\wX}^{I_3}}^{\alpha\beta}D_{{\wX}}^{I_4}\bphi w_{\omu}^{1/2}\big\rVert_2 \lesssim \slk\varepsilon_g|q|^2\lVert\tp[s-1]\tz[\gamma-\delta]\tm[-3/2+\delta]D_{\wX}^{I_4}w^{1/2}\rVert_2\lesssim \varepsilon_g\mcE_k^{3/2}. 
\]
Otherwise, the bound is similar to \eqref{HSpacetimeNeed}. For all other terms, we apply Lemma \ref{lem:PhiBadComp}, noting $|I_4|\leq k-2$. First, for $I_1, I_2 \leq k-6$ the bound $\omu < \tfrac12 - s$ along with $L^\infty$ estimates gives
\begin{equation}
\big\lVert\tp[s]\tm[1/2](\Lie{{\wX}}^{I_1}\bF)_{\wY_1\alpha}(\Lie{{\wX}}^{I_2}\bF)_{\wY_2\beta}\TDTD{{\wX}^{I_3}}^{\alpha\beta}D_{{\wX}}^{I_4}\bphi w_{\omu}^{1/2}\big\rVert_2\lesssim |q|^2(\lVert\tp[s-1]\tm[-1-s]|\OLie{\wX}^{I_3}H_1|^{\uLs\uLs}\rVert_2 + \lVert\tp[-1]\tm[-1]|\OLie{\wX}^{I_3}H_1|\rVert_2).
\end{equation}
We may bound these using \eqref{L22} and \eqref{L23}. Otherwise, we take a null decomposition and use our $L^\infty$ bounds to reduce this to an $L^2$ norm on $\bF^1$.

Bounds on the other terms are much simpler. We can bound \eqref{SecondTermSpecial} in the same way as \eqref{term:DTDTD3} and \eqref{term:DTDTD4}. In order to bound \eqref{ThirdTermSpecial} and \eqref{FourthTermSpecial} we apply \eqref{A3jBound} and \eqref{A2jBound} respectively.

Therefore,
\begin{equation}
\sum_{|I_3| + 1 + |I_4| = |I_1|} \Big\lVert\tp[s]\tm[1/2]\LieC[I_3]{\wX}[\LieC{\wY_1}, D_\alpha]\LieC[I_4]{\wX}\Big(\TDTD{{\wX}_2}^{\alpha\beta}D_\beta D_{\wX}^{I_2}\bphi\Big)w_\omu^{1/2}\Big\rVert_2\lesssim \varepsilon_g\mcE_k^{1/2}.
\end{equation}
\subsection{Bounding $C_{|I|}$}

\textbf{The Final Term:} 
We now attempt to bound \eqref{CjDef}. The primary concern here is that, even in the Minkowski spacetime, we have terms like $(\Lie{X_1}^{I_1}\bF)_{Y\uLs}D_\Ls D_{X_2}^{I_2}\bphi$. This does not decay well, so we must utilize additional cancellation noted in \cite{S92}. The commutator terms take the form:
\begin{equation}
\slk-i \LieC[I_1]{\wX}\left(\nabla_\beta(g^{\alpha\beta}( \bF_{\wY_1\alpha})) D_{\wX}^{I_2}\bphi + 2\bF_{\wY_1\alpha}g^{\alpha\beta}D_\beta(  D_{\wX}^{I_2}\bphi)\right).
\end{equation}

We can bound these terms using the pointwise estimate
\begin{align}
\slk&\left|\LieC[I_1]{\wX}\left(\nabla_\beta(g^{\alpha\beta}( \bF_{\wY_1\alpha})) D_{\wX}^{I_2}\bphi + 2\bF_{\wY_1\alpha}D_\beta(g^{\alpha\beta}  D_{\wX}^{I_2}\bphi)\right)\right|  \lesssim\\
&\lesssim\slk \left| [\Lie{\wX}^{I_1}, \nabla_\beta](g^{\alpha\beta}\bF_{\wY_1\alpha})D_{\wX}^{I_2}\bphi\right| + \slk\left|(\Lie{\wX}^{I_1}(i_{\wY_1} \bF))^\beta [\LieC[I_2]{\wX}, D_\beta]D_{\wX}^{I_3}\bphi \right| + \nonumber\\
&\quad+\slk\left| \nabla_\beta\left(\Lie{\wX}^{I_1}(g^{-1}i_{\wY_1}\bF)\right)^\beta D_{\wX}^{I_2}\bphi + 2(\Lie{\wX}^{I_1}(g^{-1}i_{\wY_1}\bF))^\beta D_\beta D_{\wX}^{I_2}\bphi\right|, \nonumber\\
&\lesssim C^1_{k} + C^2_k + C^3_k.\nonumber
\end{align}

We look at the three terms on the right. The identity \eqref{Comm3D} allows us to bound $C_k^1$ using \eqref{A3jBound}.

To bound $C_k^2$ we first use \eqref{Comm1} and \eqref{FirstTermSpecial} to reduce the problem to bounding
\begin{equation}
\slk\big\lVert \tp[s]\tm[1/2]g^{\alpha\beta}\Lie{\wX}^{I_1}\bF_{\wY_1\alpha}\Lie{\wX}^{I_2}\bF_{\wY_2\beta}D_{\wX}^{I_3}\bphi w_\omu^{1/2}\big\rVert_2.
\end{equation}
As usual, we split $\bF = \bF^0 + \bF^1$. A null decomposition gives the bounds
\begin{subequations}
\begin{align}
\big|\tp[s]\tm[1/2](g-\mhat)^{\alpha\beta}\Lie{\wX}^{I_1}\bF^0_{\wY_1\alpha}\Lie{\wX}^{I_2}\bF^0_{\wY_2\beta}\big| &\lesssim \varepsilon_g|q|^2\tp[s-1]\tz[\gamma-\delta]\tm[-3/2+\delta],\\
\big|\tp[s]\tm[1/2]\mhat^{\alpha\beta}\Lie{\wX}^{I_1}\bF^0_{\wY_1\alpha}\Lie{\wX}^{I_2}\bF^0_{\wY_2\beta}\big| &\lesssim \varepsilon_g|q|^2\tp[s-1]\tz\tm[-1/2].
\end{align}
\end{subequations}
In the support of $\bF^0$ we have the bound $\tm[-1/2]w_\omu \lesssim w'$, which, combined with the inequality and \eqref{bootstrap} gives
\begin{equation}
\slk\big\lVert \tp[s]\tm[1/2]g^{\alpha\beta}\Lie{\wX}^{I_1}\bF^0_{\wY_1\alpha}\Lie{\wX}^{I_2}\bF^0_{\wY_2\beta}D_{\wX}^{I_3}\bphi w_\omu^{1/2}\big\rVert_2\lesssim \varepsilon_g\mcE_k^{1/2}.
\end{equation}
For terms containing $\bF^1$, we note that $|I_1|+|I_2|+|I_3|\leq k-2$. It therefore suffices to bound
\begin{subequations}
\begin{align}
\label{CommgFF1}\slk&\big\lVert \tp[s-1]\tm[1-s]g^{\alpha\beta}|\Lie{\wX}^{I_1}\bF^1_{\wY_1\alpha}||\Lie{\wX}^{I_2}\bF^0_{\wY_2\beta}|\big\rVert_2,\\
\label{CommgFF2}\slk&\big\lVert \tp[s-1]\tm[1-s+\omu]g^{\alpha\beta}|\Lie{\wX}^{I_1}\bF^1_{\wY_1\alpha}||\Lie{\wX}^{I_2}\bF^1_{\wY_2\beta}|\langle \big\rVert_2.
\end{align}
\end{subequations}
 For each of these we take a null decomposition and \eqref{est:FLorentz}. To bound \eqref{CommgFF1} we apply the bounds \eqref{MetLI} and bound the $\bF$ terms using \eqref{ChargeLInfty}. In order to bound \eqref{CommgFF2} we follow the same process, bounding $|\Lie{\wX}^{I_1}\bF^1_{\wY_1\uLs}|$ and $|\Lie{\wX}^{I_2}\bF^1_{\wY_2\uLs}|$ in $L^\infty$ whenever one of them appears. Combining these results gives
 \begin{equation}
 \lVert \tp[s]\tm[1/2]C_k^2 w_\omu^{1/2}\rVert_2 \lesssim \mcE_k^{3/2}.
 \end{equation}
We focus on $C_k^3$, which we may reduce to the region $r^*>(t+1)/2$. We first write
\begin{equation}
D_\alpha D_{\wX}^{I_3}\bphi = \tfrac{1}{r^*}D_\alpha(r^*D_{\wX}^{I_3}\bphi) - \tfrac{\partial_\alpha(r^*)}{r^*}D_{\wX}^{I_3}\bphi.
\end{equation}
Therefore, we can write
\begin{equation}\label{id:PhiCommC}
C^3_k = \Big(\nabla_\alpha\left(\Lie{{\wX}}^{I_1}(g^{-1}i_{{\wY}_1}\bF)\right)^\alpha\ - 2\tfrac{\partial_\alpha(r^*)}{r^*}\left(\Lie{{\wX}}^{I_1}(g^{-1}i_{{\wY}_1}\bF)\right)^\alpha\Big)D_{\wX}^{I_2}\bphi + 2\left(\Lie{{\wX}}^{I_1}(g^{-1}i_{{\wY}_1}\bF)\right)^\alpha \Big(\tfrac{D_\alpha(r^*D_{\wX}^{I_2}\bphi)}{r^*}\Big).
\end{equation}
We first bound
\begin{equation}
\lnorm\tp[s]\tm[1/2]\left(\Lie{\wX}^{I_1}(g^{-1}i_{Y_1}\bF)\right)^\alpha \tfrac{D_\alpha\left(r^*D_{\wX}^{I_2}\bphi\right)}{r^*}w_\omu^{1/2}\rnorm_2.
\end{equation} 
When the derivative falls on $g^{-1}$ we expand using $\OLie{\wX}^{I}$. Terms with the modified deformation tensor $\TDTD{\wX^{I_1}}$ may be bounded in a similar fashion to \eqref{FirstTermSpecial}, so it remains to bound
\[
\lnorm\tp[s]\tm[1/2]g^{\alpha\beta}(\Lie{\wX}^{I_1}\bF)_{\wY_1\alpha} \tfrac{D_\beta(r^*D_{\wX}^{I_2}\bphi)}{r^*}w_\omu^{1/2}\rnorm_2
\]
We may replace $g^{\alpha\beta}$ with $\mhat^{\alpha\beta}$, noting that the norm corresponding to the difference may be bounded in the same way as \eqref{FirstTermSpecial}. Likewise, the terms containing $\bF^0$ are easily bounded. We bound the remainder by expanding $(\Lie{\wX}^{I_1}\bF^1)_{\wY_1\alpha}$ in its null decomposition, and using the $L^2(L^\infty)$ bounds \eqref{est:FTLinfty} and \eqref{est:LinftyDPhi} to bound $\alpha$ and $D_\uLs\psi$ terms whenever applicable, bounding the remainder using $E_k$ energy norms. For all other components, we use the $L^\infty(L^\infty)$ bounds \eqref{est:FTLinfty}, \eqref{est:LinftyDPhi}, and \eqref{ChargeLInfty}, and pair them with $S_k$ energy norms. Consequently,
\begin{equation}
\lnorm\tp[s]\tm[1/2]\left(\Lie{\wX}^{I_1}(g^{-1}i_{Y_1}\bF)\right)^\alpha \tfrac{D_\alpha\left(r^*D_{\wX}^{I_2}\bphi\right)}{r^*}w_\omu^{1/2}\rnorm_2 \lesssim \mathcal{E}_k(T).
\end{equation}

We now seek to bound
\begin{equation}
\Big\lVert \tp[s]\tm[1/2]\Big(\nabla_\alpha\left(\Lie{{\wX}}^{I_1}(g^{-1}i_{{\wY}_1}\bF)\right)^\alpha\ - 2\tfrac{\partial_\alpha(r^*)}{r^*}\left(\Lie{{\wX}}^{I_1}(g^{-1}i_{{\wY}_1}\bF)\right)^\alpha\Big)D_{\wX}^{I_2}\bphi  w_\omu^{1/2}\Big\rVert_2
\end{equation}
Again, when the Lie derivative falls on the metric, we decompose using the reduced Lie derivative $\OLie{}$. A null decomposition similar to the one used to bound \eqref{DTDTD}, combined with the inequality $r^{*-1}\lesssim \tp[-1]$ in our region of concern, gives
\begin{equation}
\slk\Big\lVert \tp[s]\tm[1/2]\nabla_\alpha\Big(\TDTD{\wX^{I_1}}^{\alpha\beta}(\Lie{{\wX}}^{I_2}\bF)_\beta\Big) - 2\tfrac{\partial_\alpha(r^*)}{r^*}(\TDTD{\wX^{I_1}}^{\alpha\beta})(\Lie{{\wX}}^{I_2}\bF)_\beta D_{\wX}^{I_2}\bphi  w_\omu^{1/2}\Big\rVert_2 \lesssim \mathcal{E}(T)
\end{equation}
To deal with the remainder, it suffices to bound
\begin{equation}\label{C3Mid}
\slk\lnorm \tp[s]\tm[1/2]g^{\alpha\beta}\left(\nabla_\alpha(\Lie{\wX}^{I_1}\bF)_{\wY_1\beta}) - \tfrac{2\partial_\alpha(r^*)}{r^*}(\Lie{\wX}^{I_1}\bF)_{\wY_1\beta})\right)D_{\wX}^{I_2}\bphi w_\omu^{1/2}\rnorm_2
\end{equation}
We decompose
\begin{equation}
\nabla^\alpha(\Lie{\wX}^{I_1}\bF)_{\wY_1\alpha} = \bJ[\Lie{\wX}^{I_1}\bF]_{\wY_1}D_{\wX}^{I_2}\bphi+ (\nabla^\alpha \wY_1^\beta)(\Lie{\wX}^{I_1}\bF)_{\beta\alpha}D_{\wX}^{I_2}\bphi.
\end{equation}
If at least one Lie derivative falls on $\bF$ in the first term, we may bound
\begin{equation}
\slkI{|I_1|\leq k-2}\lnorm \tp[s]\tm[1/2]J[\Lie{\wX}^{I_1}\bF]_{\wY_1}D_{\wX}^{I_2}\bphi w_\omu^{1/2}\rnorm_2 \lesssim \mathcal{E}_k^{1/2}\slk\lnorm\tp[s-1]\tm[1-s]\bJ[\Lie{\wX}^{I_1}\bF]_{\wY_1} w_\omu^{1/2}\rnorm_2.
\end{equation}
Decomposing $Y_1$ in terms of null vectors and using the relation $1-s < 1/2$ allows us to bound this quantity by
\[
\lnorm \bJ[\Lie{\wX}^{I_1}\bF]\rnorm_{L^2[w]}.
\]
In particular, we commute the Lie derivative through, bound the estimate as in \eqref{LieCurrent}, and take the null decomposition. If no Lie derivatives fall on $\bF$, we bound the corresponding norm straightforwardly using the $L^\infty$ norm
\[
|\bJ_{\wY_1}| \lesssim \mathcal{E}_k\tp[-1-s]\tm[-s]
\]
Finally, we consider
\[
\left((\nabla^\alpha \wY_1^\beta) - \tfrac{2\nabla^\alpha(r^*)}{r^*}\wY_1^\beta\right)(\Lie{X_1}^{I_1}\bF)_{\beta\alpha}D_{X_2}^{I_2}\bphi.
\]
We lower the indices and consider the null decomposition. To bound the corresponding term in \eqref{C3Mid} it suffices to show the bound
\[
\Big((\nabla_\alpha Y_{1\beta}) - \tfrac{2\nabla_\alpha(r^*)}{r^*}Y_{1\beta}\Big)\wt{U}_1^\alpha \wt{U}_2^\beta \lesssim 1,
\]
for $\wt{U}_1, \wt{U}_2 \in \mcU$ and apply our usual $L^2(L^2)$ and $L^\infty(L^\infty)$ estimates. However, to bound $\bF^{\uLs\Bs{j}}$ we instead need
\begin{equation}\label{NablaYNice}
\Big(\nabla_\alpha \wY_{1\beta}-\nabla_\beta \wY_{1\alpha} - \tfrac{2\nabla_\alpha(r^*)}{r^*}\wY_{1\beta}\Big)\Ls^\alpha{\Bs{i}}{}^\beta\lesssim \tz[1-\delta].
\end{equation}
We may replace $\nabla_\beta, \nabla_\alpha$ with $\partial_\beta, \partial_\alpha$, as a null decomposition combined with the bounds \eqref{MetLI} implies
\[
|\Gamma_{\Ls\Bs{i}\wY_1}| + |\Gamma_{\Bs{i}\Ls \wY_1}|\lesssim \tp[-1]
\]
Expanding in the null decomposition and using the bound \eqref{MetLI} lets us bound all terms by the right hand side of \eqref{NablaYNice}. The remainder follows from direct calculation. Therefore,
\begin{equation}
\left(\nabla_\alpha \wY_{1\beta}-\nabla_\beta \wY_{1\alpha}- \tfrac{2\partial_\alpha(r^*)}{r^*}\wY_{1\beta}\right)\bF^{\alpha\beta}  \lesssim \tp[-1+\delta]|\ual| + |\rho| + |\sigma| + |\alpha|.
\end{equation}
Our estimate follows.
We can combine everything:
\begin{Theorem}
If $(\bF, \bphi)$ solves \eqref{MKG}, and
\[
\mathcal{E}_k(T)\leq 1,
\]
 then
\begin{equation}
\label{PhiCommFull}
\slk\left\lVert\left(\Box^{\mathbb{C}}_g\Lie{\wX}^I\bphi\right) \tp[s]\tm[1/2]w_\omu^{1/2}\right\rVert_2 \lesssim (\mathcal{E}_k^{1/2}+\varepsilon_g)\mathcal{E}_k^{1/2}
\end{equation}
\end{Theorem}

\section{Appendix: Inequalities}
	We start by stating Kato's diamagnetic inequality, which will be useful in the estimates to follow.
Given a complex scalar field $\phi$ and a vector field $Z$, we have the inequality
\begin{equation}\label{est:Kato}
\left|Z(|\phi|)\right| \leq |D_Z\phi|.
\end{equation}
This follows from the Cauchy-Schwarz inequality applied to the identity
\[
Z(|\phi|) = \mathfrak{R}\left(\langle\phi/|\phi|, D_Z\phi\rangle\right)
\]
Consequently, in the Sobolev-type estimates to follow, we can replace all cases of $Z(\phi)$ with $D_Z(\phi)$.
\begin{Lemma}
For any $q > 2$, and for any function $\phi$ with sufficient regularity, we have the following inequality on the sphere $\mathbb{S}^2_r$ of radius $r$, as long as $r > t/2$, $r > 1/2$:
\begin{equation}
\label{IntSob1}
\sup_{\mathbb{S}^2_r}|\chi\phi| \lesssim_q \tp[-2/q]\left(\sum_{|I| \leq, Z \in \mathbb{O}}\lnorm Z(\chi\phi) \rnorm_{L^q(S_r^2)}\right)
\end{equation}
\end{Lemma}
\begin{proof}
This is a straightforward consequence of Morrey's inequality applied to two charts on the unit sphere, and scaling to the sphere of radius $r$ (and introducing a factor of $r^{-2/q}$). The presence of the cutoff $\chi$ allows us to use the estimate $r^{-2/q}\lesssim\tp[-2/q]$.
\end{proof}
\begin{Lemma}
For $2 \leq q < 4$, if $r > t/4$ and $t >1$, we have
\begin{equation}
\label{IntSob2}
\lnorm \chi \phi \rnorm_{L^\infty(r^*)L^q({\mathbb{S}^2_r})} \lesssim_q \tp[-1+2/q]\tm[-1/2]\left(\lnorm\tm \drs(\chi\phi)\rnorm_{L^2(x)} + \sum_{|I|\leq1, Z \in \mathbb{O}}\lnorm Z(\chi\phi)\rnorm_{L^2(x)}\right).
\end{equation}
\end{Lemma}
\begin{proof}
This follows from the Sobolev estimate on a cylinder, rescaled to a dyadic region. We first define the dyadic decomposition $\{\mathcal{I}_i^{\pm}\}$ for a given time slice $\Sigma_t$ as follows:
\begin{equation}
\mathcal{U}_i = \left\{ x : r^* > t/2, 2^i \leq |\us| + 1 \leq 2^{i+1}\right\}.
\end{equation}
We subdivide these as follows:
\begin{equation}
\mathcal{U}^+_i = \left\{ x : r^* > t/2, 2^i \leq |\us| +1 \leq 2^{i+1}, \us > 0\right\}, \qquad \mathcal{U}^-_i = \left\{ x : r^* > t/2, 2^i \leq |\us| + 1 \leq 2^{i+1}, \us < 0\right\}.
\end{equation}
Thus, $\mathcal{U}^+$ are supported in the interior, and $\mathcal{U}^-$ are supported in the exterior. Additionally, for any given time slice, $\mathcal{U}_i^+$ is empty for sufficiently large $i$. We can construct a partition of unity $\{\chi_{\mathcal{U}_i^\pm}\}$ such that the support of each is in the region $\{ x: 2^{i-1} \leq |\us| \leq 2^{i+2}, r^* > t/4\}$ and derivatives satisfy the bound $\drs (\chi_{\mathcal{U}_i^\pm}) \lesssim 2^{-i}$ for some constant independent of $i$.

We now define the cylindrical region
\begin{equation}
(\widetilde{r}, \omega) \in \mathcal{A} = [1/4, 4] \times \mathbb{S}^2. 
\end{equation}
We take maps from our cylinder to the region $\mathcal{U}_i^\pm$ as follows:
\begin{equation}
(\widetilde{r}, \omega) \to (t, t\pm 2^i\widetilde{r}\omega),
\end{equation}
with an appropriate cutoff (recall that $t$ is fixed). Here we scale the radial variable by approximately $\tm$ and the spherical variables by $\tp$. Then we take the fractional Sobolev estimates on the region $\mathcal{A}$
\begin{subequations}
\begin{align}
\lnorm\chi\phi\rnorm_{L^\infty(\mathbb{R})} &\lesssim \lnorm \chi\phi\rnorm_{H^{1/2 + 2\epsilon_1}},\\
\lnorm\chi\phi\rnorm_{L^q(\mathbb{R}^2)}&\lesssim \lnorm \chi\phi\rnorm_{H^{1 - 2/q + 2\epsilon_2}},
\end{align}
\end{subequations}
which hold for all $\epsilon_i > 0, 2 \leq q < 4$. Since the inequality
\begin{equation}
(1+|\xi_x|^2)^{1/4+\epsilon_1}(1+|\xi_y|^2)^{1/2-1/q+\epsilon_2} \lesssim (1+|\xi_x|^2 + |\xi_y|^2)^{1/2}
\end{equation}
holds in the phase space for sufficiently small $\epsilon_i$ (depending on $q$), taking charts gives us the inclusion inequality
\begin{equation}
\lnorm\chi\phi\rnorm_{L^\infty(r^*)L^q(\mathbb{S}^2)} \lesssim \lnorm\chi\phi\rnorm_{H^1(\mathcal{A})}.
\end{equation}
We can take our change of variables, noting scaling, to get the estimate \eqref{IntSob2}.
\end{proof}

This covers our estimates for the extended exterior. We now look at the far interior.
\begin{Lemma}
If $t \geq 1$, $r < 3/4t$, we have the following estimates for compactly supported functions $f$:
\begin{subequations}
\begin{align}
\label{IntSob3}
\lnorm f\rnorm_{L^\infty(\mathbb{R}^3)} &\lesssim t^{-1/2}\sum_{\substack{X \in \{\Ss, \Os{0i}\} \\ |I| \leq 1}}\lnorm X^I f \rnorm_{L^6(\mathbb{R}^3)} \\
\label{IntSob4}
\lnorm f\rnorm_{L^6(\mathbb{R}^3)} &\lesssim t^{-1}\sum_{\substack{X \in \{\Ss, \Os{0i}\} \\ |I| \leq 1}}\lnorm X^I f \rnorm_{L^2(\mathbb{R}^3)}
\end{align}
\end{subequations}
\end{Lemma}
\begin{proof}
This follows from a rescaling to the unit ball, noting that 
\[
t\lnorm\nabla f\rnorm_{L^p} \lesssim \sum_{\substack{X \in \{\Ss, \Os{0i}\} \\ |I| \leq 1}}\lnorm X^I(f)\rnorm_{L^p}.
\]
This follows almost identically from the proof in \cite{LS}, noting that $|\partial\phi|$ and $|\wpa\phi|$ are equivalent.
\end{proof}
Finally, we consider the light cone estimate. As in \cite{LS}, this is not strictly necessary in closing our estimate, as we can get our full results using an $L^2(t)L^\infty(x)$ estimate following from the time slice Sobolev estimates. However, this estimate gives us more precise control over the asymptotic behavior:
\begin{Lemma}
For $2 \leq q < 4$, we have the global estimate
\begin{equation}
\label{IntSob5}
\lnorm \chi \phi \rnorm_{L^\infty(u^*)L^q(\mathbb{S}^2_r)} \lesssim \tp[-3/2-2/q]\sum_{\substack{X \in \{\uls\Ls, \mathbb{O}\} \\|I| \leq 1}}\lnorm X^I(\chi\phi)\rnorm_{L^2(C(\us))}.
\end{equation}
\end{Lemma}
\begin{proof}
This is similar to the proof of inequality \eqref{IntSob2}, with two differences. First, due to boundary considerations along the light cone, we need to take a Sobolev extension function across the endpoints of the time slab $t \in [1, T]$. Second, we take our dyadic decomposition in $\uls$ instead of $\us$. This introduces a factor of $\tp$ instead of $\tm$ in the analogue to the radial derivative $\partial_{\tilde{r}}$ in the cylinder. However, this is paired with $\Ls$, a nicer behaving directional derivative.
\end{proof}
We can now put everything together:
\begin{Theorem}
\label{Sobolev}
Given a smooth test function $\phi$, we have the following estimates:
\begin{align}
\label{TSSob}
\lnorm\tp[1+\delta_+]\tm[1/2+\delta_-]\chi\phi w^{1/2}\rnorm_{L^\infty(\mathbb{R}^3)} &\lesssim \sum_{\substack{|I|, |J| \leq 1 \\ X \in \{\tm\drs\} \cup \mathbb{O}, Y \in \mathbb{O}}}\lnorm \tp[\delta_+]\tm[\delta_-]X^I Y^J (\chi\phi)w^{1/2}\rnorm_{L^2(\mathbb{R}^3)}\\
\label{LCSob}
\lnorm\tp[3/2+\delta_+]\chi\phi w^{1/2}\rnorm_{L^\infty(C_{\us})} &\lesssim \sum_{\substack{|I|, |J| \leq 1\\ X \in\{\uls\Ls, \mathbb{O}\}, Y \in \mathbb{O}}} \lnorm \tp[\delta_+]X^I Y^J(\chi\phi)w^{1/2}\rnorm_{L^2(C_{\us})} \\
\label{ISob}
\lnorm\tp[3/2+\delta_+](1-\chi)\phi w^{1/2}\rnorm_{L^\infty(\mathbb{R}^3)} &\lesssim \sum_{\substack{|J| \leq 2 \\ Z \in \{\Ss, \Os{0i}\}}}\lnorm \tp[\delta_+]Z^I((1-\chi)\phi)w^{1/2}\rnorm_{L^2(\mathbb{R}^3)} 
\end{align}
as well as their complex covariant equivalents
\begin{align}
\label{TSDSob}
\lnorm\tp[1+\delta_+]\tm[1/2+\delta_-]\chi\phi\rnorm_{L^\infty(\mathbb{R}^3)} &\lesssim \sum_{\substack{|I|, |J| \leq 1 \\ X \in \{\tm\drs\} \cup \mathbb{O}, Y \in \mathbb{O}}}\lnorm\tp[\delta_+]\tm[\delta_-] D_X^I D_Y^J (\chi\phi)\rnorm_{L^2(\mathbb{R}^3)}, \\
\label{LCDSob}
\lnorm\tp[3/2+\delta_+]\chi\phi\rnorm_{L^\infty(C_{\us})} &\lesssim \sum_{\substack{|I|, |J| \leq 1\\ X \in\{\uls\Ls, \mathbb{O}\}, Y \in \mathbb{O}}} \lnorm \tp[\delta_+]D_X^I D_Y^J(\chi\phi)\rnorm_{L^2(C_{\us})}, \\
\label{IDSob}
\lnorm\tp[3/2+\delta_+](1-\chi)\phi\rnorm_{L^\infty(\mathbb{R}^3)} &\lesssim \sum_{\substack{|J| \leq 2 \\ Z \in \{\Ss, \Os{0i}\}}}\lnorm \tp[\delta_+]D_Z^I((1-\chi)\phi)\rnorm_{L^2(\mathbb{R}^3)}.
\end{align}
\end{Theorem}
\begin{proof}
This straightforwardly follows from \eqref{IntSob1}-\eqref{IntSob5}, with powers of $w$ and $\delta_\pm$ added during the dyadic decomposition.
\end{proof}

We first look at a model inequality in 1+3 dimensions. The proof of this is adapted from an intermediate result found in \cite{H}, and can be readily generalized to results which will be useful in our $L^2$ and $L^\infty$ estimates. We go through it in detail, 
\begin{Lemma}
\label{PHBasic}
For any function $\phi \in C_0^\infty$, we have the inequality
\begin{equation}
\intsig{t} (t+r)^2\left|\tfrac{\psi}{r}\right|^2\, dx \lesssim \intsig{t}(r-t)^2\left|\tfrac{\partial_r(r\psi)}{r}\right|^2 \, dx.
\end{equation}
\end{Lemma}
\begin{proof}
By transforming into spherical coordinates and noting that the integrating factor scales in $r$ like $r^2$, we can reduce this problem to showing the inequality
\begin{equation}
\label{OneDH2}
\int_0^\infty\left(\tfrac{t+r}{r}\right)^2(r\psi)^2 \, dr \lesssim \int_0^\infty (r-t)^2\partial_r(r\psi)^2 \, dr,
\end{equation}
where we have restricted $\psi$ along lines of constant $\omega$.
To show that this is true, we first take the one-dimensional inequality
\begin{equation}
\int_0^\infty\left(Cf\partial_r\Psi + g\Psi\right)^2 - \partial_r(Cfg\Psi^2)\, dr \geq 0,
\end{equation}
which holds as long as $fg\Psi^2$ is absolutely continuous and vanishes at 0 and at $\infty$. This is satisfied for $\Psi = r\psi$, where $\psi$ is compactly supported. We can think of $f$ and $g$ as weight functions, and $C$ is an arbitrary constant. We can rewrite this as
\begin{equation}
\label{fgphi2}
\int_0^\infty \left(C\partial_r(fg) - g^2\right)(r\psi)^2 \lesssim \int_0^\infty C^2f^2(\partial_r(r\psi))^2 \, dr.
\end{equation}
As an aside, we note that if $(C\partial_r(fg) - g^2) > \epsilon g^2$, for some $\epsilon$ depending on $f, g$, this is a meaningful inequality. 

We select
\begin{align*}
f &= r-t, \\
g &= \tfrac{r+t}{r}.
\end{align*}
Then,
\[
\partial_r(fg) = \tfrac{r^2+t^2}{r^2}.
\]
Selecting $C = 4$ we see that
\[
\left(C\partial_r(fg) - g^2\right) \geq \left(\tfrac{r+t}{r}\right)^2.
\]
The inequality \eqref{OneDH2} follows.
\end{proof}
\begin{Lemma}
\label{HardyEstimate}
For $\tfrac12 < s \leq 1$, and for compactly supported $f$, we have the estimate
\begin{align}
\intsig{t} \tp[2s] \left|\tfrac{\psi}{r}\right|^2 w \, dx & \lesssim \intsig{t} \tm[2s]\left|\tfrac{D_{r^*}(r^*\psi)}{r}\right|^2 w \, dx.
\end{align}
\end{Lemma}
\begin{proof}
Due to \eqref{est:Kato}, it suffices to prove this if $D_{r^*}$ is replaced with $\drs$. As in the previous lemma, we reduce to the one-dimensional inequality
\begin{equation}
\int_0^\infty \tp[2s]|\psi|^2 w \, dr \lesssim \int_0^\infty \tm[2s]|\drs(r^*\psi)|^2 w \, dr.
\end{equation}
Since $dr$ and $dr^*$ are equivalent, we can replace the former with the latter without issue. We now take inequality \eqref{fgphi2}, with $r^*$ in place of $r$, and
\begin{subequations}
\begin{align}
f &= |r^*-t|^s\text{sgn}(r^*-t)w^{1/2}, \\
g &= \tfrac{|r^*+t|^s}{r^*}w^{1/2}.
\end{align}
\end{subequations}
We have that in this case is equal to
\[
\partial_r^*(fg) = \text{sgn}(r^*-t)\tfrac{2r^{*2}s|r^{*2}-t^2|^{s-1}\text{sgn}(r^*-t) - |r^{*2}-t^2|^s}{r^{*2}}w + \tfrac{\drs(w)}{w}(fg).
\]
The last term is strictly positive, as $\drs(w)$ is supported when $r^* - t > 0$. We can rewrite
\[
\partial_r^*(fg) \geq \tfrac{((2s-1)r^{*2} + t^2)|r^{*2}-t^2|^{s-1}}{r^{*2}}w.
\]
Choose $C$ such that $(2s-1)C \geq 4$. Then, noting $s-1 \leq 0$,  it follows that
\[
C\partial_r^*(fg) \geq \tfrac{4|r^{*2}+t^2|^{s}}{r^{*2}}w.
\]
For $s \leq 1$, we have
\[
C\partial_r^*(fg) - g^2 \geq g^2.
\]
This gives us the preliminary estimate
\begin{equation}
\intsig{t}|r^*+t|^{2s}\left|\tfrac{\psi}{r}\right|^2 w \, dx \lesssim \intsig{t}|r^*-t|^{2s}\left|\tfrac{D_{r^*}(r^*\psi)}{r}\right|w\,dx.
\end{equation}
We can add a time-shifted estimate replacing $t$ with $t+1$ to get the full estimate.
\end{proof}

Similar reasoning gives us the inequality
\begin{equation}
\label{intHardy}
\intsig{t}\tp[2s]\tz[1+2\delta]\left|\tfrac{\psi}{r^*}\right|^2(w')\,dx \lesssim \intsig{t} \tm[2s]\tz[1+2\delta]\left|\tfrac{D_{r^*}(r^*\psi)}{r^*}\right|^2(w')\, dx,
\end{equation}
as long as we have the inequality $s + \delta < 1$.

We now prove an estimate along the same lines which is better suited to our conformal Morawetz estimate. This is an alternate proof to a similar result in \cite{LS}
\begin{Lemma}
\label{PoincareMain}
For $p, q$ such that $p > -1$, $|q| < p+1$, and for test functions $\phi$, we have the  inequality
\begin{equation}
\intsig{t}\tm[p]\tp[q]|\psi|^2 w\, dx \lesssim \intsig{t}\tm[p+2]\tp[q]\left|\tfrac{D_{r^*}(r^*\psi)}{r^*}\right|^2 w\,dx.
\end{equation}
\end{Lemma}
\begin{proof}
This follows from the one dimensional inequality
\begin{equation}
\int_0^\infty\tm[p]\tp[q]|r^*\psi|^2 w\, dr^* \lesssim \int_0^\infty \tm[p+2]\tp[q]|\drs(r^*\psi)|^2 w\, dr^*.
\end{equation}
Additionally, we can replace $\tm$ and $\tp$ with $1+|r^*-t|$ and $1+r^*+t$ respectively.
We again seek to apply \eqref{fgphi2}, with
\begin{align*}
f &= (1+|r^*-t|)^{p/2 + 1}\sgn(r^*-t)(1+r^*+t)^{q/2}w^{1/2} \\
g &= (1+|r^*-t|)^{p/2}(1+r^*+t)^{q/2}w^{1/2}
\end{align*}
Then, 
\[
\partial_{r^*}(fg) = \begin{cases}
\left(\tfrac{p+1+2\delta}{1+|r^*-t|} + \tfrac{q}{1+r^*+t}\right)fg & r^* > t, \\
\left(\tfrac{p+1}{1+|r^*-t|} - \tfrac{q}{1+r^*+t}\right)fg & r^* < t.
\end{cases}
\]
Since $g^2 = fg(1+|r^*-t|)^{-1}$ and $1+|r^*-t| < 1+r^*+t$, it follows that
\[
\partial_{r^*}(fg) \geq (p+1-|q|)g^2.
\]The result follows.
\end{proof}

\printbibliography
\end{document}